\documentclass[reqno]{amsart}
\usepackage{amscd}
\usepackage{amsrefs}
\usepackage{amssymb} 
\usepackage{amsthm}
\usepackage{bbm}
\usepackage{enumerate} 
\usepackage{esint}
\usepackage{graphicx}
\usepackage{color}
\usepackage{mathrsfs} 

\theoremstyle{plain} 
\newtheorem{theorem}{Theorem}[section]
\newtheorem{proposition}[theorem]{Proposition}
\newtheorem{lemma}[theorem]{Lemma} 

\theoremstyle{definition} 
\newtheorem{definition}[theorem]{Definition}

\theoremstyle{remark} 

\newtheorem{example}[theorem]{Example}

\newtheorem{remark}[theorem]{Remark}

    \newcommand{\IGNORE}[1]{}

   \newcommand{\C}{{\mathrm{C}}}
   \newcommand{\CC}{{\mathscr{C}}}
    \newcommand{\D}{{\mathscr{D}}}
    \renewcommand{\H}{{\mathscr{H}}}
    \newcommand{\J}{{\mathcal{J}}}
    \newcommand{\K}{{\mathscr{K}}}
    \renewcommand{\L}{{\mathscr{L}}}
    \newcommand{\M}{{\mathscr{M}}}
    \newcommand{\N}{\mathbb{N}} 
    \renewcommand{\P}{\mathsf{P}}
    \newcommand{\R}{\mathbb{R}}
    \newcommand{\T}{\mathscr{T}}

    \newcommand{\GS}{\geqslant}
    \newcommand{\ID}{{\mathrm{id}}}
    \newcommand{\LS}{\leqslant}
    \newcommand{\MM}{\mathbb{M}}
    \newcommand{\SP}{{\mathscr{P}}}

    \newcommand{\LEB}{{\mathcal{L}}}
     \newcommand{\calT}{{\mathcal{T}}}
    \newcommand{\LIP}{{\mathrm{Lip}}}
    
     \newcommand{\OP}{{\mathcal{O}}}
    \newcommand{\OPT}{{\mathrm{opt}}}
    \newcommand{\REG}{{\mathrm{reg}}}
    \newcommand{\RHO}{\varrho}
    \newcommand{\SUB}[1]{
        _{\raisebox{1ex}{\scriptsize{$#1$}}}}
    \newcommand{\VAN}{{\mathcal{N}}}

    \newcommand{\CHAR}{{\mathbbm{1}}}
    
    \newcommand{\PROJ}[1]{\sfP\kern-1pt_{#1}}
    \newcommand{\RRHO}{{\boldsymbol{\varrho}}}
    \newcommand{\WEAK}{
        \DOTSB\protect\relbar\protect\joinrel\rightharpoonup}
    \newcommand{\restr}[1]{\lower3pt\hbox{$|_{#1}$}}

    \DeclareMathOperator{\SPT}{{\mathrm{spt}}}

    \numberwithin{equation}{section}

\usepackage[normalem]{ulem} 

\newcommand{\EEE}{\color{black}} 
\newcommand{\KK}{\mathbb{K}}
\newcommand{\KKN}{\mathbb{K}^N}

\newcommand{\w}{m}

\newcommand{\p}{2}
\newcommand{\leb}{\mathfrak m}
\title
    [Sticky particle dynamics with interactions]
    {Sticky particle dynamics with interactions}
\author
    {Y. Brenier, W. Gangbo, G. Savar\'{e}, and M. Westdickenberg} 
    \address
    {{\em Yann Brenier} \newline
     Departement de Math\'ematiques,
     Universit\' de Nice,
    Parc Valrose, 06108 Nice,
     France}
\email
    {brenier@math.unice.fr}
\urladdr
    {http://people.math.gatech.edu/\textasciitilde gangbo}

\address
    {{\em Wilfrid Gangbo} \newline
     School of Mathematics,
     Georgia Institute of Technology,
     686~Cherry Street,
     Atlanta, GA~30332-0160,
     U.S.A.}
\email
    {gangbo@math.gatech.edu}
\urladdr
    {http://people.math.gatech.edu/\textasciitilde gangbo}

\address
    {{\em Giuseppe Savar\'{e}} \newline
     Dipartimento di Matematica, 
     Universit\`{a} di Pavia. 
     Via Ferrata, 
     1--27100 Pavia, 
     Italy.}
\email
    {giuseppe.savare@unipv.it}
\urladdr
    {http://www.imati.cnr.it/\textasciitilde savare}

\address
    {{\em Michael Westdickenberg} \newline
     School of Mathematics,
     Georgia Institute of Technology,
     686 Cherry Street,
     Atlanta, GA 30332-0160,
     U.S.A.}
\email
    {mwest@math.gatech.edu}
\urladdr
    {http://people.math.gatech.edu/\textasciitilde mwest}

\date{\today}



\begin{document}

\begin{abstract}
We consider compressible pressureless fluid flows in Lagrangian coordinates in
one space dimension. We assume that the fluid self-interacts through a force
field generated by the fluid itself. We explain how this flow can be described
by a differential inclusion on the space of transport maps, in particular when a
sticky particle dynamics is assumed. We study a discrete particle approximation
and we prove global existence and stability results for solutions of this
system. In the particular case of the Euler-Poisson system 
in the attractive regime our approach yields an explicit representation
formula for the solutions.
\end{abstract}

    \maketitle \tableofcontents

%


\newcommand{\cA}{{\ensuremath{\mathcal A}}}
\newcommand{\cB}{{\ensuremath{\mathcal B}}}
\newcommand{\calC}{{\ensuremath{\mathcal C}}}
\newcommand{\cD}{{\ensuremath{\mathcal D}}}
\newcommand{\cE}{{\ensuremath{\mathcal E}}}
\newcommand{\cF}{{\ensuremath{\mathcal F}}}
\newcommand{\cG}{{\ensuremath{\mathcal G}}}
\newcommand{\cH}{{\ensuremath{\mathcal H}}}
\newcommand{\cK}{{\ensuremath{\mathcal K}}}
\newcommand{\cJ}{{\ensuremath{\mathcal J}}}
\newcommand{\cL}{{\ensuremath{\mathcal L}}}
\newcommand{\cM}{{\ensuremath{\mathcal M}}}
\newcommand{\cN}{{\ensuremath{\mathcal N}}}
\newcommand{\cO}{{\ensuremath{\mathcal O}}}
\newcommand{\cP}{{\ensuremath{\mathcal P}}}
\newcommand{\cT}{{\ensuremath{\mathcal T}}}
\newcommand{\cX}{{\ensuremath{\mathcal X}}}
\newcommand{\cY}{{\ensuremath{\mathcal Y}}}
\newcommand{\cV}{{\ensuremath{\mathcal V}}}
\newcommand{\cW}{{\ensuremath{\mathcal W}}}
\renewcommand{\aa}{{\mbox{\boldmath$a$}}}
\newcommand{\bb}{{\mbox{\boldmath$b$}}}
\newcommand{\cc}{{\mbox{\boldmath$c$}}}
\newcommand{\dd}{{\mbox{\boldmath$d$}}}
\newcommand{\ee}{{\mbox{\boldmath$e$}}}
\newcommand{\ff}{{\mbox{\boldmath$f$}}}
\renewcommand{\gg}{{\mbox{\boldmath$g$}}}
\newcommand{\hh}{{\mbox{\boldmath$h$}}}
\newcommand{\ii}{{\mbox{\boldmath$i$}}}
\newcommand{\jj}{{\mbox{\boldmath$j$}}}
\newcommand{\mm}{{\mbox{\boldmath$m$}}}
\newcommand{\nn}{{\mbox{\boldmath$n$}}}
\newcommand{\oo}{{\mbox{\boldmath$o$}}}
\newcommand{\pp}{{\mbox{\boldmath$p$}}}
\newcommand{\rr}{{\mbox{\boldmath$r$}}}
\renewcommand{\ss}{{\mbox{\boldmath$s$}}}
\renewcommand{\tt}{{\mbox{\boldmath$t$}}}
\newcommand{\uu}{{\mbox{\boldmath$u$}}}
\newcommand{\vv}{{\mbox{\boldmath$v$}}}
\newcommand{\ww}{{\mbox{\boldmath$w$}}}
\newcommand{\xx}{{\mbox{\boldmath$x$}}}
\newcommand{\yy}{{\mbox{\boldmath$y$}}}
\newcommand{\zz}{{\mbox{\boldmath$z$}}}
\newcommand{\saa}{{\mbox{\scriptsize\boldmath$a$}}}
\newcommand{\sbb}{{\mbox{\scriptsize\boldmath$b$}}}
\newcommand{\scc}{{\mbox{\scriptsize\boldmath$c$}}}
\newcommand{\sdd}{{\mbox{\scriptsize\boldmath$d$}}}
\renewcommand{\see}{{\mbox{\scriptsize\boldmath$e$}}}
\newcommand{\sff}{{\mbox{\scriptsize\boldmath$f$}}}
\newcommand{\sgg}{{\mbox{\scriptsize\boldmath$g$}}}
\newcommand{\shh}{{\mbox{\scriptsize\boldmath$h$}}}
\newcommand{\sii}{{\mbox{\scriptsize\boldmath$i$}}}
\newcommand{\sjj}{{\mbox{\scriptsize\boldmath$j$}}}
\newcommand{\smm}{{\mbox{\scriptsize\boldmath$m$}}}
\newcommand{\snn}{{\mbox{\scriptsize\boldmath$n$}}}
\newcommand{\soo}{{\mbox{\scriptsize\boldmath$o$}}}
\newcommand{\spp}{{\mbox{\scriptsize\boldmath$p$}}}
\newcommand{\sqq}{{\mbox{\scriptsize\boldmath$q$}}}
\newcommand{\srr}{{\mbox{\scriptsize\boldmath$r$}}}
\newcommand{\sss}{{\mbox{\scriptsize\boldmath$s$}}}
\newcommand{\stt}{{\mbox{\scriptsize\boldmath$t$}}}
\newcommand{\suu}{{\mbox{\scriptsize\boldmath$u$}}}
\newcommand{\svv}{{\mbox{\scriptsize\boldmath$v$}}}
\newcommand{\sww}{{\mbox{\scriptsize\boldmath$w$}}}
\newcommand{\sxx}{{\mbox{\scriptsize\boldmath$x$}}}
\newcommand{\syy}{{\mbox{\scriptsize\boldmath$y$}}}
\newcommand{\szz}{{\mbox{\scriptsize\boldmath$z$}}}
\newcommand{\aA}{{\mbox{\boldmath$A$}}}
\newcommand{\bB}{{\mbox{\boldmath$B$}}}
\newcommand{\cC}{{\mbox{\boldmath$C$}}}
\newcommand{\dD}{{\mbox{\boldmath$D$}}}
\newcommand{\eE}{{\mbox{\boldmath$E$}}}
\newcommand{\fF}{{\mbox{\boldmath$F$}}}
\newcommand{\gG}{{\mbox{\boldmath$G$}}}
\newcommand{\hH}{{\mbox{\boldmath$H$}}}
\newcommand{\iI}{{\mbox{\boldmath$I$}}}
\newcommand{\jJ}{{\mbox{\boldmath$J$}}}
\newcommand{\mM}{{\mbox{\boldmath$M$}}}
\newcommand{\nN}{{\mbox{\boldmath$N$}}}
\newcommand{\oO}{{\mbox{\boldmath$O$}}}
\newcommand{\pP}{{\mbox{\boldmath$P$}}}
\newcommand{\qQ}{{\mbox{\boldmath$Q$}}}
\newcommand{\rR}{{\mbox{\boldmath$R$}}}
\newcommand{\sS}{{\mbox{\boldmath$S$}}}
\newcommand{\tT}{{\mbox{\boldmath$T$}}}
\newcommand{\uU}{{\mbox{\boldmath$U$}}}
\newcommand{\vV}{{\mbox{\boldmath$V$}}}
\newcommand{\wW}{{\mbox{\boldmath$W$}}}
\newcommand{\xX}{{\mbox{\boldmath$X$}}}
\newcommand{\yY}{{\mbox{\boldmath$Y$}}}
\newcommand{\saA}{{\mbox{\scriptsize\boldmath$A$}}}
\newcommand{\sbB}{{\mbox{\scriptsize\boldmath$B$}}}
\newcommand{\scC}{{\mbox{\scriptsize\boldmath$C$}}}
\newcommand{\sdD}{{\mbox{\scriptsize\boldmath$D$}}}
\newcommand{\seE}{{\mbox{\scriptsize\boldmath$E$}}}
\newcommand{\sfF}{{\mbox{\scriptsize\boldmath$F$}}}
\newcommand{\sgG}{{\mbox{\scriptsize\boldmath$G$}}}
\newcommand{\shH}{{\mbox{\scriptsize\boldmath$H$}}}
\newcommand{\siI}{{\mbox{\scriptsize\boldmath$I$}}}
\newcommand{\sjJ}{{\mbox{\scriptsize\boldmath$J$}}}
\newcommand{\smM}{{\mbox{\scriptsize\boldmath$M$}}}
\newcommand{\snN}{{\mbox{\scriptsize\boldmath$N$}}}
\newcommand{\soO}{{\mbox{\scriptsize\boldmath$O$}}}
\newcommand{\spP}{{\mbox{\scriptsize\boldmath$P$}}}
\newcommand{\sqQ}{{\mbox{\scriptsize\boldmath$Q$}}}
\newcommand{\srR}{{\mbox{\scriptsize\boldmath$R$}}}
\newcommand{\ssS}{{\mbox{\scriptsize\boldmath$S$}}}
\newcommand{\stT}{{\mbox{\scriptsize\boldmath$T$}}}
\newcommand{\suU}{{\mbox{\scriptsize\boldmath$U$}}}
\newcommand{\svV}{{\mbox{\scriptsize\boldmath$V$}}}
\newcommand{\swW}{{\mbox{\scriptsize\boldmath$W$}}}
\newcommand{\sxX}{{\mbox{\scriptsize\boldmath$X$}}}
\newcommand{\syY}{{\mbox{\scriptsize\boldmath$Y$}}}
\newcommand{\aalpha}{{\mbox{\boldmath$\alpha$}}}
\newcommand{\bbeta}{{\mbox{\boldmath$\beta$}}}
\newcommand{\ggamma}{{\mbox{\boldmath$\gamma$}}}
\newcommand{\ddelta}{{\mbox{\boldmath$\delta$}}}
\newcommand{\eeta}{{\mbox{\boldmath$\eta$}}}
\newcommand{\llambda}{{\mbox{\boldmath$\lambda$}}}
\newcommand{\mmu}{{\mbox{\boldmath$\mu$}}}
\newcommand{\nnu}{{\mbox{\boldmath$\nu$}}}
\newcommand{\ppi}{{\mbox{\boldmath$\pi$}}}
\newcommand{\rrho}{{\mbox{\boldmath$\rho$}}}
\newcommand{\ssigma}{{\mbox{\boldmath$\sigma$}}}
\newcommand{\ttau}{{\mbox{\boldmath$\tau$}}}
\newcommand{\pphi}{{\mbox{\boldmath$\phi$}}}
\newcommand{\ppsi}{{\mbox{\boldmath$\psi$}}}
\newcommand{\xxi}{{\mbox{\boldmath$ \xi$}}}
\newcommand{\zzeta}{{\mbox{\boldmath$ \zeta$}}}
\newcommand{\saalpha}{{\mbox{\scriptsize\boldmath$\alpha$}}}
\newcommand{\sbbeta}{{\mbox{\scriptsize\boldmath$\beta$}}}
\newcommand{\sggamma}{{\mbox{\scriptsize\boldmath$\gamma$}}}
\newcommand{\sllambda}{{\mbox{\scriptsize\boldmath$\lambda$}}}
\newcommand{\smmu}{{\mbox{\scriptsize\boldmath$\mu$}}}
\newcommand{\snnu}{{\mbox{\scriptsize\boldmath$\nu$}}}
\newcommand{\sttau}{{\mbox{\scriptsize\boldmath$\tau$}}}
\newcommand{\seeta}{{\mbox{\scriptsize\boldmath$\eta$}}}
\newcommand{\sssigma}{{\mbox{\scriptsize\boldmath$\sigma$}}}
\newcommand{\srrho}{{\mbox{\scriptsize\boldmath$\rho$}}}
\newcommand{\ssttau}{{\mbox{\boldmath$\scriptscriptstyle\tau$}}}
\newcommand{\sseeta}{{\mbox{\boldmath$\scriptscriptstyle\eta$}}}
\newcommand{\ssxx}{{\mbox{\boldmath$\scriptscriptstyle x$}}}

\newcommand{\sfa}{{\sf a}}
\newcommand{\sfb}{{\sf b}}
\newcommand{\sfc}{{\sf c}}
\newcommand{\sfd}{{\sf d}}
\newcommand{\sfe}{{\sf e}}
\newcommand{\sfg}{{\sf g}}
\newcommand{\sfh}{{\sf h}}
\newcommand{\sfi}{{\sf i}}
\newcommand{\sfm}{{\sf m}}
\newcommand{\sfn}{{\sf n}}
\newcommand{\sfp}{{\sf p}}
\newcommand{\sfq}{{\sf q}}
\newcommand{\sfr}{{\sf r}}
\newcommand{\sfs}{{\sf s}}
\newcommand{\sft}{{\sf t}}
\newcommand{\sfu}{{\sf u}}
\newcommand{\sfv}{{\sf v}}
\newcommand{\sfx}{{\sf x}}
\newcommand{\sfy}{{\sf y}}
\newcommand{\sfK}{{\sf K}}
\newcommand{\sfC}{{\sf C}}
\newcommand{\sfD}{{\sf D}}
\newcommand{\sfP}{{\sf P}}
\newcommand{\sfM}{{\sf M}}
\newcommand{\sfS}{{\sf S}}
\newcommand{\sfT}{{\sf T}}
\newcommand{\sfU}{{\sf U}}
\newcommand{\sfV}{{\sf V}}
\newcommand{\sfX}{{\sf X}}
\newcommand{\sfY}{{\sf Y}}
\newcommand{\sfZ}{{\sf Z}}
\newcommand{\fra}{{\frak a}}
\newcommand{\frb}{{\frak b}}
\newcommand{\frc}{{\frak c}}
\newcommand{\frd}{{\frak d}}
\newcommand{\fre}{{\frak e}}
\newcommand{\frg}{{\frak g}}
\newcommand{\frh}{{\frak h}}
\newcommand{\frl}{{\frak l}}
\newcommand{\frm}{{\frak m}}
\newcommand{\frv}{{\frak v}}
\newcommand{\frV}{{\frak V}}
\newcommand{\frU}{{\frak U}}
\newcommand{\frx}{{\frak x}}
\newcommand{\fry}{{\frak y}}
\newcommand{\frK}{{\frak K}}
\newcommand{\frC}{{\frak C}}
\newcommand{\frD}{{\frak D}}
\newcommand{\frP}{{\frak P}}
\newcommand{\frM}{{\frak M}}
\newcommand{\Kliminf}{K\kern-3pt-\kern-2pt\mathop{\rm lim\,inf}\limits}  
\newcommand{\supp}{\mathop{\rm supp}\nolimits}   
\newcommand{\conv}{\mathop{\rm conv}\nolimits}   
\newcommand{\diam}{\mathop{\rm diam}\nolimits}   
\newcommand{\argmin}{\mathop{\rm argmin}\limits}   
\newcommand{\Lip}{\mathop{\rm Lip}\nolimits}          
\newcommand{\interior}{\mathop{\rm int}\nolimits}   
\newcommand{\aff}{\mathop{\rm aff}\nolimits}   
\newcommand{\tr}{\mathop{\rm tr}\nolimits}     
\newcommand{\signo}{\mathop{\rm sign}\nolimits_0}
\newcommand{\sign}{\mathop{\rm sign}\nolimits}
\newcommand{\card}{\mathop{\rm card}\nolimits}
\newcommand{\esssup}{\mathop{\rm ess\text{-}sup}\nolimits}
\renewcommand{\d}{{\mathrm d}}
\newcommand{\dt}{{\d t}}
\newcommand{\rmI}{{\mathrm I}}
\newcommand{\ddt}{{\frac \d\dt}}
\newcommand{\dx}{{\d x}}
\newcommand{\rmD}{{\mathrm D}}
\newcommand{\pd}[1]{{\partial_{#1}}}
\newcommand{\pdt}{{\pd t}}
\newcommand{\pdx}{{\pd x}}
\newcommand{\topref}[2]{\stackrel{\eqref{#1}}#2}
\newcommand{\Haus}[1]{{\mathscr H}^{#1}}     
\newcommand{\Leb}[1]{{\mathscr L}^{#1}}      
\newcommand{\la}{{\langle}}                  
\newcommand{\ra}{{\rangle}}
\newcommand{\down}{\downarrow}              
\newcommand{\up}{\uparrow}
\newcommand{\eps}{\varepsilon}  
\newcommand{\nchi}{{\raise.3ex\hbox{$\chi$}}}
\newcommand{\weakto}{\rightharpoonup}
\newcommand{\Rd}{{\R^d}}
\newcommand{\Rn}{{\R^n}}
\newcommand{\Rm}{{\R^m}}
\newcommand{\RN}{{\R^N}}
\newcommand{\media}{\mkern12mu\hbox{\vrule height4pt           %
          depth-3.2pt                                 
          width5pt}\mkern-16.5mu\int\nolimits}        

\newcommand{\Dist}{D}
\newcommand{\Semi}{U}
\newcommand{\Char}{\mathbbm{1}}
\newcommand{\X}{X}
\newcommand{\Scalar}[2]{(#1|#2)}
\newcommand{\Rightdert}{{\frac{\d}{\d t}\kern-5pt}^+}
\newcommand{\Rightderw}{{\frac{\d}{\d w}\kern-5pt}^+}
\newcommand{\Leftderw}{{\frac{\d}{\d w}\kern-5pt}^-}
\newcommand{\Proj}[1]{\sfP\kern-1pt_{#1}}
\newcommand{\ProjH}[1]{\Proj{\cH_{#1}}}
\newcommand{\PositiveCone}{\cN}
\newcommand{\PC}[1]{\PositiveCone_{#1}}
\newcommand{\Lebref}{\lambda}
\newcommand{\Xspace}[1]{\cV_{#1}}
\newcommand{\DXspace}{\cV^{\rm discr}}
\newcommand{\Lspace}[1]{\cX_{#1}(0,1)}
\newcommand{\pseudonorm}[2]{[#2]_{#1}}
\newcommand{\pseudonormp}[2]{\pseudonorm{#1}{#2}^{#1}}
\newcommand{\En}[3]{E_{#1}(#2|#3)}



\section{Introduction}\label{S:I}

In this paper, we consider a simple model for one-dimensional compressible fluid
flows under the influence of a force field that is generated by the fluid
itself. It takes the form of a hyperbolic conservation law for the density
$\RHO$, which is a nonnegative measure in time and space and describes the
distribution of mass or electric charge, and the real-valued Eulerian velocity
field $v$. For suitable initial data $(\RHO,v)(t=0,\cdot) =: (\bar{\RHO},
\bar{v})$, the unknowns $(\RHO,v)$ satisfy 
%
\begin{equation}
    \left.\begin{aligned}
        \partial_t\RHO
            + \partial_x(\RHO v) &= 0
\\
        \partial_t(\RHO v)
            + \partial_x(\RHO v^2)
            &=   f[\RHO]
    \end{aligned}\right\}
    \quad\text{in $[0,\infty)\times\R$.}
\label{E:CONS2}
\end{equation}

The first equation in \eqref{E:CONS2}, called the continuity equation, describes
the local conservation of mass or electric charge. Without loss of generality,
we will assume in the following that the total mass/charge is equal to one
initially and that the quadratic moment is finite so that $\RHO(t,\cdot) \in
\SP_2(\R)$ for all $t\GS 0$, with $\SP(\R)$ the space of probability measures
with finite quadratic moment. The second equation in \eqref{E:CONS2} describes
the conservation of momentum. We will assume in the following that $v(t,\cdot) \in
\L^2(\R, \RHO(t,\cdot))$ for all $t\GS 0$ so the kinetic energy is finite.

The  continuous map $f\colon\SP_2(\R) \longrightarrow \M(\R)$ in \eqref{E:CONS2}
describes the force field, with $\M(\R)$ the space of all signed Borel measures
with finite total variation. The force depends on the distribution of mass or
electric charge and  we will assume that $f[\RHO]$ is absolutely continuous with
respect to $\RHO$. For further assumptions see Section~\ref{S:GE}. The typical (simplest) form of $f$ is
\begin{equation}
  \label{eq:1}
  f[\RHO]=-\RHO\, \partial_x q_\RHO
  \quad\text{with}\quad
  q_\RHO(x)=V(x)+\int_\R W(x-y)\,\d\RHO(y)
\end{equation}
for suitable $\C^1$ potential functions $V, W$ with (at most) linearly growing derivatives.

Another relevant example we have in mind is the Euler-Poisson system, for which 
\begin{equation}
  f[\RHO] = -\RHO \,\partial_x q_\RHO
  \quad\text{with $q_\RHO$ solution of}\quad
  -\partial_{xx}^2 q_\RHO = \lambda ( \RHO - \sigma).
\label{eq:2}
\end{equation} 
When $\RHO$ is absolutely continuous with respect to the one-dimensional
Lebesgue measure $\LEB^1$, then the function $q_\RHO$ admits a
representation similar to \eqref{eq:1}, with
\begin{equation}
  \label{eq:3}
  V(x):=-\frac \lambda2\int_\R |x-y|\,\d\sigma(y),\quad
  W(x):=\frac \lambda 2|x|.
\end{equation}
If $\rho$ is not absolutely continuous with respect to $\LEB^1$, then we have a
similar representation with a nondifferentiable $W$, so that $f[\rho]$ must be
defined by a suitable approximation process.

The Euler-Poisson equations in the repulsive regime (with $\lambda<0$ and
negative concave potential $W$) is a simple model for semiconductors. In this
case, $\RHO$ describes the electron or hole distribution and the scalar function
$q_\RHO$ represents the electric potential generated by the distribution of
charges in the material. Here $\sigma$ is the concentration of ionized
impurities. The Euler-Poisson equations in the attractive regime (with
$\lambda>0$ and positive convex potential $W$) is the one-dimensional version of
a cosmological model for the universe at an early stage, describing the
formation of galaxies. Now $q_\RHO$ represents the gravitational potential and
$\sigma=0$.


\subsection{Singular solutions and particle models.}

Since there is no pressure in \eqref{E:CONS2}, there is no mechanism that forces
the density $\RHO$ to be absolutely continuous with respect to the Lebesgue
measure. In fact, the system~\eqref{E:CONS2} admits solutions that are singular
measures. Assume that we are given initial data in the form of a finite linear
combination of Dirac measures:
\begin{equation}
  \bar{\RHO} = \sum_{i=1}^N \bar m_i \delta_{\bar{x}_i}
  \quad\text{and}\quad
  \bar{\RHO}\bar{v} = \sum_{i=1}^N \bar m_i \bar{v}_i \delta_{\bar{x}_i},
\label{E:DISCDATA}
\end{equation}
where $\bar{\xx}=(\bar{x}_1,\cdots, \bar x_N)\in\R^N$ are the initial locations
of $N$ particles denoted $P_1,\cdots, P_N$, with corresponding masses $\bar
\mm=(\bar m_1, \cdots,\bar m_N)$ and initial velocities $\bar\vv=(\bar
v_1,\cdots,\bar v_N)$. We require that $\bar m_i>0$ and $\sum_i \bar m_i = 1$ so
that $\bar{\RHO}\in\SP(\R)$. For all times $t\GS 0$, we can assume that the
positions $\xx(t)= (x_1(t),\cdots,x_N(t))$ are monotonically ordered, so that
they are unambiguously determined and attached to the particles. Then (at least
formally) there is a solution of \eqref{E:CONS2} in the form of a linear
combination of Dirac measures:
\begin{equation}
  \RHO(t,\cdot) = \sum_{i=1}^N m_i \delta_{x_i(t)}
  \quad\text{and}\quad
  (\RHO v)(t,\cdot) = \sum_{i=1}^N m_i v_i(t) \delta_{x_i(t)},
\label{E:PART}
\end{equation}
where the functions $(x_i,v_i)$ solve the system of ordinary differential
equations
\begin{equation}
  \dot{x}_i(t) = v_i(t),\quad
  \dot v_i(t)=a_{\bar \smm,i}(\xx(t))
  \quad\text{and}\quad
  (x_i,v_i)(t=0) = (\bar{x}_i,\bar{v}_i)\label{eq:4} 
\end{equation}
between particle collisions. Here $ a_{\smm,i}(\xx)\EEE$ is the value in the
point $x_i(t)$ of the Radon-Nikodym derivative of the force $f[\RHO]$ with
respect to the measure $\RHO$, so that
\begin{equation}
  \label{eq:57}
  f[\RHO]=\sum_{i=1}^N a_{\smm,i}(\xx)\,m_i\delta_{x_i}\quad
  \text{if}\quad\RHO=\sum_{i=1}^N m_i\,\delta_{x_i},
\end{equation}
which is well defined when all $N$ particles are distinct.

Upon collision of, say, two particles with masses $m_k$ and $m_{k+1}$ at some
time $t>0$, the velocities of each one of them are changed to
\begin{equation}
  v_k(t+) = v_{k+1}(t+) = \frac{m_k v_k(t-) + m_{k+1} v_{k+1}(t-)}{m_k+m_{k+1}},
\label{E:CONSMOM}
\end{equation}
so that the momentum is preserved during the collision. Since both particles
continue their journey with the same velocity, they may be considered as one
bigger particle with mass $m_k+m_{k+1}$. Collisions of more than two particles
can be handled in a similar fashion. We will refer to any solution of
\eqref{E:CONS2} in the form \eqref{E:PART} as a discrete particle solution and
we will say that it satisfies a \emph{global sticky condition} if particles
after collision are not allowed to split. In this case, after each collision,
one could relabel the particles so that the system \eqref{eq:4} still makes
sense (with $N$ reduced in each particle collision) and induces a global in time
evolution.

Let us denote by $\KKN$ the closed cone
\begin{equation}
  \label{eq:15}
  \KKN:=\big\{\xx\in \R^N:x_1\LS x_2\LS \cdots\LS x_N\big\},
\end{equation}
whose interior is $\interior{\KKN}=\big\{\xx\in \R^N:x_1 <x_2< \cdots<
x_N\big\}$. The construction of discrete particle solutions as outlined above
can be done rigorously whenever the functions $a_{\smm,i}:\interior{\KKN}\to
\R^N$ are uniformly continuous in each bounded set (so that they admit a
continuous extension to $\KKN$ still denoted by $a_{\smm,i}$) and satisfies the
compatibility condition
\begin{equation}
  \label{eq:9}
  a_{\smm,k}(\xx) = a_{\smm,k+1}(\xx)
  \quad\text{if $x_k=x_{k+1}$ for some $1\LS k<N$.}
\end{equation}
This is certainly the case when the potentials $V,W$ considered in \eqref{eq:1}
are of class $\C^1$. On the other hand, the case of the Euler-Poisson system is
much more subtle and presents different features in the attractive or the
repulsive case.


\subsubsection*{The Euler-Poisson case in the repulsive regime: splitting and collapsing of masses}

Let us consider the simplest situation of $N$ distinct particles with equal
initial velocities $\bar v$, in the repulsive regime $\lambda=-1$ with
$\sigma=0$. Let $M_0:=0$, $M_i:=\sum_{j=1}^i m_j$ for $i=1,\ldots,N$ and set
$A_i:=\frac 12(M_{i-1}+M_i-1)$. Then it is not difficult to check (see Example
\ref{ex:EP1}) that in the repulsive regime
\begin{equation}
  \label{eq:5}
  a_{\smm,i}(\xx)=A_i
  \quad \text{for all $i$ if $\xx\in \interior{\KKN}$,}
\end{equation}
and so there is no continuous extension satisfying \eqref{eq:9}. Starting from
distinct initial positions, particles follow (at least for a small time
interval) the free motion paths
\begin{equation}
  \label{eq:6}
  x_i(t)=\bar x_i+t\bar v+\frac 12 A_i t^2.
\end{equation}
Since $A_i\LS A_{i+1}$ for all $i$, there are no collisions. Taking the limit as
the initial positions of two or more particles coincide we obtain the same
representation for every $\xx\in \KKN$.
%
%
On the other hand, if two particles $P_k, P_{k+1}$ coincide at the time $t=0$,
i.e.\ $\bar x_k=\bar x_{k+1}=\bar x$ with the same initial velocity $\bar v$,
then the ``sticky'' solution $x_k(t)=x_{k+1}(t)=\bar x+t \bar
v+\frac14(A_k+A_{k+1}) t^2$ gives raise to an admissible solution to
\eqref{E:CONS2} which is different from the previous one.
%
%
One could also consider a solution where $P_k,P_{k+1}$ stick in a small initial
time interval $[0,s]$ and then evolve according to \eqref{eq:6}.

Considering a situation where the number $N$ of admissible particles grows to
infinity with a uniform initial mass distribution concentrating at the origin,
one can guess that a ``repulsive'' solution arising from a unit mass
concentrated at $\bar x$ should istantaneously diffuse, becoming absolutely
continuous with respect to the Lebesgue measure $\LEB^1$: the explicit
formula is
\begin{equation}
  \label{eq:8}
  \RHO(t,\cdot)=u(t,\cdot)\LEB^1
  \quad\text{with}\quad
  u(t,x)=\frac 1{2t^2} \nchi_{(\bar x+\bar v t -\frac 14 t^2,
    \bar x+\bar v t+\frac 14 t^2)}(x)
  \quad\text{for all $x\in\R$ and $t\GS 0$.}
\end{equation}
An even more complicated situation occurs e.g.\ if $\bar v_i=0$ for $i\neq
k,k+1$, but $\bar v_k>0>\bar v_{k+1}$ in such a way that a collision occurs
between $P_k$ and $P_{k+1}$ at some time $t=r$, after which the particles could
stick or wait for some time and then evolve as in the previous example.

It would be important to find a selection mechanism that gives raise to a stable
notion of solution and to obtain a continuous model by passing to the limit in
the number of particles. In this paper we study a criterium of the following
type : Assume that two particle $P_k,P_{k+1}$ collide at some time $r>0$ with
incoming velocities $v_k(r_-)\GS v_{k+1}(r_-)$. Then the particles will stick
together for all times $r<t<s$ provided that $s$ is small enough so that
\begin{equation}
  \label{eq:10}
  v_k(r_-)+\int_r^t a_{\smm,k}(\xx(\tau))\,\d\tau\GS v_k(t)=v_{k+1}(t)\GS
  v_{k+1}(r_-)+\int_r^t a_{\smm,k+1}(\xx(\tau))\,\d\tau
\end{equation}
for all $r<t<s$. Conversely, if \eqref{eq:10} becomes false for some time $s>r$,
then the particles may separate again. A rigorous formulation of condition
\eqref{eq:10} in the case of a simultaneous collision or separation of more than
two particles, or of a continuous distribution of masses, can be better
understood in the famework of differential inclusions in a Lagrangian setting,
which we will describe in Section \ref{S:DLS}. Before giving an idea of this
approach, let us brefly consider how \eqref{eq:10} greatly simplifies in the
attractive regime.


\subsubsection*{The attractive Euler-Poisson system and the sticky condition}

In the attractive case, we can simply invert the signs in \eqref{eq:5}. It turns
out, however, that the behaviour of the two-particles example considered in the
previous paragraph changes completely, since the limit when two particles
collapse exhibit a strong stability: after a collision, two or more particles
stick together and do not split ever again, giving raise to a \emph{global sticky
solution.}

This reflects the fact that the sticky condition in the attractive regime
implies \eqref{eq:10} for all $s>r$: the functions $a_{\smm,i}$ defined by the
negative of \eqref{eq:5} always satisfy $a_{\smm,k}(\xx)\GS a_{\smm,k+1}(\xx)$
and the incoming velocities of two particles $P_k,P_{k+1}$ colliding at some
time $r$ satisfies $v_k(r-)\GS v_{k+1}(r_-)$, so that any sticky evolution
corresponding to $x_k(t)=x_{k+1}(t)$ for $t\GS r$ will satisfy \eqref{eq:10}.

As we shall see, the differential description in the Lagrangian setting we will
adopt encodes \eqref{eq:10} and corresponds to a sticky condition whenever the
acceleration field is continuous (as in \eqref{eq:9}) or it is of attractive
type. In the repulsive case it will model a suitable relaxation mechanism
allowing for separation of particles after collision, still preserving the
stability of the evolution.


\subsection{Lagrangian description and differential inclusions}

In this paper, we will give an interpretation of system \eqref{E:CONS2} in the
framework of differential inclusions. As before, let us first consider the
simpler case of the dynamic of a finite number of particles. We can identify the
positions of a collection of particles $P_1,\cdots,P_N$ with a vector
$\xx=(x_1,\cdots,x_N)\in \R^N$: since we labeled the particles in a monotone
way, it is not admissible for particles to pass by one another, so the order of
the locations must be preserved and the vector $\xx$ is confined in the closed
convex cone $\KKN$ defined by \eqref{eq:15}. Denoting by $\vv
=(v_1,\cdots,v_N)\in \R^N$ the vector of the velocities of the particles, their
trajectories between collisions are mostly determined by a system of
differential equation
\begin{equation}
  \label{eq:16}
  \dot\xx(t)=\vv(t),\quad
  \dot\vv(t)=\aa_\smm(\xx(t)),
\end{equation}
where $ \aa_\smm(\xx):=(a_{\smm,1}(\xx),\cdots,a_{\smm,N}(\xx))$ is a vector
field defined for $\xx\in \KKN$ as in \eqref{eq:4}, which in the simplest case
is continuous. Whenever the vector $\xx(t)$ hits the boundary of the domain
\begin{equation}
  \label{eq:21}
  \partial\KKN:=\big\{\xx\in \KKN:\Omega_\sxx\neq \varnothing\big\},\quad
  \Omega_\sxx := \big\{ j \colon x_j=x_{j+1}, j=1\ldots N-1 \big\},
\end{equation}
however, an instantaneous force changes its velocity in such a way that it stays
inside of $\KKN$.

In order to find a mathematical model that describes this situation, we must first
identify the set of admissible velocities at each point $\xx\in\KKN$,
which is called the tangent cone of $\KKN$ at $\xx$. It is defined by
\begin{equation}
  \label{eq:17}
  T_\sxx\KKN 
    := \mathrm{cl} \Big\{
      \theta(\yy-\xx):\yy\in\KKN,\ \theta\GS0 \Big\}.
\end{equation}
In our situation, it is not difficult to check that 
\begin{equation}
    T_\sxx\KKN = \Big\{ \vv\in\R^N \colon 
    \text{$v_j\LS v_{j+1} $ for all $j\in\Omega_\sxx$} \Big\}.
\label{E:DISCVEL}  
\end{equation}
Identity \eqref{E:DISCVEL} shows that when two particles collide, then the
velocity of the left particle cannot be greater than the velocity of the right
particle, so the left particle cannot pass.

Assume now that $\xx(t)\in\partial\KKN$ at some time $t$ and let $\vv(t-)$ be
the velocity immediately before the impact. That is, let $\vv(t-)$ be the
left-derivative of the curve $t \mapsto \xx(t)$ at time $t$. The instan\-taneous
force that is active on impact must change the velocity to a new value in the
tangent cone $T_{\sxx(t)}\KKN$ of admissible velocities. Typically, there are
many possibilities. Assuming inelastic collisions, we impose the impact law:
\begin{equation}
  \vv(t+) := \P\SUB{T_{\ssxx(t)}\KKN} \vv(t-),
\label{E:IMPACT}  
\end{equation}
where $\vv(t+)$ is the velocity immediately after impact (the right-derivative
of $\xx(t)$). We denote by $\PROJ{T_{\ssxx(t)}\KKN}$ the metric projection onto
$T_{\sxx(t)}\KKN$ so that
$$
  \|\vv(t-)-\vv(t+)\| = \min\Big\{ \|\vv(t-)-\uu\|_\smm \colon 
    \uu\in T_{\sxx(t)}\KKN \Big\}.
$$
Hence $\vv(t+)$ is the element in $T_{\sxx(t)}\KKN$ closest to $\vv(t-)$ with
respect to the weighted Euclidean distance induced by the norm
\begin{equation}
  \label{eq:24}
  \|\vv\|_{\smm}:=\sqrt{ \sum_{i=1}^N m_i v^2_i }
  \quad\text{for all $\vv\in T_{\sxx(t)}\KKN\subset \R^N$,}
\end{equation}
and therefore unique: see Figure~\ref{F:IMPACT} for a graphic representation of
this rule.
\begin{figure}[t]
\includegraphics[scale=0.75]{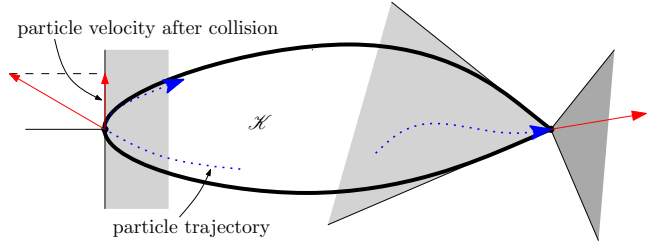}
\caption{Projection of velocities onto the tangent cone.}
\label{F:IMPACT}
\end{figure}

It is well-known that the metric projection onto closed convex cones admits a
variational characterization of its minimizers; see \cite{Zarantonello}. In
particular, we have
$$
\Big( \vv(t-)-\vv(t+) \Big) \cdot \uu \LS 0
    \quad\text{for all $\uu\in T_{\sxx(t)}\KKN$.}
$$
We deduce that the instantaneous force that changes the velocity upon impact
onto the boundary $\partial\KKN$, must be an element of the normal cone
$N_{\sxx(t)}\KKN$, which is defined as
\begin{equation}
  N_\sxx\KKN := \Big\{ \nn\in\R^N \colon 
  \text{$\nn\cdot (\yy-\xx) \LS 0$ for all $\yy \in\KKN$} \Big\}.
\label{E:NORM}
\end{equation}
Note that the normal cone $N_\sxx\KKN$ equals the subdifferential $\partial
I_{\KKN}(\xx)$ of the indicator function $I_{\KKN}$ of $\KKN$ at the point
$\xx$. This follows immediately from the definition of the subdifferential.

This suggests to consider the second-order differential inclusion
\begin{equation}
  \dot{\xx} = \vv,
  \quad\dot{\vv} + N_{\sxx}\KKN \ni \aa_\smm(\xx)
  \quad\text{in $[0,\infty)$.}
\label{E:SODI}
\end{equation}
Notice that since $\vv$ can exhibit jumps, solutions to \eqref{E:SODI} should be
properly defined in a weak sense in the framework of functions of bounded
variation. Second-order differential inclusion have been studied in the
literature and existence of solutions has been shown in a {\em genuinely finite
dimensional setting}. We refer the reader to \cites{BernicotVernel, Moreau,
Schatzman} and the references therein for further information. Due to the
possible {\em nonuniqueness} of solutions to second-order differential
inclusions \cite{Schatzman} and to the lack of estimates to pass to the limit
when $N\to\infty$, we need a better understanding of the particular features of
our setting, in particular of the convex cones $\KKN$.


\subsubsection*{The sticky condition and an equivalent formulation of
\eqref{E:SODI}}

It has been shown in \cite{NatileSavare} that the one-parameter family of normal
cones $N_{\sxx(t)}\KKN$ along an evolution curve $\xx:[0,\infty)\to\KKN$ for
which a gobal stickyness condition holds, satisfies the remarkable monotonicity
property
\begin{equation}
  \label{eq:19}
  N_{\sxx(s)}\KKN\subset N_{\sxx(t)}\KKN
  \quad\text{for all $s<t$.}
\end{equation}
Consequently, for any selection $\xxi:[0,\infty)\to\R^N$ satisfying $\xxi(t)\in
N_{\sxx(t)}\KKN$ (such as $\aa_\smm(\xx(t))-\dot\vv(t)$ in \eqref{E:SODI}) we
have
$$
  \int_s^t \xxi(r)\,\d r\in N_{\sxx(t)}\KKN
  \quad \text{for all $s<t$.}
$$
An integration of \eqref{E:SODI} yields, at least formally, that
$$
  \vv(t)+N_{\sxx(t)}\KKN\ni \vv(s)+\int_s^t \aa(\xx(r))\,\d r
$$
and therefore the system \eqref{E:SODI} can be rewritten in the form
\begin{equation}
  \label{eq:18}
  \dot\xx=\vv,\quad \vv+N_{\sxx}\KKN\ni \yy,\quad \dot \yy=\aa_\smm(\xx).
\end{equation}
Introducing new unknowns $(\xx,\yy)$, we can rewrite \eqref{eq:18} as a
first order evolution inclusion
\begin{equation}
  \label{eq:20}
  \begin{aligned}
    \dot\xx+N_{\sxx(t)}\KKN &\ni \yy\\
    \dot\yy &= \aa_\smm(\xx)
  \end{aligned}
\end{equation}
for which an existence and stability theory is available, at least
when $\aa_\smm$ is a Lipschitz map.

We will show that formulation \eqref{eq:20} enjoys interesting features and
always induces a measure-valued solution to \eqref{E:CONS2}. When the field
$\aa_\smm$ satisfies the compatibility condition \eqref{eq:9}, solutions to
\eqref{eq:20} satisfies the sticky condition, and the same property holds also
for the Euler-Poisson equation in the attractive regime. In the repulsive case,
we will see that \eqref{eq:20} is a robust formulation of condition
\eqref{eq:10}. Let us now consider the infinite-dimensional case.


\subsection{Diffuse measures and differential inclusions for Lagrangian
para\-metrizations}\label{S:DMADIFLP}

In order to deal with general measure-valued solutions of \eqref{E:CONS2}, we
had to recourse to Lagrangian coordinates, using ideas of optimal transport as
considered in \cite{NatileSavare}.


\subsubsection*{Monotone Lagrangian rearrangemens}

In this approach, the discrete set of parameters $\{1,2,\cdots,N\}$ involved in
the representation of discrete particle measures \eqref{E:DISCDATA} will be
substituted by $\Omega=(0,1)$. For every particle labeled by $\w\in \Omega$, we
will denote by $X(t,\w) \in \R$ its position at time $t$. The map $X$ can be
uniquely characterized in terms of the measure $\RHO$: it is the uniquely
determined nondecreasing and right-continuous map $X:\Omega\to\R$ such that
\begin{equation}
  \label{eq:25}
  X(\w)\LS x
  \quad\Longleftrightarrow\quad
  \w\LS \RHO\big((-\infty,x]\big)
  \quad\text{for all $x\in \R$}.
\end{equation}
Equivalently, the push-forward $X_\# \leb$ of the one-dimensional Lebesgue
measure $\leb := \LEB^1|_\Omega$ under the map $X$ equals $\RHO$. Recall that
the push-forward measure is defined by
\begin{equation}
  X_\# \leb (A) := \leb(X^{-1}(A)) 
    \quad\text{for all Borel sets $A \subset \R$.}
\label{E:PUSH}
\end{equation}
Therefore the map $X$ is the optimal transport map pushing $\leb$ forward to
$\RHO$. We refer the reader to Section~\ref{S:P} for further explanation.

In this way, to any solution $(\RHO,v)$ of \eqref{E:CONS2}, we can associate a
map $X \colon [0,\infty) \times \Omega \longrightarrow \R$ with $X(t,\cdot)$
nondecreasing and a velocity $V:[0,+\infty)\times \Omega\to \R$ such that
\begin{equation}
    X(t,\cdot)_\# \leb = \RHO(t,\cdot)
    \quad\text{for all $t\GS 0$},\quad
    V(t,\cdot)=v(t,X(t,\cdot))=\partial_t X(t,\cdot).
\label{E:TRANS}
\end{equation}
Our goal is to show that \eqref{E:CONS2} can be associated to a differential
inclusion in terms of $(X,V)$. This observation allows us to derive existence
and stability results (see Sections \ref{S:SR} and \ref{S:SGL}) for (suitably
defined) solutions of \eqref{E:CONS2}, which together with the existence of
discrete particle solutions (see Section~\ref{S:DODP}) imply a global existence
result for \eqref{E:CONS2} for general initial data; see Section~\ref{S:GE}.


\subsubsection*{Differential inclusions}

The framework of first-order differential inclusions, analogous to the
setting we already discussed for the discrete case \eqref{E:SODI}, serves as a
guiding principle for our discussion. The role of the cone $\KKN$ is now played
by the cone of optimal transport maps
\begin{equation}\label{E:CONEK}
  \K := \Big\{ X\in\L^2(\Omega) \colon \text{$X$ is nondecreasing} \Big\}
\end{equation}
in the Hilbert space $\H:=\L^2(\Omega)$. Even if in this infinite dimensional
setting the boundary of $\K$ is dense, we can still consider the normal cone
$N_X\K$ for given $X\in\K$, which is given by
\begin{equation}
  N_X\K := \Big\{ W\in\L^2(\Omega) \colon 
    \text{$\int_\Omega W \Big( \tilde{X}-X \Big) \,\d m \LS 0$
      for all $\tilde{X}\in\K$} \Big\}.
\label{E:NORCO}
\end{equation}
Again we have that $N_X\K = \partial I_\K(X)$. It can be shown that
$N_X\K=\{0\}$ if and only if the map $X$ is not strictly increasing in $\Omega$.
That is, whenever $\Omega_X \neq \varnothing$ where
\begin{equation}
  \Omega_X := \Big\{ m\in\Omega \colon
    \text{$X$ is constant in a neighborhood of $m$} \Big\}.
\label{E:OMEX}
\end{equation}
Note that $\Omega_X$ is the complement of the support of the distributional
derivative of $X$.

Consider now a family of densities $t\mapsto\RHO(t,\cdot)$ that satisfies
\eqref{E:CONS2}. Let $t\mapsto X(t,\cdot) \in \K$ be the associated family of
optimal transport maps; see \eqref{E:TRANS}. We want to interpret $X$ as a
solution of differential inclusions, similar to \eqref{E:SODI} and
\eqref{eq:20}.

Even at the continuous level, the monotonicity property \eqref{eq:19} for sticky
particle evolutions plays a crucial role. Note that the optimal transport map
$X\in\K$ takes a constant value $x\in\R$ on some interval $(\alpha,\beta)
\subset \Omega$ if the mass $\beta-\alpha$ (the Lebesgue measure of the
interval) is moving to the same location, thereby forming a Dirac measure at
$x$. Therefore sticky evolutions will be characterized as curves $t\mapsto
X(t,\cdot)$ with the property that
\begin{equation}
  \text{for any $t_1 \LS t_2$ we have $\Omega_{X(t_1)} \subset
    \Omega_{X(t_2)}$.}
\label{E:MON2}
\end{equation}
Notice that \eqref{E:MON2} implies that once a Dirac measure is formed, it may
accrete more mass over time, but it can never lose mass. It also implies the
following statement: It is not possible for mass to jump from one side of a
Dirac measure to the opposite side. Whenever mass is crossing a Dirac measure,
it gets absorbed.

A formulation via differential inclusions needs a Lagrangian expression of the
force term in \eqref{E:CONS2}. That is, we must find a map $F\colon \K
\longrightarrow \L^2(\Omega)$ with the property that
\begin{equation}
  \int_\R \psi(x) \, f[\RHO](dx) = \int_\Omega \psi(X(m)) F[X](m) \,\d m
  \quad\text{for all $\psi\in\D(\R)$,}
\label{E:RETURN}
\end{equation}
whenever $X\in\K$ and $X_\#\leb = \RHO$. We refer the reader to
Section~\ref{S:GE} for further discussion about the existence and properties of
maps $F$ satisfying \eqref{E:RETURN}. In the following, we will assume that $F$
is continuous as a map of $\K$ into $\L^2(\Omega).$

We then could expect $X$ to be a solution of a
second-order differential inclusion,
but arguing as for the discrete case \eqref{E:SODI} 
at least in the case of sticky evolutions \eqref{E:MON2}
we end up with 
\begin{equation}
  \dot{X}(t) + \partial I_\K(X(t)) \ni \bar{V} + \int_0^t F[X(s)] \,\d
  s\label{eq:59}
\end{equation}
for a.e.\ $t>0$. This formulation and its consequences is at the heart of our
argument.

It is a remarkable fact (see Theorem \ref{T:BASIC}) that solutions to
\eqref{eq:59} always parametrize measure-valued solutions to the partial
differential equation \eqref{E:CONS2}. Provided $F$ satisfies suitable continuity
properties, it will be possible to prove existence (and uniqueness, when $F$ is
Lipschitz) of solutions to \eqref{eq:59} for any initial data $(\bar X,\bar
V)\in \K\times \L^2(\Omega)$ by combining the theory of gradient flows of convex
functionals in Hilbert spaces \cite{Brezis} with suitable compactness
arguments.

When $F$ satisfies a suitable {\em sticking} condition, which is satisfied e.g.\
in the case of $\C^1$ potentials in \eqref{eq:1} and of the Euler-Poisson system
in the attractive regime, then solutions to \eqref{eq:59} form a semigroup and
have the sticky evolution property \eqref{E:MON2}. Even for general $F$ (and in
particular for the Euler-Poisson system in the repulsive regime) the
differential inclusion \eqref{eq:59} still selects a stable parametrization of
solutions to \eqref{E:CONS2}. This is somewhat surprising since the reduction
from second-order to first-order differential inclusion was motivated by the
monotonicity \eqref{E:MON2}, which typically is false without additional
assumptions on $F$.  In this Introduction we refer to such solutions as ``robust.'' 


\subsubsection*{Representation formulae for the Euler-Poisson system}

In the case of the Euler-Poisson system \eqref{eq:2} with $\sigma=0$ one can
show that the Lagrangian representation of the force $f$ is given by
\begin{equation}
  \label{eq:58}
  F[X](m)= -\lambda A(m)
  \quad\text{where}\quad
  A(m):=m-\frac 12.
\end{equation}
Note that the map $F[X]$ is {\em independent of $X$}
and \eqref{eq:59} becomes
$$
  \dot{X}(t) + \partial I_\K(X(t)) \ni \bar{V} - \lambda t A
$$
In the attractive regime when $\lambda\ge0$, an explicit representation
formula for the Lagrangian solution can be obtained (see Theorems~\ref{T:RFAE}).
In fact a careful analysis shows that the solution $X$ to \eqref{eq:59} can be
computed by solving the trivial ODE in $X$ obtained by eliminating the
$\K$-constraint:
$$
  \frac \d{\d t}\tilde X(t)=\bar V-\lambda tA
$$
(whose solution is $\tilde X(t)=\bar X+t\,\bar V-\frac{1}{2}\lambda t^2 A$),
and then projecting $\tilde X$ on $\K$:
$$
  X(t) = \PROJ{\K}(\tilde X(t))
    =\PROJ{\K}\bigg( \bar X+t\,\bar V-\frac{1}{2}\lambda t^2 A \bigg).
$$
Applying the characterization given in \cite{NatileSavare}, the metric
projection of $\PROJ\K$ onto $\K$ can be found by introducing the primitive
functions
$$
  \bar{\mathcal X}(m):=\int_0^m \bar X(\ell)\,\d\ell,\quad
  \bar {\mathcal V}(m):=\int_0^m \bar V(\ell)\,\d\ell,\quad
  \mathcal A(m):=\frac 12 (m^2-m),
$$
and the time evolution
$$
  \mathcal X(t,m) :=  \bar{\mathcal X}(m)+t\, \bar {\mathcal V}(m)
    -\frac{1}{2}\lambda t^2 \mathcal A(m),
$$
and then taking the derivative with respect to $m$ of the convex envelope
$\mathcal X^{**}(t,\cdot)$ of $\mathcal X(t,\cdot)$:
$$
  X(t,m)=\frac{\partial}{\partial m}\mathcal X^{**}(t,m),
$$
which defines a density $\RHO_t=X(t)_\#\leb$. It is then a simple exercise to
recover formula \eqref{eq:8} in the case $\bar X(m)\equiv \bar x,\ \bar
V(m)\equiv \bar v$, since $\mathcal X(t,m)=\mathcal X^{**}(t,m)$ and
$X(t,m)=\bar x+t\bar v+\frac{1}{2} t^2(m-\frac 12).$


\subsection{Time discrete schemes}\label{SS:TDS}

In this section, we show that the first-order differential inclusion
\eqref{eq:59} can be used to design a stable explicit numerical scheme to
compute robust solutions to \eqref{E:CONS2}. In fact, this scheme is essentially the
same as the one introduced in \cite{Brenier2} for ``order-preserving vibrating
strings'' and ``sticky particles'', with just mild modifications. For
simplicity, we concentrate on the pressureless repulsive Euler-Poisson system
with a neutralizing background
\begin{align}
\label{repulsive1}
  \partial_t \RHO+\partial_x(\RHO v)&=0,
\\
\label{repulsive2}
\partial_t(\RHO v)+\partial_x(\RHO v^2)&=-\RHO \,\partial_x
q_\RHO\qquad
  -\partial_{xx} q_\RHO=-(\RHO-1),
\end{align}
(cf.\ \eqref{eq:2} with $\lambda=-1$ and $\sigma=1$). We assume the initial
conditions to be 1-periodic in $x$ and the density $\RHO$ to have unit mean so
that the system is globally neutral and the electric potential $q_\RHO$ is
1-periodic in $x$. Note that we choose the periodic setting only for
convenience. In fact, for any non-periodic solution $\RHO$ of \eqref{E:CONS2},
one could consider the push-forward of $\RHO$ under the map $x\mapsto x-[x]$ for
all $x\in\R$, with $[x]$ the largest integer not greater than $x$. One obtains a
new density $\RHO^*$ that is concentrated on $[0,1)$ and therefore can be
extended 1-periodically to the whole real line. One can then show that $\RHO^*$
satisfies the same equation. We refer the reader to \cite{GangboTudorascu} for
details.

For smooth solutions without mass concentration, written in
mass coordinates
$$
  V(t,m)=\dot{X}(t,m)=v(t,X(t,m)),
  \quad
  \partial_m X(t,m)\RHO(t,X(t,m))=1
$$
(which requires that $\partial_mX(t,m)\GS 0$), one can show that the whole
system reduces to a collection of independent linear pendulums labeled by their
equilibrium position $m$ and subject to
\begin{equation}
\label{harmonic}
  \dot{X}(t,m)=V(t,m),
  \quad
  \dot{V}(t,m)+X(t,m)-m=0.
\end{equation}
(Notice that, due to the spatial periodicity of the initial conditions, the new
unknown $X(t,m)-m$ and $V(t,m)$ are 1-periodic in $m$.) This reduction is valid
as long as the pendulums stay ``well-ordered'' and do not cross each other,
i.e., as long as $X(t,m)$ stays monotonically nondecreasing in $m$. This
``non-crossing'' condition is not sustainable for large initial conditions and
collision generally occur in finite time. To handle sticky collisions, the
concept of robust solutions introduced in Section~\ref{S:DMADIFLP} is a good way
to obtain a well-posed mathematical model beyond collisions.

We are now ready to describe the semi-discrete scheme. Given a time step
$\tau>0$ and suitable initial data $(\bar{X},\bar{V}) =: (X_{\tau,0},
V_{\tau,0})$, we denote by $(X_{\tau,n}(m), V_{\EEE \tau,n}(m))$ the approximate
solution at time $t_n:=n\tau$, for $n=0,1,2,\ldots$, defined in two steps as
follows:
\begin{enumerate}
\item {\em Predictor step:} we first integrate the ODE (\ref{harmonic}) and get
$U_{\tau,n+1}$ and $\hat X_{\tau,n+1}$ accordingly
\begin{gather}
\label{predictor 1}
  \hat X_{\tau,n+1}(m)=m + (X_{\tau,n}(m)-m)\cos(\tau) + V_{\tau,n}(m)\sin(\tau),
\\
\label{predictor 2}
  V_{\tau,n+1}(m)=-(X_{\tau,n}(m)-m)\sin(\tau) + V_{\tau,n}(m)\cos(\tau).
\end{gather}
\item {\em Corrector step:} we rearrange $\hat X_{\tau,n+1}(m)$ in nondecreasing
order with respect to $m$ and obtain $X_{\tau,n+1}(m)$. Because of the periodic
boundary conditions, we have to perform this step with care. We rely on the
existence, for each map $m\mapsto Y(m)$ such that $Y(m)-m$ is 1-periodic and
locally Lebesgue integrable, of a unique map $m\mapsto Y^*(m)$ such that $Y^*(m)$ 
is nondecreasing in $m$ and
$$
\int_0^1 \eta(Y^*(m)) \,\d m = \int_0^1 \eta(Y(m)) \,\d m
$$
for all continuous 1-periodic function $\eta$.
\end{enumerate}

This time discrete scheme becomes a fully discrete scheme, if the initial data
$X_{\tau,0}(m)-m$ and $V_{\tau,0}(m)$ are piecewise constant on a uniform
cartesian grid with step $h$. (We just have to be careful with the corrector
step, by using a suitable sorting algorithm for periodic data.)

To illustrate the scheme, we show the numerical solutions corresponding to
initial conditions
\begin{equation}
\label{initial conditions}
  X_0(m)=m,
  \quad
  V_0(m)=4\sin(2\pi m).
\end{equation}
We use 400 equally spaced grid points $m$ (which corresonds to 400
``well-ordered'' pendulums with $m$ as equilibrium position) and $5000$ time
steps (see Figures~\ref{fig1}--\ref{fig3}):
\begin{figure}[t]
\includegraphics[scale=0.5]{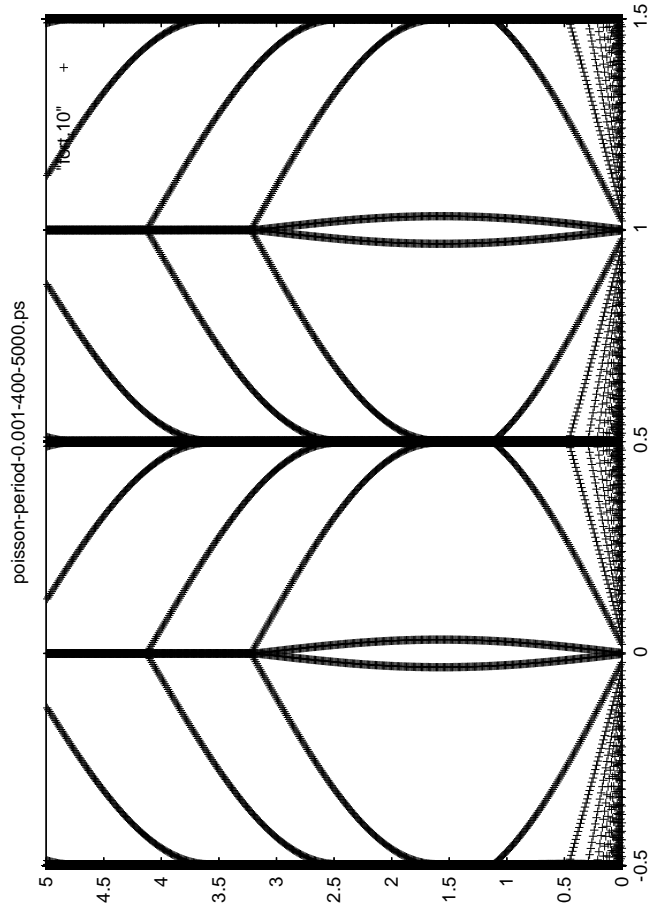}
\caption{Space-time trajectories of pendulums, with timestep $\tau=0.001$.}
\label{fig1}
\end{figure}
\begin{figure}[t]
\includegraphics[scale=0.5]{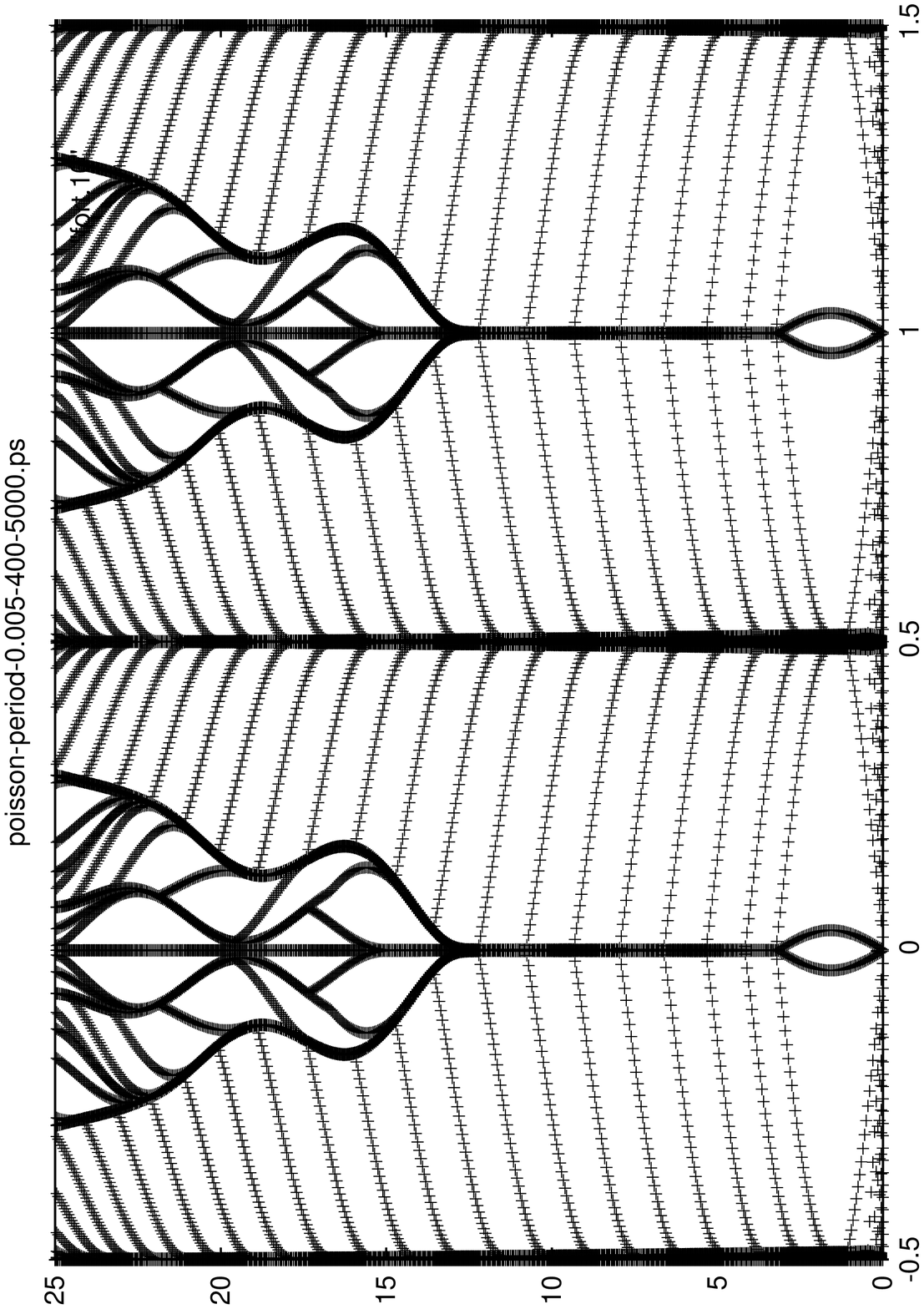}
\caption{Space-time trajectories of pendulums, with timestep $\tau=0.005$.}
\label{fig2}
\end{figure}
\begin{figure}[t]
\includegraphics[scale=0.5]{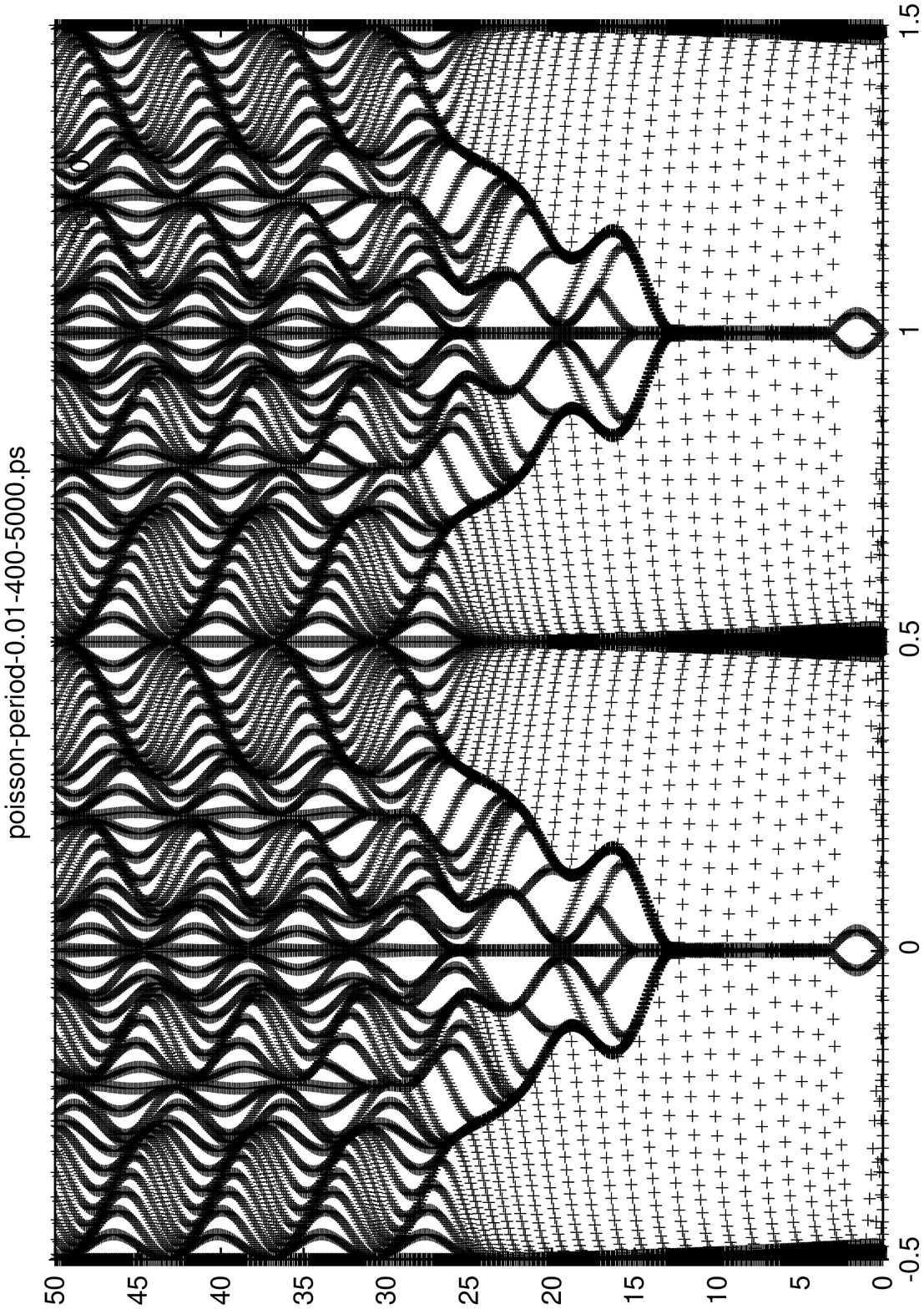}
\caption{Space-time trajectories of pendulums, with timestep $\tau=0.01$.}
\label{fig3}
\end{figure}
$$
  \tau=0.001,
  \quad
  0.005,
  \quad
  0.01,
$$
so that the final time $T$ of observation is respectively given by
$$
  T=5,
  \quad
  25,
  \quad
  50.
$$
On each picture, we show the space-time trajectories of 50 of the 400 pendulums,
with space coordinate on the horizontal axis and time coordinate on the vertical
one. On these pictures, we observe a strong concentration, with sticky
collisions, of the pendulums at a very early stage (up to time $t=\pi/2$) around
$x=0.5$. Later on, some pendulums start to unstick and detach from each other
(which allows new concentrations at later times $t\GS \pi$ around $x=0$ and
$x=1$). Much later, after $t=10\pi$, there is no further dissipation of energy,
and, as pendulums touch each other, they always do so with zero relative speed.
Then the corrector step is no longer active, and the scheme becomes exact (due
to the exact integration of the predictor step). At this late stage, the
solution becomes $2\pi$-periodic in time. We study the convergence of the scheme
in Section~\ref{S:CONV}.


\subsection{Plan of the paper}

We collect in Section~\ref{S:P} a few basic results on optimal transport in one
dimension, on convex analysis (concerning in particular the properties of the
convex cone $\K$), and on convex functionals in $\L^2(\Omega)$.

In Section~\ref{S:SR}, after a brief discussion of the basic properties of the
Lagrangian force functional $F$, we introduce the notion of Lagrangian solutions
to the differential inclusion \eqref{eq:59}. Theorem \ref{T:BASIC} collects
their main properties, in particular in connection with measure-valued solutions
to \eqref{E:CONS2}. Sections~\ref{S:EUSL} and \ref{SS:FPALS} provide the main
existence, uniqueness, and stability results for Lagrangian solutions, whereas
Section~\ref{S:SLSSP} is devoted to the particular case of sticky evolutions.

We study in Section \ref{S:SGL} a different class of solutions to \eqref{eq:59},
still linked to \eqref{E:CONS2}, that naturally arise as limit of sticky
particle systems when $F$ does not obey the sticking condition. These solutions
exhibit better semigroup properties than the Lagrangian solutions introduced in
Section~\ref{S:SR}, but lack uniqueness.

Section \ref{S:DODP} we carefully study the dynamics of discrete particle
systems, which we already briefly discussed in the Introduction. Discrete
Lagrangian solutions associated to systems like \eqref{eq:20} are treated in
Section~\ref{S:DLS}, where we also show that they can be used to approximate any
continuous Lagrangian solution, as the one considered in Section \ref{S:SR}. The
sticky dynamic at the particle level is considered in \S \ref{S:SED}: the main
Theorem \ref{T:EXDISCRETE} provides the basic results, which allow us to replace
second-order with first-order evolution inclusion at the discrete level and to
get sticky evolutions for sticking forces. The particle approach is a crucial
step of our analysis, since it avoids many technical difficulties arising at the
continuous level. The general idea is to prove fine properties of the solutions
(such as the monotonicity \eqref{E:MON2} in the sticking case or a  
representation formula) at the discrete level and then to extend them to the
general case by applying suitable stability results with respect to the initial
conditions. Those are typically obtained by applying contraction estimates (in
the case when $F$ is Lipschitz) or compactness via Helly's Theorem, by
exploiting higher integrability and monotonicity of transport maps.

Section~\ref{S:GE} applies the Lagrangian formulation to \eqref{E:CONS2},
presenting some existence and stability results for solutions in the Eulerian
formalism.

In Section~\ref{S:CONV} we prove the convergence of the time discrete scheme introduced in Section~\ref{SS:TDS}.


\section{Preliminaries}\label{S:P}

Let us first gather some definitions and results that will be needed later.


\subsection{Optimal Transport}\label{SS:OT}

We denote by $\SP(\R^m)$ the space of all Borel probability measures on $\R^m$.
The push-forward $\nu := Y_\#\mu$ of a given measure $\mu\in\SP(\R^m)$
under a Borel map $Y \colon \R^m \longrightarrow \R^n$ is the measure
defined by $\nu(A) := \mu(Y^{-1}(A))$ for all Borel sets $A\subset \R^n$. We
will repeatedly use the change-of-variable formula
\begin{equation}
    \int_{\R^n} \zeta(y) \,(Y_\#\mu)(\d y)
        = \int_{\R^m} \zeta(Y(x)) \,\mu(\d x),
\label{EQ:30}
\end{equation}
which holds for all Borel maps $\zeta \colon \R^n \longrightarrow [0,\infty]$.

We denote by $\SP_2(\R^n)$ the space of all Borel
probability  measures $\RHO \in \SP(\R^n)$ with finite quadratic moment: $\int_{\R^n}
|x|^\p \,\RHO(\d x) < \infty$. The $\L^\p$ Kantorovich-Rubinstein-Wasserstein
distance $W_\p(\RHO_1,\RHO_2)$ between two measures $\RHO_1,\RHO_2\in
\SP_\p(\R^n)$ can be defined in terms of couplings, i.e.\ of probability measures
$\RRHO\in \SP(\R^n\times\R^n)$ satisfying $\pi^i_\#\RRHO = \RHO_i$ for
$i=1\ldots 2$, by the formula
\begin{equation}
    W_\p^\p(\RHO_1,\RHO_2)
        :=\min\bigg\{ \int_{\R^n\times\R^n} |x-y|^\p \,\RRHO(\d x,\d y) \colon
            \RRHO \in \SP(\R^n\times\R^n), \pi^i_\#\RRHO = \RHO_i \bigg\}.
\label{EQ:4}
\end{equation}
Here $\pi^i(x_1,x_2) := x_i$ is the projection on the $i$th coordinate. It can
be shown that there always exists an optimal transport plan $\RRHO$ for which
the $\inf$ in \eqref{EQ:4} is in fact attained. We denote by $\Gamma_\OPT
(\RHO_1,\RHO_2)$ the set of optimal transport plans.

In the one-dimensional case $n=1$, there exists a unique coupling $\RRHO \in
\Gamma_\OPT(\RHO_1,\RHO_2)$ realizing the minimum of \eqref{EQ:4} \IGNORE{and
of \eqref{EQ:97}} (at least when the cost is finite). It can be explicitly
characterized by inverting the distribution functions of $\RHO_1$ and $\RHO_2$:
for any $\RHO\in\SP(\R)$ we consider its cumulative distribution function,
which is defined as
\begin{equation}
    M_\RHO(x) := \RHO\big( (-\infty,x] \big)
    \quad\text{for all $x\in\R$.}
\label{E:53}
\end{equation}
Note that then $\RHO = \partial_x M_\RHO$ in $\D'(\R)$. Its monotone
rearrangement is given by
\begin{equation}
    \X_\RHO(m) := \inf\Big\{ x \colon M_\RHO(x) > m \Big\}
    \quad\text{for all $m\in \Omega$,}
\label{E:90}
\end{equation}
where $\Omega := (0,1)$. The map $X_\RHO$ is right-continuous and
nondecreasing. We have
\begin{equation}
  \leb:=\LEB^1\restr\Omega,\quad
  (X_\RHO)_\#\leb = \RHO
    \quad\text{and}\quad
    \int_\R \zeta(x) \,\RHO(\d x) = \int_\Omega \zeta(X_\RHO(m)) \,\d m
\label{EQ:6}
\end{equation}
for all Borel maps $\zeta \colon \R \longrightarrow [0,\infty]$.
In particular, we have that 
$\RHO \in \SP_\p(\R)$ if and only if $\X_\RHO \in \L^\p(\Omega)$. The
Hoeffding-Fr\'echet theorem \cite{Rachev-Ruschendorf98I}*{Section~3.1} shows
that the joint map $\X_{\RHO_1,\RHO_2} \colon \Omega \longrightarrow
\R\times\R$ defined by
$$
    \X_{\RHO_1,\RHO_2}(m) := \Big( \X_{\RHO_1}(m), \X_{\RHO_2}(m) \Big)
    \quad\text{for all $m\in\Omega$,}
$$
characterizes the optimal coupling $\RRHO \in \Gamma_\OPT(\RHO_1,\RHO_2)$
by the formula 
\begin{equation}
    \RRHO = (\X_{\RHO_1,\RHO_2}\big)_\# \leb;
\label{E:146}
\end{equation}
see \cites{DallAglio56, Rachev-Ruschendorf98I, Villani} for further
information. As a consequence, we obtain that
\begin{equation}
  W_\p^\p(\RHO_1, \RHO_2) 
    = \int_\Omega |\X_{\RHO_1}(m)-\X_{\RHO_2}(m)|^\p \,\d m
    = \|X_{\RHO_1}-X_{\RHO_2}\|^\p_{\L^\p(\Omega)}.
\label{E:W2}
\end{equation}
The map $\RHO \mapsto X_\RHO$ is an isometry between $\SP_2(\R)$ and
$\K$, where $\K \subset \L^2(\Omega)$ is the set of nondecreasing functions.
Without loss of generality, we may consider precise representatives of
nondecreasing functions only, which are defined everywhere.



\subsection{Some Tools of Convex Analysis for $\K$}\label{SS:MP}

Let $\K$ be the collection of right-continuous nondecreasing functions in
$\L^2(\Omega)$ introduced in \eqref{E:CONEK}. Then one can check that $\K$ is a
closed convex cone in the Hilbert space $\L^2(\Omega)$. 

\subsubsection*{Metric Projection and Indicator Function}
 
It is well-known that the metric projection onto a nonempty closed convex
set of an Hilbert space is a well defined Lipstchitz map (see e.g.\
\cite{Zarantonello}):
we denote it by $\PROJ{\K}
\colon \L^2(\Omega) \longrightarrow \K$. For all
${X}\in\L^2(\Omega)$ it is characterized by
$$
Y=\PROJ{\K}({X}) \quad\Longleftrightarrow\quad 
Y\in \K,\quad \|Y- X\|_{\L^2(\Omega)}\LS \|\tilde Y-{X}\|_{\L^2(\Omega)}
\quad\text{for every }\tilde Y\in \K,
$$
or, equivalently, by the
following families of variational inequalities
\begin{equation}\label{E:EQUIVA2}
 Y = \PROJ{\K}({X})
  \quad\Longleftrightarrow\quad
    Y\in\K, \quad
    \int_\Omega ({X}-Y)(\tilde Y- Y) \LS 0 
      \quad\text{for all $\tilde Y\in\K$.}
\end{equation}
$\PROJ\K(X)$ admits a more explicit characterization in terms of 
the convex envelope of the primitive of $X$ \cite[Theorem
3.1]{NatileSavare}:
\begin{equation}
  \label{eq:60}
  Y=\PROJ{\K}(X)=\frac{\d^+}{\d m}\mathcal X^{**}(m),\quad
  \mathcal X(m):=\int_0^m X(\ell)\,\d\ell,
\end{equation}
where $\frac{\d^+}{\d m}$ denotes the right derivative 
and
\begin{equation}
\mathcal X^{**}(m):=\sup\Big\{a+b\,m\colon
a,b\in \R,\quad a+b\,m\LS \mathcal X(m)\quad \forall\, m\in (0,1)\Big\}\label{eq:75}
\end{equation}

is the greatest convex and l.s.c.\ function below $\mathcal X$.

\EEE
Let now $I_\K \colon \L^2(\Omega) \longrightarrow [0,+\infty]$ be the indicator
function of $\K$, defined as
$$
  I_\K(X) := \begin{cases} 
      0 & \text{if $X\in\K$,}
\\
      +\infty & \text{otherwise,}
    \end{cases}
$$
which is convex and lower semicontinuous. Its subdifferential is given by
\begin{equation}\label{eq:31}
  \partial I_\K(X) := \bigg\{ Z\in\L^2(\Omega) \colon 
    \text{$I_\K(X) \GS I_\K(X) + \int_\Omega Z(X-X)$
      for all $X \in \L^2(\Omega)$} \Big\}
\end{equation}  
and it is a maximal monotone operator in $\L^2(\Omega)$; in particular
its graph is strongly-weakly closed in $\L^2(\Omega)\times
\L^2(\Omega)$.
\EEE
Notice that $\partial I_\K(X) = \varnothing$ for all $X
\not\in \K$  since in this case $I_\K(X)=\infty$; \EEE
whenever $X\in\K$ we find that 
\begin{equation}
  \partial I_\K(X) = \bigg\{ Z\in\L^2(\Omega) \colon 
    \text{$0 \GS \int_\Omega Z(X-X)$ for all $X \in \K$} \Big\},\label{eq:61}
  \end{equation} 
so that $\partial I_\K(X)$ coincides with the normal cone
$N_X\K$ defined by \eqref{E:NORCO}.
\EEE

\eqref{E:EQUIVA2} implies the following equivalence: For all $X, Y
\in \L^2(\Omega)$ we have
\begin{equation}\label{E:EQUIVA3}
  Y = \PROJ{\K}(X) 
    \quad\Longleftrightarrow\quad
      X-Y \in \partial I_\K(Y).
\end{equation}
Decomposing $X = Y + Z$ in \eqref{E:EQUIVA3} with $Y,Z\in\L^2(\Omega)$, we
find that
\begin{equation}
  Z \in \partial I_\K(Y)
  \quad\Longleftrightarrow\quad
  Y = \PROJ{\K}(Y+Z).
\label{E:DEFPARTIAL}
\end{equation}
 
\begin{lemma}[Contraction]\label{L:CONTRACTION}
Let $\psi\colon \R\longrightarrow (-\infty,+\infty]$ be a convex, lower
semicontinuous function. For all $X_1,X_2 \in \L^2(\Omega)$ we then have 
$$
  \int_\Omega \psi(\PROJ{\K}(X_1)-\PROJ{\K}(X_2)) 
    \LS \int_\Omega \psi(X_1-X_2).
$$
In particular, the metric projection $\PROJ{\K}$ is a contraction with respect
to the $\L^p(\Omega)$-norm with $p\in[1,\infty]$
and for all $X_1,X_2 \in
\L^\p(\Omega)$ we can estimate
\begin{equation}\label{E:CONTR2}
  \|\PROJ{\K}(X_1)-\PROJ{\K}(X_2)\|_{\L^p(\Omega)}
    \LS \|X_1-X_2\|_{\L^p(\Omega)}.
\end{equation}
\end{lemma}
We refer the reader to Theorem~3.1 in \cite{NatileSavare} for a proof. Notice
that by choosing $X_2=0$ in Lemma~\ref{L:CONTRACTION}, for which
$\PROJ{\K}(X_2)=0$, we obtain the inequalities
\begin{gather}
  \int_\Omega \psi(\PROJ{\K}(X)) \LS \int_\Omega \psi(X)
  \quad\text{for all $X\in\L^2(\Omega)$,}
\label{E:ESTPSI}\\
  \vphantom{\int}
  \|\PROJ{\K}(X)\|_{\L^p(\Omega)} \LS \|X\|_{\L^p(\Omega)}
  \quad\text{for all $X\in\L^\p(\Omega)$.}
\label{E:CONTLP}
\end{gather} 
A similar result holds for the $\L^2$-orthogonal projection $\PROJ{\H_X}$ 
onto the closed subspace
$\H_X$,
$X\in \K$, defined by
\begin{equation}
  \H_X := \Big\{ U\in\L^2(\Omega) \colon 
    \text{$U$ is constant in each interval $(\alpha,\beta) \subset 
      \Omega_X$} \Big\}.
\label{E:HX}
\end{equation}
Notice that we have $\PROJ{\H_X}(V) = V$ a.e.\ in
$\Omega\setminus\Omega_X$ and
\begin{equation}
  \PROJ{\H_X}(V) = \fint_\alpha^\beta V(s) \,\d s
  \quad\text{in any maximal interval $(\alpha,\beta) \subset \Omega_X$,}
\label{E:CONSTANT}
\end{equation}
for all $V\in\L^2(\Omega)$.
Jensen's inequality then easily yields
\begin{lemma}[$\H_X$--Contraction]\label{L:CONTRACTION2}
Let $\psi :\R\to [0,\infty]$ be a convex l.s.c.\ function. 
Then for all $X \in \K$ we have 
$$
  \int_\Omega \psi(\PROJ{\H_X}(X) )\,\d m
    \LS \int_\Omega \psi(X)\,\d m.
$$
\end{lemma}  
\EEE
%
%
%
For any pair of functions $X,Y\in\L^1(\Omega)$ we say that $Y$ is dominated by
$X$ and we write $Y\prec X$ if the value of each convex integral functional
on $Y$ is less than the corresponding value on $X$, i.e.
$$
Y\prec X\quad\Leftrightarrow\quad
\int_\Omega \psi(Y)\,\d m \LS \int_\Omega \psi(X)\,\d m
$$
for all convex, lower semicontinuous $\psi\colon \R \longrightarrow
[0,\infty]$. 
Estimate
\eqref{E:ESTPSI} shows that $\PROJ{\K}(X)\prec X$ for all
$X\in\L^2(\Omega)$.


\subsubsection*{Normal and Tangent Cones}
It is immediate to check that the subdifferential \eqref{eq:31} of the indicator
function $I_\K$
coincides with the normal cone
$N_X\K$ of $\K$ at $X\in\K$ defined by
\eqref{E:NORCO}.
Applying \cite[Thm.~3.9]{NatileSavare} we get the following useful
characterization:
\EEE
\begin{lemma}\label{L:NORMAL}
Let $X\in\K$ be given. For given $W\in\L^2(\Omega)$ we denote by
$$
  \Xi_W(m) := \int_0^m W(s) \,\d s
  \quad\text{for all $m\in[0,1]$,}
$$
its primitive. Then $W\in N_X\K$ if and only if $\Xi_W\in\VAN_X$, where
$$
  \VAN_X := \Big\{ \Xi\in\C([0,1]) \colon 
    \text{$\Xi\GS 0$ in $[0,1]$ and $\Xi=0$ in $\Omega\setminus\Omega_X$}
      \Big\}.
$$
\end{lemma}
That is, a function $W$ is in the normal cone $N_X\K$ if and only if it is the
derivative of a nonnegative function $\Xi$ that vanishes in $\Omega\setminus
\Omega_X$. This implies in particular that $W$ vanishes a.e.\ in $\Omega
\setminus\Omega_X$. Moreover, for any maximal interval $(\alpha, \beta)$ in the
open set $\Omega_X$ we have that $\Xi(\alpha) = \Xi(\beta) = 0$, by continuity
of $\Xi$. Thus 
\begin{equation}
  \int_\alpha^\beta W(s) \,\d s = 0
  \quad\text{for any maximal interval $(\alpha,\beta) \subset \Omega_X$.}
\label{E:AVERAGE}
\end{equation}
For later use, we also highlight the following fact: Let $X_1, X_2\in\K$. Then
\begin{equation}
  \Omega_{X_1} \subset \Omega_{X_2}
  \quad\Longrightarrow\quad
  N_{X_1}\K \subset N_{X_2}\K.
\label{E:MONO}
\end{equation}
This follows immediately from the corresponding monotonicity for $\VAN_X$.

Let us now consider the Tangent cone $T_X\K$ to $\K$ at $X\in \K$: it
can be defined as in \eqref{eq:17} by
\begin{equation}
  \label{eq:32}
  T_X \K:=\mathrm{cl}\Big(\bigcup_{\theta\ge0}\theta(\K-X)\Big)=
  \mathrm{cl}\Big\{\theta(\tilde X-X):\tilde X\in\K,\ \theta\ge0\Big\},
\end{equation}
or, equivalently, as the polar cone of $N_X\K$, i.e.
\begin{equation}
  \label{eq:33}
   T_X\K := \Big\{ U\in\L^2(\Omega) \colon
    \text{$\int_\Omega U(m) W(m) \,\d m \LS 0$ for all $W\in N_X\K$} \Big\}. 
\end{equation}
\EEE

\begin{lemma}\label{L:TANGENT}
Let $X\in\K$ be given. Then 
$$
  T_X\K = \Big\{ U\in\L^2(\Omega) \colon 
    \text{$U$ is nondecreasing in each interval $(\alpha,\beta) \subset 
      \Omega_X$} \Big\}.
$$
\end{lemma}

More precisely, the map $U \in T_X\K$ must be nondecreasing up to Lebesgue
null sets. We may assume that $U$ is right-continuous in each $(\alpha,\beta)
\subset \Omega_X$.

\begin{proof}[Proof of Lemma~\ref{L:TANGENT}]
Let $U\in T_X\K$ be given and fix some interval $(\alpha,\beta) \subset
\Omega_X$. For all nonnegative $\varphi\in\D(\Omega)$ with $\SPT\varphi
\subset(\alpha,\beta)$ we have $\varphi' \in N_X\K$ because of
Lemma~\ref{L:NORMAL}. By definition of the tangent cone $T_X\K$ we find that
$$
  \int_\alpha^\beta U(m) \varphi'(m) \,\d m \LS 0
  \quad\text{for all nonnegative $\varphi\in\D(\Omega)$ with
    $\SPT\varphi\subset(\alpha,\beta)$.}
$$
This shows that the distributional derivative of $U$ in $(\alpha,\beta)$ is
a nonnegative Radon measure, and so $U$ is nondecreasing in the interval.

Conversely, assume that $U\in\L^2(\Omega)$ is nondecreasing in each interval
$(\alpha,\beta)$ that is contained in $\Omega_X$. For any $W\in N_X\K$ we then
decompose the integral
\begin{equation}
  \int_\Omega U(m) W(m) \,\d m
    = \int_{\Omega\setminus\Omega_X} U(m) W(m) \,\d m 
    + \sum_n \int_{\alpha_n}^{\beta_n} U(m) W(m) \,\d m,
\label{E:SCALAR} 
\end{equation}
where the sum is over all maximal intervals $(\alpha_n, \beta_n) \subset
\Omega_X$ (at most countably many). Then the first integral on the right-hand
side vanishes because $W(m) = 0$ for a.e.\ $m\in\Omega\setminus\Omega_X$. For
each integral in the sum, an approximation argument (see again Lemma~3.10 in
\cite{NatileSavare}) allows us to integrate by parts to obtain
$$
  \int_{\alpha_n}^{\beta_n} U(m) W(m) \,\d m
    = -\int_{\alpha_n}^{\beta_n} \Xi_W(s) \gamma(ds),
$$
where $\gamma$ is the distributional derivative of $U$ in $(\alpha_n,
\beta_n)$. Since $U$ is assumed nondecreasing and $\Xi_W$ is nonnegative, we
conclude that $U\in T_X\K$.
\end{proof}

%

Recalling Lemma \ref{L:TANGENT} it is immediate to check that 
\begin{equation}
  \label{eq:30}
  U\in \H_X\quad\Longleftrightarrow\quad
  U\text{ and } -U\ \in T_X\K.
\end{equation}
\EEE
Observe that if $X_1,X_2\in\K$ then
\begin{equation} \label{E:INCLUSION-SPACE}
  \Omega_{X_1} \subset \Omega_{X_2}
  \quad\Longrightarrow\quad
  \H_{X_2} \subset \H_{X_1}.
\end{equation}
Whenever $U\in\H_X$, then \eqref{E:SCALAR}
equals zero because every term in the sum vanishes since $U$ is constant and $W$
has vanishing average. Thus
\begin{equation}
  N_X\K \subset \H_X^\perp
  \quad\text{for all $X\in\K$,}
\label{E:ORTHO}
\end{equation}
with $\H_X^\perp$ the orthogonal complement of $\H_X$. 
In particular
\begin{equation}
  \label{eq:70}
  Y=\PROJ{\K}(X)\quad\Longrightarrow\quad
  Y=\PROJ{\H_Y}(X),\quad
  \H_Y\subset \H_X.
\end{equation}
We have in fact a more precise characterization of $\H^\perp_X$ in
terms of $N_X\K$: in the following, let us denote by $\mathscr{I}(\Omega_X)$ the collection of all
maximal intervals $(\alpha,\beta)$ (the connected components) of $\Omega_X$.
\begin{lemma}
  \label{L:LINEAR}
  For every $X\in \K$ the closed subspace $\H^\perp_X$ is
  \begin{equation}
    \label{eq:29}
    \begin{aligned}
      \H^\perp_X=\Big\{W\in \L^2(\Omega)&:W=0\ \text{a.e.\ in
      }\Omega\setminus\Omega_X,\\& \int_\alpha^\beta W(m)\,\d  m=0\quad
      \text{for every }(\alpha,\beta)\in \mathscr I(\Omega_X)\Big\},
    \end{aligned}
  \end{equation}
  and it is the closed linear subspace of $\L^2(\Omega)$ generated by
  $N_X\K$.
  Moreover, it admits the equivalent characterization
 \begin{equation}
  \H_X^\perp = \Big\{ W\in\L^2(\Omega) \colon
      \text{$\int_\Omega W(m) \varphi(X(m)) \,\d m = 0$ for all
        $\varphi\in\C_b(\R)$} \Big\}.\label{eq:35}
          \end{equation}
\end{lemma}
\begin{proof}
  \eqref{eq:29} follows immediately by the definition \eqref{E:HX} of
  $\H_X$.
  \eqref{E:ORTHO} shows that the linear subspace generated by $N_X\K$
  is contained in $\H^\perp_X$; to prove the converse inclusion it is
  sufficient to check that any $U\in \L^2(\Omega)$ orthogonal to all
  the 
  elements of $\K$ is also orthogonal to $\H^\perp_X$, i.e.\ it
  belongs to $\H_X$. This is true, since if $U$ is orthogonal to
  $N_X\K$ then 
  both $U$ and $-U$ belongs to the polar cone to $N_X\K$ which is
  $T_X\K$:
  by \eqref{eq:30} we deduce that $U\in \H_X$.

  Concerning \eqref{eq:35} 
  we simply notice that
   all $U\in\H_X$ can be written as $U=u\circ X$
  for a map $u\in\L^2(\R, \RHO)$, where
  $\RHO=X_\#\leb$. Approximating $u$ in $L^2(\R,\RHO)$ by a family of 
  functions $\varphi\in \C_b(\R)$ we obtain \eqref{eq:35}.
\end{proof}
\EEE
\begin{lemma}\label{L:MINIMAL}
For any $X\in\K$ and $U\in T_{X}\K$ we have that
\begin{equation}
  (\PROJ{\H_X}-\ID) U \in N_X\K.\label{eq:36}
\end{equation}
\end{lemma}

\begin{proof} 
Lemma~3.11 in \cite{NatileSavare} shows that \eqref{eq:36} holds if
$U\in \K$. Since $\PROJ{\H_X}X-X=0$, \eqref{eq:36} holds 
for $U-X$ and, since $N_X\K$ is a cone, for arbitrary $\theta(U-X)$,
$\theta\ge0$ and $U\in \K$. We conclude recalling \eqref{eq:32}.
%
\end{proof} 
%
\begin{remark}\label{R:MINKOWSKI} If $\CC$ 
is a closed convex subset of $\L^2(\Omega)$ and $X\in L^1((0,T);X)$ with
$X(t)\in \CC$ for a.e.\ $t\in (0,T)$ then it is easy to check that
\begin{equation}
  \label{eq:37}
  \fint_0^T X(t)dt \in \CC;\quad\text{moreover}\quad
  \int_0^T X(t)dt \in \CC\quad\text{if $\CC$ is a cone.}
\end{equation}
In fact, 
applying Jensen's inequality to the indicator function of $\CC$ we get
$$
I_\CC \Bigl( \fint_0^T X(t)\,\d t \Bigr) \leq  \fint_0^T I_\CC(X(t))\,\d t =0.
$$ 
If in addition $\CC$ is a cone then $I_\CC(T X)=I_\CC(X)$ for every
$X\in \L^2(\Omega)$ and we deduce the second implication of \eqref{eq:37}.
\end{remark}
\EEE

\subsection{Convex Functions}\label{SS:CF}

In this section we recall some auxiliary results on convex functions. We are
interested in functions $\psi \colon \R \longrightarrow [0,\infty)$ that are
\begin{equation}
    \text{even, convex, of class $\C^1(\R)$, with $\psi(0)=0$,}
\label{EQ:130}
\end{equation}
and for which the homogeneous {\em doubling condition} holds:
\begin{equation}
    \text{there exists $q\GS 1$ such that $\psi(\lambda r) \LS \lambda^q
        \psi(r)$ for all $r\in\R, \lambda\GS 1$.}
\label{E:131}
\end{equation}
Notice that if condition \eqref{E:131} holds for $\psi$, then it also holds for
the map $r \mapsto \psi^\p(r)$, with exponent $qp$. Combining \eqref{EQ:130} and
\eqref{E:131}, we obtain the inequality
\begin{equation}
  \psi(r_1+r_2) \LS 2^{q-1} (\psi(r_1)+\psi(r_2))
  \quad\text{for all $r_1,r_2\in \R$.}
\label{EQ:112}
\end{equation}
We will denote by $\Psi \colon \L^1(\Omega) \longrightarrow [0,\infty]$ the
associated convex functional
\begin{equation}
    \Psi[X] := \int_\Omega \psi(X(m)) \,\d m
    \quad\text{for all $X \in \L^1(\Omega)$.}
\label{E:113}
\end{equation}

\begin{lemma}\label{LE:SUFFICIENT}
Suppose $\psi \colon \R \longrightarrow [0,\infty)$ satisfies \eqref{EQ:130}.
Then the doubling condition \eqref{E:131} holds if and only if $\psi$ has one
of the following, equivalent properties:
\begin{gather}
    \text{$r\psi'(r) \LS q \psi(r)$ for all $r>0$;}
\label{EQ:129}\\
    \text{there exists $C\GS 0$ such that $\psi(2r) \LS C\psi(r)$
        for all $r>0$.}
\label{EQ:136}
\end{gather}
\end{lemma}

\begin{proof}
Property \eqref{EQ:136} is a  consequence of \eqref{E:131}, and
\eqref{EQ:129} follows from
$$
    r\psi'(r) 
        = \lim_{\lambda\rightarrow 1+} \frac{\psi(\lambda r)-\psi(r)}{\lambda-1}
        \LS \lim_{\lambda\rightarrow 1+} \frac{\lambda^q-1}{\lambda-1} \psi(r)
        = q \psi(r)
$$
for all $r>0$ and $q>1$.

To prove the converse statement, we notice first that since $\psi$ is an
even, smooth function, we have that $\psi'(0)=0$ and so $\psi$ is nonnegative
and nondecreasing for all $r>0$, by convexity. Moreover, if \eqref{EQ:136}
holds, then again by convexity we find
$$
    \psi'(r) 
        \LS \frac{\psi(2r)-\psi(r)}r
        \LS (C-1)\frac{\psi(r)}r
    \quad\text{for all $r>0$.}
$$
Thus \eqref{EQ:129} holds with $q := C-1$, which must not only be a nonnegative number but must be greater than or equal to $1.$ 

Assuming now that \eqref{EQ:129} is true, we consider the Cauchy problem
\begin{equation}
    \eta'(s) = q\frac{\eta(s)}{s}
    \quad\text{for $s\in[r,\infty)$, with $\eta(r)=\psi(r)$,}
\label{EQ:133}
\end{equation}
which admits a unique solution $\eta(s) = \psi(r)(r^{-1}s)^q$ for all $s\GS
r>0$. A standard comparison estimate for solutions of ordinary differential
equation yields
\begin{equation}
    \psi(s) \LS \eta(s) = \Big(\frac{s}{r}\Big)^q \psi(r)
    \quad\text{for all $s\GS r > 0$.}
\label{EQ:134}
\end{equation}
Since $\psi$ is nondecreasing, we conclude that $q\GS 1$ and then \eqref{E:131}
follows for $r>0$. By evenness of $\psi$ and since $\psi(0)=0$, the inequality
extends to $r\LS 0$ as well.
\end{proof}

\begin{lemma}\label{L:DOUBLE}
Let $p \in [1,\infty)$ be given and supppose that the function $\eta \colon \R
\longrightarrow [0,\infty)$ satisfies \eqref{EQ:130} and
the $p$-coercivity condition
\begin{equation}
    0 < \liminf_{r\rightarrow 0+} \frac{\eta(r)}{r^\p}
    \quad\text{and}\quad
    \lim_{r\rightarrow\infty} \frac{\eta(r)}{r^\p} = \infty.
\label{EQ:139}
\end{equation}
For every $q>p$, there exists a map $\psi \colon \R \longrightarrow
[0,\infty)$ satisfying \eqref{EQ:130}/\eqref{E:131} with
\begin{equation}
    \text{$\psi(r) \LS \eta(r)$ for all $r\in\R$}
    \quad\text{and}\quad
    \lim_{r\rightarrow\infty} \frac{\psi(r)}{r^\p} = \infty.
\label{EQ:128}
\end{equation}
\end{lemma}

\begin{proof}
By \cite{Rossi-Savare03}*{Lemma 3.7} it is not restrictive to assume that $\eta$
is of the form $\bar{\eta}^\p$ for a suitable convex function $\bar{\eta}$ with
superlinear growth, and so we may just consider the case $p=1$ (see the remark
following \eqref{E:131}). By convolution, we can assume that $\eta$ is smooth
in the open interval $(0,\infty)$, with $\delta := \inf_{r>0} \eta'(r) > 0$.

We then choose $q>1$ and we set $\psi(r) := \delta r^q / q$ for all
$r\in [0,1]$, so that 
$$
    \text{$\psi(r) \LS \eta(r)$ in $[0,1]$}
    \quad\text{and}\quad
    \psi'(1) = \delta \LS \eta'(1).
$$
For $r\GS 1$ we define $\psi$ to be the solution of the Cauchy problem
\begin{equation}
    \psi'(r) = \min\bigg\{ \eta'(r), q \frac{\psi(r)}{r} \bigg\}
    \quad\text{for $r\in[1,\infty)$, with $\psi(1) = \delta/q$.}
\label{EQ:135}
\end{equation}
Then $\psi(r) \LS \eta(r)$ for all $r\GS 0$ and $\psi$ satisfies
\eqref{EQ:129} of the previous lemma.

To prove that $\psi$ also satisfies \eqref{EQ:130}, notice first that $\psi'(0)
= 0$ since $q>1$, and that $\psi'$ is continuous at $r=1$. Hence $\psi$ can
be extended to an even $\C^1(\R)$-function. In order to check that $\psi'$ is
nondecreasing, let us first observe as a general fact that if a continuous
function $\beta$ is nondecreasing in each connected component of an open set
$A\subset \R$ that is dense in $[1,+\infty)$, then $\beta$ is nondecreasing in
$[1,+\infty)$. We apply this observation to $\beta := \psi'$ and we set
$A:=A_0\cup A_1$, where
\begin{gather*}
    A_0 := \bigg\{ r\in (1,\infty) \colon
        q\frac{\psi(r)}{r} < \eta'(r) \bigg\},
\\
    A_1 := \text{interior of } 
        \bigg\{ r\in (1,\infty) \colon
        \eta'(r) \LS q\frac{\psi(r)}{r} \bigg\}.
\end{gather*}
In each connected component of $A_0$, the function $\psi$ solves the
differential equation $\psi'(r)=q\psi(r)/r$, and so $\psi$ is of the form $cr^q$
for a suitable constant $c>0$. Therefore $\psi'$ is nondecreasing in $A_0$.
On the other hand, in each connected component of $A_1$, we have that
$\psi'(r)=\eta'(r)$ and $\eta'$ is nondecreasing, by assumption. Finally,
notice that $\psi$ is nondecreasing on the interval $[0,1]$ since $\psi(r) =
\delta r^q / q$ there. We can now apply Lemma~\ref{LE:SUFFICIENT}
to conclude that $\psi$ has the doubling property \eqref{E:131}.

It only remains to prove the second statement in \eqref{EQ:128}. Since $\eta$
has superlinear growth, its derivative $\eta'(r) \longrightarrow \infty$ as
$r\rightarrow\infty$. Assume now that $\psi(r)/r$ remains bounded as
$r\rightarrow\infty$. Then there exists a number $r_1\GS 1$ such that 
$$
    \psi'(r) = q\frac{\psi(r)}{r}
    \quad\text{for all $r \in [r_1,\infty)$;}
$$
see \eqref{EQ:135}. But this implies that $\psi'(r) = cr^q+c_0$ for all
$r\in[r_1, \infty)$ and suitable constant $c>0$ and therefore is unbounded as
$r\rightarrow\infty$. This is a contradiction. 
\end{proof}

\begin{lemma}[Compactness in $\K$]
\label{L:COMPACTNESS}
Let $\Psi$ be the integral functional defined
in \eqref{E:113} corresponding to an even, convex function $\psi \colon \R
\longrightarrow [0,\infty)$ with
\begin{equation}
    \lim_{|r|\rightarrow\infty} \frac{\psi(r)}{|r|^\p} = \infty.
\label{EQ:116}
\end{equation}
Then each sublsevel of $\Psi$
$$
    \K(\Psi,\alpha) := \{ X\in \K \colon \Psi[X] \LS \alpha \}
    \quad\text{is compact in $\L^\p(\Omega)$}\quad
\text{for every }\alpha\ge0.
$$
\end{lemma}

\begin{proof}
Because of \eqref{EQ:116}, the $\L^2(\Omega)$-norm of elements of
$\K(\Psi,\alpha)$ is bounded by some constant $A$ that depends on $\alpha$ and
$\psi$ only. By monotonicity, we find that
\begin{align*}
    X(w) 
        \LS \frac{1}{1-w} \int_w^1 X(m) \,\d m
        &\LS \bigg( \frac{1}{1-w} \int_w^1 |X(m)|^2 \,\d m \bigg)^{1/2}
\\
        &\LS \frac{A}{(1-w)^{1/2}}
    \quad\text{for all $w\in\Omega$,}
\end{align*}
for all $X\in\K(\Psi,\alpha)$. Analogously, we obtain a lower bound
$$
    X(w) \GS -\frac{A}{w^{1/p}}
    \quad\text{for all $w\in\Omega$.}
$$
Any sequence $\{X_n\}$ in $\K(\Psi,\alpha)$ is therefore uniformly bounded in
each compact interval $[\delta,1-\delta]$ where $\delta>0$. Applying Helly's
theorem and a standard diagonal argument we can find a subsequence (still
denoted by $\{X_n\}$ for simplicity) that converges pointwise to an element
$X\in \K$. Since $\psi$ satisfies \eqref{EQ:116}, the sequence
$\{|X_n-X|^\p\}$ is uniformly integrable and thus $X_n\longrightarrow X$ in
$\L^\p(\Omega)$.
\end{proof}

\section{Lagrangian solutions}\label{S:SR} 
As explained in the Introduction, when studying system
\eqref{E:CONS2}, one is lead to consider 
solutions to the Cauchy problem for the first-order differential
inclusion in $\L^2(\Omega)$ 
\begin{equation}\label{E:FODI}
  \dot{X}(t) + \partial I_\K(X(t)) \ni \bar{V} + \int_0^t F[X(s)] \,\d s
  \quad\text{for $t\GS 0$,}\quad
  X(0)=\lim_{t\downarrow0}X(t)=\bar X,
\end{equation}
and, possibly,
satisfying further properties.

\EEE
Before discussing \eqref{E:FODI}, we will state below the
precise assumptions on the force  operator \EEE $F$;
examples, covering the case of \eqref{eq:1} or \eqref{eq:2},
are detailed in Section~\ref{S:GE}. 

\subsection{The force operator $F$}
Let us first recall the link of the map $F\colon \K\to
\L^2(\Omega)$ with the force distribution $f\colon\SP(\R)
\longrightarrow \M(\R)$ in \eqref{E:CONS2}:
as in \eqref{E:RETURN} we will assume that 
\begin{equation}
  \label{eq:39}
    \int_\R \psi(x) \, f[\RHO](\d x) = \int_\Omega \psi(X_\RHO(m)) F[X](m) \,\d m
  \quad\text{for all $\psi\in\D(\R)$,\ $\RHO\in \SP(\R)$}
\end{equation}
recalling \eqref{eq:35} one immediately sees that $F[X]$ is uniquely
characterized by \eqref{eq:39} only when $\H_X^\perp=\{0\}$ or,
equivalently, when $\H_X=\L^2(\Omega)$ i.e. $\Omega_X=\varnothing$:
this is precisely the case when $X$ is (essentially) strictly
increasing.

One could, of course, always take the orthogonal projection of $F[X]$
onto $\H_X$ in order to characterize it starting from
\eqref{eq:39}. This procedure, however, could lead to a discontinuous
operator which would be hard to treat by the theory of first order
differential inclusions. This happens, e.g., for the (attractive or
repulsive) Euler-Poisson system. We thus prefer to 
allow for a greater flexibility in the choice of $F$ complying with \eqref{eq:39}, asking that it
is everywhere defined on $\K$ and satisfies suitable boundedness and
continuity properties.
\EEE
\begin{definition}[Boundedness]\label{D:BOUNDEDNESS}
An operator
$F \colon \K \longrightarrow \L^2(\Omega)$
is \emph{bounded} if there exists a constant $C \GS 0$
such that
\begin{equation}
\label{eq:27}
  \|F[X]\|_{\L^\p(\Omega)} \LS C_2 \Big( 1 + \|X\|_{\L^\p(\Omega)} \Big)
  \quad\text{for all $X\in\K$.}
\end{equation}
We say that $F$ is pointwise linearly bounded if there exists a
constant $C_\sfp \GS
0$ such that
\begin{equation}\label{D:LINEARBOUNDEDNESS}
  \big|F[X](m)\big | \LS C_\sfp \Big(1 + |X(m)| 
  +\|X\|_{\L^1(\Omega)}\EEE\Big)
  \quad\text{for a.e.\ $m\in\Omega$ and all $X\in\K$.}
\end{equation}
\end{definition}
Note that if $F$ is pointwise linearly bounded, then $F$ is bounded
and satisfies \eqref{eq:27} with the constant $C_2:= 2C_\sfp$.

Let us recall that a modulus of continuity is a concave continuous function $\omega
\colon [0,\infty) \longrightarrow [0,\infty)$ with the property that $0
= \omega(0)  < \omega(r)$ for all $r>0$.

\begin{definition}[Uniform continuity]\label{D:CONTINUITY}
We say that an operator $F\colon\K\longrightarrow \L^\p(\Omega)$ is 
uniformly continuous if it is bounded
as in Definition~\ref{D:BOUNDEDNESS} and there exists a modulus of
continuity $\omega$ with the property that 
\begin{equation}
  \big\| F[X_1]-F[X_2] \big\|_{\L^\p(\Omega)} 
    \LS \omega\Big( \|X_1-X_2\|_{\L^\p(\Omega)} \Big)
  \quad\text{for all $X_1,X_2 \in \K$.}
\label{E:UNFO1}
\end{equation}
We say that $F$ is 
Lipschitz continuous if it is 
uniformly continuous 
and \eqref{E:UNFO1} holds with $\omega(r) = Lr$ for all $r\GS 0$, where $L\GS 0$
is some constant.
\end{definition} 
Notice that if $F$ is uniformly continuous then it is also bounded.
Whenever a uniformly continuous $F$ is defined by \eqref{eq:39} on the convex
subset $\K_{si}$ of all the strictly increasing maps and satisfies
\eqref{E:UNFO1} in $\K_{si}$, then it admits a unique extension to
$\K$
preserving the continuity property \eqref{E:UNFO1} and 
the compatibility condition \eqref{eq:39}.

As we observed at the beginning of this section, a last property of $F$ which will play a crucial role concerns 
its behaviour on the subset $\Omega_X$ where the map $X$ is constant. \EEE
Since the force functional determines the change in velocity,  
in the framework of sticky evolution it would be natural to assume
that
\begin{displaymath}
  F[X]\in \H_X\quad \text{for every $X\in \K$}.
\end{displaymath}
 We shall see that a weaker peroperty is still sufficient to preserve
the sticky condition: it will turn particularly useful when the
attractive Euler-Poisson equation will be considered.
\EEE
\begin{definition}[Sticking]\label{D:NONSPLITTING}
The map $F\colon\K \longrightarrow \L^\p(\Omega)$ is called
sticking if for all transport maps $X,Y\in\K$ with $Y\in\H_X$ we
have
$$
  F[X] - \PROJ{\H_X}(F[X]) \in \partial I_\K(Y).
$$
\end{definition}

\subsection{Lagrangian Solutions}

Let us start by giving a suitable notion of solutions to
\eqref{E:FODI}. 

\begin{definition}[Lagrangian solutions to the differential inclusion \eqref{E:FODI}]\label{D:STRONG} 
Let $F:\K\longrightarrow L^\p(\Omega)$ be a uniformly
continuous operator and let $\bar X\in \K$ and $\bar V\in
\H=\L^2(\Omega)$
be given. 
A Lagrangian solution to \eqref{E:FODI} with
initial data $(\bar{X},\bar{V})$ is a curve $X\in
\LIP_{\rm loc}([0,\infty);\K)$ satisfying $X(0)=\bar X$ and \eqref{E:FODI} for a.e.\ $t\in
(0,\infty)$. 
\end{definition} 
By introducing the new variable 
\begin{displaymath}
  Y(t):=\bar V+\int_0^t F[X(s)]\,\d  s 
\end{displaymath}
we immediately see that \eqref{eq:28} is equivalent to the evolution
system
\begin{equation}
  \label{eq:26}
  \left\{
    \begin{aligned}
      \dot{X}(t) + \partial I_\K(X(t)) &\ni Y(t),\\
      \dot Y(t)&=F[X(t)],
    \end{aligned}
    \right.\quad \text{for $t\GS 0$,}\quad
    (X(0),Y(0))=(\bar X,\bar V),
\end{equation}
  Notice that the continuity of $F$ yields $Y\in
  \C^1([0,\infty);\L^2(\Omega))$.
We state in the following Theorem the main properties of the solution
$X$ to
\eqref{E:FODI}
\begin{theorem}
  \label{T:BASIC}
  Let $F:\L^2(\Omega)\to\K$ be a uniformly continuous operator and 
  let $(X,Y)$ be a solution to \eqref{eq:26}. 
  Then the following
  properties hold:
\begin{itemize}
  \item {\em Right-Derivative:}
\begin{equation}
  \text{The right-derivative $V := \frac{\d^+}{\d t} X$ exists for all $t\GS 0$.} \label{E:RIGHTD}
\end{equation}
%
\item {\em Minimal Selection:}
\begin{equation}
  V(t) = \Big( Y(t) - \partial I_\K(X(t)) \Big)^\circ
  \quad\text{for all $t\GS 0$,}
\label{E:MINIMAL}
\end{equation}
where $A^\circ$ denotes the unique element of minimal norm in any
closed convex set of $A\subset
\L^2(\Omega)$. 
In particular if we replace $\dot{X}(t)$ by $V(t)$ then
\eqref{E:FODI} and \eqref{eq:26} hold for all $t\GS 0$.
\item {\em Projection on the tangent cone:}
\begin{equation}
  V(t) = \PROJ{T_{X(t)}\K}\big( Y(t)\big)
  \quad\text{for all $t\GS 0$.}
\label{E:TANGENT}
\end{equation}
\item {\em Continuity of the velocity:} 
  \begin{equation}
  \text{$V$ is right-continuous for all $t\GS 0$;}\label{eq:40}
\end{equation}
in particular
\begin{equation}
  \label{eq:63}
  \lim_{t\downarrow0}V(t)=\bar V\quad\text{if and only if}\quad
  \bar V\in T_{\bar X}\K.
\end{equation}
If $\calT^0\subset (0,\infty)$ is the subset of all times at which the
map
$s\mapsto 
\|V(s)\|_{\L^2(\Omega)}$ is continuous, then $(0,\infty)\setminus
\calT^0$ is negligible and at every point of $\calT^0$ 
 $V$ is continuous and $X$ 
 is differentiable in $\L^2(\Omega)$. Setting $\RHO_t:=X(t)_\#\leb$
 there exists a unique map $v_t\in L^2(\R,\RHO_t)$ such that
\begin{equation}
  \label{eq:34}
  \dot X(t)=V(t)=\PROJ{\H_{X(t)}}(Y(t))=v_t\circ X_t\in \H_{X(t)}\quad
  \text{for every $t\in \calT^0$}.
\end{equation}
\item{\em Solution to \eqref{E:CONS2}:}
  If moreover $F$ is linked to $f$ by \eqref{eq:39},
  $\bar \RHO=\bar X_\#\mm$ and $\bar V=\bar v\circ \bar X$, 
  then 
  the couple
  $(\RHO,v)$ defined as above is a distributional solution to
  \eqref{E:CONS2}
  such that 
  \begin{equation}
    \label{eq:64}
    \lim_{t\downarrow 0}\RHO(t,\cdot)=\bar\RHO\text{ in }\SP_2(\R),\quad
    \lim_{t\downarrow0}\RHO(t,\cdot) v(t,\cdot)=
    \bar\RHO\,\bar v \text{ in }\M(\R).
  \end{equation}
\end{itemize}
\end{theorem}
\begin{proof}
  \eqref{E:RIGHTD}, \eqref{E:MINIMAL}, and \eqref{eq:40}
  are consequence of the general theory of
  \cite{Brezis}, Theorem 3.5; 
  \eqref{E:TANGENT} follows immediately by
  \eqref{E:MINIMAL} since $V(t)\in T_{X(t)}\K$ by \eqref{E:RIGHTD} and 
  $V(t)+N_{X(t)}\K\ni Y(t)$.
  
  Concerning \eqref{eq:34} we can apply
  the Remark 3.9 (but see also Remark 3.4) of \cite{Brezis}, which
  shows that at each differentiability point $t$ of $X$ its derivative is
  the projection of $0$ onto the affine space generated by
  $Y(t)-\partial I_\K(X(t))$, i.e.\ the orthogonal projection of
  $Y(t)$ onto the orthogonal complement of the space generated by
  $\partial I_\K(X(t))$. Recalling Lemma \ref{L:LINEAR} we get \eqref{eq:34}.

  In order to prove the last statement, 
  we use the crucial information of \eqref{eq:34} that $V(t) \in \H_{X(t)}$  for
  a.e. $t>0$, a fact which may have been noticed for the first time
  in\cite{GNT1}.
  In particular, the projected velocities
  \begin{equation}
  V^\ast(t)= P_{\H_{X(t)}} V(t)
  \quad  t \geq 0,\label{eq:38bis}
\end{equation}
  concide with $V(t)$ for  every $ t\in \calT^0$, where $\calT^0$ is a
  set of full measure in $(0,\infty)$.
  Since any element $V$ of $\H_{X(t)}$ can be written as $v\circ X$
  for a suitable Borel map $v\in \L^2(\Omega)$ 
  we deduce that there exists a Borel map 
  $v:[0,\infty) \times \R \rightarrow \R$ such that $v(t,\cdot) \in
  \L^2(\R, \RHO(t,\cdot))$ and 
\begin{equation}
  V^\ast(t,\cdot) = v(t,X(t, \cdot))
  \quad\hbox{a.e. in}\  \Omega, \quad \text{for every }t\GS 0.
\label{E:VELOC}
\end{equation} 
From Equation \eqref{eq:38bis} we also have $V(t)=v(t, X(t))$ for $t \in \calT^0.$
We then argue as follows: For all test functions $\varphi \in \D((0,T)\times \R)$ we have
\begin{align}
  & \int_0^\infty \int_\R \bigg( \partial_t\varphi(t,x) v(t,x) 
      +\partial_x\varphi(t,x) v^2(t,x) \bigg) \,\RHO(t,dx) \,\d t
\nonumber\\
   & \quad 
      = \int_0^\infty \int_\Omega \bigg( \partial_t\varphi(t,X(t,m))
        +\partial_x\varphi(t,X(t,m)) v(t,X(t,m)) \bigg) V(t,m) \,\d m \,\d t
\nonumber\\
    & \quad
      = \int_0^\infty \int_\Omega \bigg( \partial_t\varphi(t,X(t,m))
        +\partial_x\varphi(t,X(t,m)) v(t,X(t,m)) \bigg) Y(t,m) \,\d m \,\d t
 \nonumber\\
    & \quad
      = \int_0^\infty \int_\Omega \bigg( {\d \over \d t} [\varphi(t,X(t,m))]
      \bigg) Y(t,m) \,\d m \,\d t
\nonumber\\
    & \quad
      =- \int_0^\infty \int_\Omega \varphi(t,X(t,m))
          F[X(t,\cdot)](m)  \,\d m \,\d t
\label{E:MKO}
\end{align}
Applying formula \eqref{eq:39} in (\ref{E:MKO}), we
obtain 
\begin{align}
& 
   \int_0^\infty \int_\R \bigg( \partial_t\varphi(t,x) v(t,x) 
      +\partial_x\varphi(t,x)v^2(t,x) \bigg) \,\RHO(t,dx) \,\d t\nonumber\\
  &     = -\int_0^1 \int_\R \varphi(t,x) \,f[\RHO(t,\cdot)](\d x) \,\d t,
\label{E:ARG}
\end{align}
which yields the momentum equation in \eqref{E:CONS2} in distributional sense.
An (even easier) analogous argument  holds for the continuity equation. This shows that
the pair $(\RHO, v)$ defined by \eqref{eq:38bis} and \eqref{E:VELOC} is a
solution of \eqref{E:CONS2}.

The first limit of \eqref{eq:64} 
follows since $\lim_{t\downarrow0}X(t)=\bar X$ in $\L^2(\Omega)$ and 
$\bar{X}_\#\leb =:
\bar{\RHO}$.
Concerning the second limit of \eqref{eq:64} 
we have to show that  
\begin{equation}\label{E:IDENTITY-VEL}
  \int_\Omega \varphi \bar v(x) \bar \RHO(\d x)=
  \lim_{ t  \rightarrow0} \int_\Omega \varphi v(t, x) \RHO(t, dx)
\end{equation}
for every $\varphi\in \C_b(\R)$.
Since $\bar V=\bar v\circ \bar X$
we have $\bar V\in \H_{\bar X}\subset
T_{\bar X}\K$ so that $\lim_{t\downarrow0}V(t)=\bar V$ in
$\L^2(\Omega)$
and therefore
$$
\begin{aligned}
  \int_\R \varphi(x) \bar v(x) \bar \RHO(\d x)&= \int_\Omega \varphi(\bar X) \bar
  V\, \d m= 
  \lim_{t \downarrow 0}
  \int \varphi( X(t)) V(t)\,\d m\\&=
  \lim_{t \downarrow 0} \int \varphi( X(t)) V^\ast(t)\,\d m= 
  \lim_{t \downarrow 0} \int \varphi(x) v(t,x) \RHO(t,dx),
\end{aligned}
$$
where we used the fact that $V(t)-V^\ast(t)$ is perpendicular to $\H_{X(t)}$.
\end{proof}
As we already observed in the previous proof, notice that 
\eqref{eq:63} surely holds if $\bar V\in \H_{\bar X}$.
\subsection{Existence, uniqueness, and stability of Lagrangian
  solutions for Lipschitz forces}
\label{S:EUSL}
Applying the general results of \cite{Brezis} is not difficult to
prove
\begin{theorem}
\label{T:EUS}
  Let us suppose that $F:\K\to \L^2(\Omega)$ is Lipschitz. Then for 
  every $(\bar X,\bar V)\in \K\times \L^2(\Omega)$ there exists a
  unique Lagrangian solution $X$ to \eqref{E:FODI} and
  for every $T\GS 0$ there exists a constant $C_T\GS 0$
  independent of the initial data such that
  for every $t\in [0,T]$
  \begin{equation}
    \label{eq:62}
    \|X(t)\|_{\L^2(\Omega)}+\|V(t)\|_{\L^2(\Omega)}\LS C_T\Big(1+
    \|\bar X\|_{\L^2(\Omega)}+\|\bar V\|_{\L^2(\Omega)}\Big).    
  \end{equation}
  Moreover, for any $T\GS 0$ there exists a constant
  $C_T\GS 0$ with the following property: For any pair of strong Lagrangian
  solutions $X_i$ with initial data $(\bar{X}_i, \bar{V}_i)$ for $i=1,2$ we
  have that for all $t\in[0,T]$ it holds
\begin{align}
  \|X_1(t)-X_2(t)\|_{\L^\p(\Omega)}
    &\LS C_T \Big( \|\bar{X}_1-\bar{X}_2\|_{\L^\p(\Omega)}
      + \|\bar{V}_1-\bar{V}_2\|_{\L^\p(\Omega)} \Big)
\label{E:XCONTR}\\
  \int_0^T \|V_1(t)-V_2(t)\|_{\L^2(\Omega)}^2 \,\d t
    &\LS C_T \sum_{i=1\ldots 2} \Big( \|\bar{X}_i\|_{\L^2(\Omega)}
        + \|\bar{V}_i\|_{\L^2(\Omega)} \Big)
\nonumber\\
    & \quad \times \Big( \|\bar{X}_1-\bar{X}_2\|_{\L^2(\Omega)}
        + \|\bar{V}_1-\bar{V}_2\|_{\L^2(\Omega)} \Big).
\label{E:VCONTR}
\end{align}
\end{theorem}
\begin{proof}
  Recalling the equivalent formulation \eqref{eq:26}, we introduce the Hilbert space $H:\L^2(\Omega)\times L^2(\Omega)$
  and the (multivalued) operator $A(X,Y):=(\partial I_\K(X)-Y,F[X])$. 
  It is easy to check that $A$ is a Lipschitz perturbation of the
  subdifferential of the proper, convex, and l.s.c.\ functional
  $\Phi(X,Y):=I_\K(X)$.
  Thus existence, uniqueness, and the estimates 
  \eqref{eq:62}, \eqref{E:XCONTR} follow by \cite[Theorem
  3.17]{Brezis}. 

  The same estimate also yields for the second component $Y_i(t)=\bar
  V_i+\int_0^tF[X_i(s)]\, \d s$
  \begin{equation}
    \label{eq:43}
    \|Y_1(t)-Y_2(t)\|_{\L^\p(\Omega)}
    \LS C_T \Big( \|\bar{X}_1-\bar{X}_2\|_{\L^\p(\Omega)}
      + \|\bar{V}_1-\bar{V}_2\|_{\L^\p(\Omega)} \Big)
  \end{equation}
  and, by the boundedness of $F$,
$$
\int_0^T \|\dot{Y}_i(t)\|^2_{\L^2(\Omega)} \,\d t
\LS C_T^2 \bigg( 1 + \|\bar{X}_i\|_{\L^\p(\Omega)}^2
+ \|\bar{V}_i\|^2_{\L^\p(\Omega)} \Big).
$$
Applying Theorem~2 in \cite{Savare96} to the first equation of
\eqref{eq:26} we get \eqref{E:VCONTR}.  
\end{proof}

A straightforward application of the previous Theorem shows that
Lagrangian solutions are stable if $F$ is Lipschitz: a
sequence of Lagrangian solutions with strongly converging initial data
converges to another Lagrangian solution.

\subsection{Sticky lagrangian
solutions and the semigroup property}
\label{S:SLSSP}
We consider here an important class of Lagrangian solutions.
\begin{definition}[Sticky lagrangian solutions]\label{def:sticky}
  We say that a Lagrangian solution $X$ is \emph{sticky}
  if
  \begin{equation}
    \text{for any $t_1 \LS t_2$ we have $\Omega_{X(t_1)} \subset
      \Omega_{X(t_2)}$.}
    \label{E:MODI1}
  \end{equation}
\end{definition}
By \eqref{E:MONO} and \eqref{E:INCLUSION-SPACE} any sticky Lagrangian
solution satisfies the monotonicity condition
\begin{equation}
  \label{eq:39bis}
  \partial I_\K(X(t_1))\subset \partial I_\K(X(t_2)),\quad
  \H_{X(t_2)}\subset \H_{X(t_1)}\qquad
  \text{for any $t_1 \LS t_2$}.
\end{equation}
The nice features of sticky Lagrangian solutions are summarized in the
next results.
\begin{proposition}[Projection formula]
\label{P:PROJECTION} 
If $X$ is a sticky Lagrangian solution then
\begin{equation}
  V(t) \in \H_{X(t)}
  \quad\text{for all times $t\GS 0$}
\label{E:VELOCITY}
\end{equation}
and it satisfies
\begin{gather}
  X(t) = \PROJ{\K}\bigg(\bar X+\int_0^t Y(s)\,\d s\bigg)=\PROJ{\K} \bigg( 
      \bar X + t\bar V + \int_{0}^{t} (t-s) F[X(s)] \,\d s \bigg),
\label{E:PROJREP01}\\
V(t) =\PROJ{\H_{X(t)}}\big(Y(t)\big)= \PROJ{\H_{X(t)}} \bigg( 
\bar V + \int_{0}^{t} F[X(s)] \,\d s \bigg) .
\label{E:PROJREP03}
\end{gather}
\end{proposition}
\begin{proof}
\eqref{E:VELOCITY} follows easily by \eqref{eq:40} and \eqref{eq:34},
thanks to the monotonicity property \eqref{eq:39bis}.
Equation \eqref{E:PROJREP03} then follows from \eqref{eq:26} (where $\dot
X$ is replaced by $V$)
and \eqref{E:ORTHO}.
  
In order to prove \eqref{E:PROJREP01}, we set for $s\GS 0$
\begin{equation}
  \label{eq:41}
  \Xi(s):=Y(s)-V(s)\in \partial I_\K(X(s))\subset \partial
  I_\K(X(t))\quad \text{if }t\GS s,
\end{equation}
and we integrate \eqref{eq:41} w.r.t. $s$ from $0$ to $t$ to obtain
\begin{align*}
  \int_0^t \Xi(s)\, \d s=\bar X-X(t)+
  \int_0^t Y(s)\, \d s\in \partial I_\K(X(t))
\end{align*}
Recalling \eqref{E:EQUIVA3} we get \eqref{E:PROJREP01}.
%
\end{proof}
\begin{lemma}[Concatenation property]
\label{P:CONCATENATION}
Let $X_1, X_2$ be Lagrangian solutions with
initial data $\bar X_1,\bar V_1$ and $\bar X_2,\bar V_2$
respectively and let us suppose that 
\begin{equation}
  \label{eq:66}
  \Omega_{\bar X_2}\subset \Omega_{X_2(t)}\quad\text{for every }t\ge0.
\end{equation}
If 
for some $\tau>0$
\begin{equation}
  \label{eq:45}
    \bar X_2=X_1(\tau),\quad 
    Y_1(\tau)-\bar V_2=\bar \Xi_2\in \partial I_\K(\bar X_2),
  \end{equation}
  then the curve 
  \begin{equation}\label{eq:68}
    \tilde X:=
  \begin{cases}
    X_1(t)&\text{if }0\LS t\LS \tau,\\
    X_2(t-\tau)&\text{if }t\GS \tau,
  \end{cases}
\end{equation}
is a Lagrangian solution with initial data $(\bar X_1,\bar V_1)$.
In particular, if $X_1,X_2$ are sticky Lagrangian solutions, then
$\tilde X$ is also sticky.
\end{lemma}
Notice that 
\begin{equation}
  \label{eq:71}
  \text{the choice}\quad \bar V_2:=V(\tau_1)
  \quad\text{always satisfies \eqref{eq:45}.}
\end{equation}
\begin{proof}
  It is easy to check that 
  \begin{displaymath}
    \tilde Y(t):=
    \begin{cases}
      Y_1(t)&\text{if }0\LS t\LS \tau,\\
      Y_1(\tau)-\bar V_2+Y_2(t-\tau)&\text{if }t\GS \tau,
    \end{cases}
  \end{displaymath}
  is Lipschitz continuous and satisfies $\frac \d{\d t}\tilde
  Y(t)=F[\tilde X(t)]$ a.e.\ in $(0,\infty)$.
  We have to check that $\tilde X$ satisfies the first differential
  inclusion
  of \eqref{eq:26} for $t\GS \tau$ w.r.t.\ $\tilde Y$.
  By definition of $\tilde X$ we have for $t\GS \tau$
  \begin{align*}
    \tilde Y(t)-\frac \d{\d t}\tilde
    X(t)&=Y_1(\tau)-\bar V_2+Y_2(t-\tau)-
    V_2(t-\tau)\\&=
    \bar \Xi_2+\Xi_2(t-\tau)\in \partial I_\K(X_2(t-\tau))=\partial
    I_\K(\tilde X(t))
  \end{align*}
  since by 
  \eqref{eq:45} and \eqref{eq:66}  $\bar \Xi_2\in \partial I_\K(X_2(t-\tau))$.
\end{proof}
It would not be difficult to show that Lagrangian solutions in general
do not satisfy the sticky property nor the semigroup property. 
If the force is sticking then the next property
shows that 
these properties
are strictly related.
\begin{theorem}[Semigroup property]
\label{P:SEMIGROUP}
If the force operator $F$ is Lipschitz and sticking
and 
\begin{equation}
  \label{eq:67}
  \Omega_{\bar X}\subset \Omega_{X(t)}\quad\text{for every Lagrangian
    solution $X$ starting from $(\bar X,\bar V)\in \K\times \H_{\bar X}$}
\end{equation}
 then
every Lagrangian solution
$X$ starting from $(\bar X,\bar V)\in \K\times \H_{\bar X}$ 
is sticky and satisfies the following semigroup property:
for every $\tau>0$ the curve $\tilde X(t):=X(t-\tau)$ 
is the unique Lagrangian solution with initial data 
$X(\tau),V(\tau)$.
  
  In particular, for all $t\GS t_1\ge0$ we have
  %
  %
  \begin{gather}
    X(t) = \PROJ{\K} \bigg( 
    X(t_1) + (t-t_1) V(t_1) + \int_{t_1}^{t} (t-s) F[X(s)] \,\d s \bigg) 
    \label{E:PROJREP1}\\
    V(t) = \PROJ{\H_{X(t)}} \bigg( 
  V(t_1) + \int_{t_1}^{t} F[X(s)] \,\d s \bigg) .
  \label{E:PROJREP3}
\end{gather}
%
\end{theorem}
\begin{proof}
Let $\calT^0\in [0,\infty)$ as in \eqref{eq:34} ($(0,\infty)\setminus
\calT^0$ is negligible). For every $\tau\in \calT^0$ consider the 
Lagrangian solution $X_2$ with initial datum $(X(\tau),V(\tau))$: by
the concatenation property (with the choice \eqref{eq:71})
the map $\tilde X$ defined as in
\eqref{eq:68} (with $X_1:=X$) is a Lagrangian solution and therefore
coincides with $X$, since $F$ is Lipschitz. \eqref{eq:67} yields that 
\begin{equation}
  \label{eq:69}
  \Omega_{X(\tau)}\subset \Omega_{X(t)}\quad \text{for every }0\LS \tau<
  t,\quad
  \tau\in \calT^0\cup \{0\}.
\end{equation}
Let us now fix $s>0$ and consider a sequence $h_n\downarrow0$ such
that 
\begin{displaymath}
  \frac 1{h_n}\big(X(s)-X(s-h_n)\big)\WEAK V_-\quad\text{in }L^2(\Omega).
\end{displaymath}
Since $T_{X(s)}\K$ is a closed convex cone, it is also weakly closed,
so that by its very definition definition we have $-V_-\in
T_{X(s)}\K$.

We set $\Xi(t):=Y(t)-\dot X(t)\in \partial I_\K(X(t))$ thanks to the
differential inclusion of \eqref{eq:26}; an integration in time from
$s-h_n$ to $s$ and \eqref{eq:69} yield
\begin{displaymath}
  \fint_{s-h_n}^s Y(r)\,\d r-\frac 1{h_n}\big(X(s)-X(s-h_n))=
  \fint_{s-h_n}^s \Xi(r)\,\d r\in \partial I_\K(X(s)).
\end{displaymath}
Passing to the limit as $n\to\infty$ we obtain
\begin{displaymath}
  \Xi_-:=Y(s)-V_-\in \partial I_\K(X(s))
\end{displaymath}
and therefore by \eqref{E:ORTHO} 
\begin{displaymath}
  \bar V:=\PROJ{\X_{X(s)}}(V_-)=\PROJ{\X_{X(s)}}(Y(s)).
\end{displaymath}
Since 
$$Y(s)-\bar V=Y-V_-+\big(-\bar V-(-V_-)\big)=\Xi_-+\big(-\bar
V-(-V_-))\in \partial I_\K(X(s))$$ 
by \eqref{eq:36} and the fact that $-V_-\in T_{X(s)}\K$,
we can apply the concatenation property as before, joining at the time $s$
the Lagrangian solution $X_1:=X$ with the Lagrangian solution $X_2$
arising from the initial data $\bar X:=X(s)$ and $\bar V$.
The uniqueness theorem shows that this map coincides with $X$ and
therefore \eqref{eq:66} yields
$\Omega_{X(s)}\subset \Omega_{X(t)}$ for every $t>s$.

In particular we have $V(t)\in \X_{X(t)}$ for every $t\ge0$ so
that a further application of the concatenation Lemma
\ref{P:CONCATENATION}
yields the semigroup property.
\eqref{E:PROJREP1} and \eqref{E:PROJREP3} follow then 
by the corresponding \eqref{E:PROJREP01} and \eqref{E:PROJREP03}
\end{proof}
We conclude this section with our main result conerning the existence
of sticky Lagrangian solution;
the proof will require a careful analysis of the discrete particle models
and therefore will be postponed at the end of Section \ref{S:DODP},
see Remark \ref{R:MONO}.
\begin{theorem}[Sticking forces yields sticky Lagrangian
  solutions]
\label{T:MAINSTICK}
  If the for\-ce operator $F$ is Lipschitz and sticking (according to
  Defintion \ref{D:NONSPLITTING}) then every Lagrangian solution 
  to \eqref{E:FODI} with $\bar X\in \K$ and $\bar V\in \H_{\bar X}$ is
  sticky.  
\end{theorem}
\begin{remark}
We have seen that the right-derivative $V$ of a sticky Lagrangian
solution
is right-continuous everywhere. It
is continuous for all $t\GS 0$ for which the function $t\mapsto
\|V(t)\|_{\L^2(\Omega)}$ (which represents the kinetic energy) is continuous;
see Proposition~3.3 in \cite{Brezis}. At such times the map $t\mapsto X(t)$ is
differentiable. We do not know whether 
the velocity is of bounded variation. But \eqref{E:PROJREP3} 
and \eqref{E:TANGENT}
show the
following statement: For any $t\GS 0$ let $V_- \in \L^2(\Omega)$ be any weak
accumulation point of $V(s)$ as $s\uparrow t$. Then $V(t) =
\PROJ{T_{X(t)}\K}(V_-)=\PROJ{\H_{X(t)}}(V-)$. This is the analogue of the impact law \eqref{E:IMPACT}
we discussed in the Introduction. It follows easily from
\eqref{E:PROJREP3}.
\end{remark}

\EEE
\subsection{Lagrangian solutions for continuous force
  fields}\label{SS:FPALS} 
The goal of this section is to extend the existence Theorem \ref{T:EUS}
to the case of (uniformly) continuous force operators.
\begin{theorem}\label{T:SCHAUDER} Suppose that  $F:\K\to \L^2(\Omega)$ satisfies the pointwise linear condition (\ref{D:LINEARBOUNDEDNESS}) 
and it is uniformly continuous according to 
(\ref{E:UNFO1}).
Then for every $(\bar X,\bar V) \in \K\times \L^2(\Omega)$ there exists 
a Lagrangian solution $(X,Y)$ of \eqref{eq:26}. \\
Moreover, for any $T\GS 0$ there exists a constant $C_{T} \GS 0$
such that any Lagrangian solution $X$ with velocity $V:=\frac {\d^+}{\d t}X$ satisfy
\begin{equation}
  \|X(t)\|_{\L^\p(\Omega)} + \|V(t)\|_{\L^\p(\Omega)}
    \LS C_{T} \Big( 1 + \|\bar{X}\|_{\L^\p(\Omega)} 
      +\|\bar{V}\|_{\L^\p(\Omega)} \Big)
\label{E:BOUNDLP}
\end{equation}
for all $t\in[0,T]$. 
If 
$\psi \colon \R
\longrightarrow [0,+\infty)$ is an integrand satisfying \eqref{EQ:130} and
\eqref{E:131} for some $q\GS 1$, then there exists a constant $C_{q,T} \GS
0$ such that 
\begin{equation}
  \Psi[X(t)] + \Psi[V(t)] \LS C_{q,T} \Big( 1 +
    \Psi[\bar{X}] + \Psi[\bar{V}] \Big)
\label{E:BOUNDPSI}
\end{equation}
for all $t\in[0,T]$, with functional $\Psi$ defined in \eqref{E:113}.

\end{theorem}
%
\begin{proof} 
It suffices to show that there exists a solution to \eqref{eq:26} in a
bounded interval $[0,T]$ with $T$ independent of the initial
condition.
We will choose 
$$
T:={1 \over 2C_\sfp} <1.
$$ 
where $C_\sfp$ is the constant 
of  (\ref{D:LINEARBOUNDEDNESS}),
which is not restrictive to assume greater than $1$.

We consider the following operators defined in $\C([0,T];\L^2(\Omega))$:
the first one $\OP_1$ maps $W\in \C([0,T];\L^2(\Omega))$ into $Y:=\OP_1(W)$
defined by
\begin{equation}
  \label{eq:46}
  Y(t):=\bar V+\int_0^t F[W(s)]\, \d s;
\end{equation}
the second one, $\OP_2$, maps $Y\in \C([0,T];L^2(\Omega))$ into the
solution $X=\OP_2(Y)$ of the differential inclusion
\begin{equation}
  \label{eq:47}
  \dot X+\partial I_\K(X)\ni Y,\quad X(0)=\bar X.
\end{equation}
Both of them are continuous, since 
\begin{equation}
  \label{eq:48}
  \|\OP_1(W_1)-\OP_2(W_2)\|_\infty\LS 
  T\omega \big(\|W_1-W_2\|_\infty\big)
\end{equation}
(where we denoted by $\|\cdot\|_\infty$ the usual $\sup$ norm in
$\C([0,T];\L^2(\Omega))$)
and
\begin{equation}
  \label{eq:49}
  \|\OP_2(Y_1)-\OP_2(Y_2)\|_\infty\LS T \|Y_1-Y_2\|_\infty.
\end{equation}
We want to show that $\OP:=\OP_2\circ \OP_1$ has a fixed point $X$,
which is 
a Lagrangian solution with initial data
$(\bar X,\bar V)$.
\EEE
We may use de la Vall\'ee Poussin Theorem and Lemma \ref{L:DOUBLE} to obtain $\psi \colon \R \longrightarrow
[0,\infty)$ satisfying \eqref{EQ:130}/\eqref{E:131} for some $q>2$ and
(using the notation (\ref{E:113}))
\begin{equation}
    \lim_{r\rightarrow\infty} \frac{\psi(r)}{r^\p} = \infty, \quad
    \Psi[\bar X]+ \Psi[\bar V] <\infty.
\label{EQ:128new}
\end{equation} %
Choose $m$ large enough so that  
\begin{equation} \label{E:CONDITIONonM}
\Bigl( ||\bar V||_2 +  2 C_\sfp (1+m+ ||\bar V||_2)\Bigr)T \leq m  ,
\end{equation}
and let 
$$
\K(\Psi)=\{ W \in \K \colon  ||W-\bar X||_{\L^2(\Omega)} \leq m,\;\; \Psi\Bigl( W-\bar X \Bigr) \leq D  \}
$$ 
where 
$$
D=T^{q-1}\Bigl[\Lambda(T) +  (12 C_\sfp T)^{q} T 
\Bigl( \psi(m)+ \Psi(\bar X) \Bigr)\Bigr].
$$
and
$$ 
\Lambda(t):= 2^{q-1}\Bigl( t \Psi(\bar V) + 
            ( 3 \EEE
T)^{q-1} C^q \psi(1){t^2 \over 2} \Bigr). 
$$ 
We eventually set  
$$
\begin{aligned}
  \CC=\Big\{X \in \C([0,T]; \L^2(\Omega)) \colon &X(t) \in \K(\Psi) \;
  \forall \, t \in [0,T] ,\\
  &\|X(s)-X(r)\|_{\L^2(\Omega)}\LS \frac mT
  |r-s| 
  \; \forall \,r,s\in [0,T]\Big\},
\end{aligned}
$$
which by Arzel\`a-Ascoli Theorem is a nonempty, compact, and convex
subset of $\C([0,T];\L^2(\Omega))$.
In light of the Schauder Fixed Point Theorem it suffices to show that 
$\OP$ maps $\CC$ into itself.

\EEE
Let $W \in \CC$ and set $Y=\OP_1(W), X=\OP_2(Y)=\OP(W)$   
with $V=\frac {\d^+}{\d t}X$.
We exploit 
Lemma \ref{L:VELOCITY1} and equation (\ref{E:CONDITIONonV1}) below
with Jensen's inequality to obtain
\begin{eqnarray}
\Psi\Bigl( X(t)-\bar X \Bigr) 
               &= & 
                \Psi\Bigl( \int_0^t V(s)ds \Bigr)  \nonumber\\
               &\leq & 
               T^{q-1}\Bigl[\Lambda(t) + ( 6 C_\sfp T^2)^{q} 
               \int_0^t \Psi(W(l))dl\Bigr] \nonumber\\ 
               &\leq & 
               T^{q-1}\Bigl[\Lambda(t) +  (12 C_\sfp T)^{q}  \int_0^t \Bigl( \Psi(W(l)-\bar X)+
               \Psi(\bar X) \Bigr) dl \Bigr]\nonumber\\ 
              &\leq & 
               T^{q-1}\Bigl[\Lambda(t) +  (12 C_\sfp T)^{q} T   
               \Bigl( \Psi(m)+
                 \Psi(\bar X) \Bigr) \Bigr]\leq D.  \label{E:CONDITIONonX3}
\end{eqnarray}
Using  $\psi(r)=r^\p$ in (\ref{E:CONDITIONonV-1}) we have 
$$
\|V(t)\|_2
\leq  \big\|\bar V + \int_0^t F[W(s)] ds\big\|_2 
 \leq  \| \bar V\|_\p + \int_0^t \big\|F[W(s)]\big\|_\p ds,
$$ 
where we have used the $\H_X$--Contraction property in Lemma
\ref{L:CONTRACTION2}.  We use that $F$ is 
also bounded 
 with constant $C_2=2C_\sfp$ 
and that $Y \in \CC$ to  conclude that 
\begin{equation} \label{E:CONDITIONonX4.5}
  \|V(t)\|_\p \leq  \| \bar V\|_\p + 2C_\sfp(1+|m|+ \|\bar X\|_\p) \leq {m \over T}.
\end{equation}
Thus,   for every $0\LS r\LS s\LS T$
$$
\|X(r)-X(s)\|_\p\LS \int_r^s \|V(s)\|\,\d s\LS \frac mT|r-s|,\quad
\|X(s)-\bar X\|_\p\LS m.
$$
These prove that $X \in \CC.$

Concerning the estimates \eqref{E:BOUNDLP} and \eqref{E:BOUNDPSI}   
we simply set $W=X$ in the next Lemma \ref{L:VELOCITY1} and apply Gronwall's lemma.
\end{proof}
%


We conclude this section with the uniform bounds for solutions to
differential inclusions invoked by the
previous fixed point argument. 
In view of the next applications, we state them in a slightly more
general form.

\begin{lemma}[A priori bounds]\label{L:VELOCITY1} 
Let $F\colon \K \longrightarrow\L^\p(\Omega)$ be pointwise linearly
bounded so that there exists $C_\sfp>0$ 
such that (\ref{D:LINEARBOUNDEDNESS}) holds. Let $\psi \colon \R \longrightarrow [0,+\infty)$ be an integrand satisfying \eqref{EQ:130} and
\eqref{E:131} for some $q\GS 1$. Let $X \in \LIP(0,T; \L^2(\Omega))$ and $Z,W \in \L^\infty(0,T; \L^2(\Omega))$ be such that $X(0)=\bar X$ and 
\begin{equation} \label{E:DIFFERENTIALX1}
  V(t)+ \partial I_\K(X(t)) \ni \bar V + \int_0^t Z(s)\,\d s \quad \forall t\; \in [0,T)
\end{equation}
where 
$$
V(t)= {\d^+X \over \d t}(t)\quad \hbox{and} \quad  Z(s)\prec  F[W(s)]. 
$$
Then for a.e. $t \in (0,T)$
\begin{equation}\label{E:CONDITIONonV-1}
 \Psi( V(t)) \leq \Psi \Bigl(\bar V + \int_0^t Z(s)ds\Bigr),
\end{equation}
\begin{equation}\label{E:CONDITIONonV0} 
 \Psi( V(t)) \leq 
2^{q-1}\biggl( \Psi(\bar V) + ( 3\EEE t)^{q-1} C_\sfp^q \Big(
\psi(1)t +  2\EEE
\int_0^t \Psi(W(s))ds \Big)
\biggr)
\end{equation} 
and 
\begin{equation}\label{E:BOUNDonX0}
\Psi(X(t)) \leq  2^{q-1} 
            \biggl( 
                   \Psi(\bar X) + 
                   T^q \Lambda(T) + ( 6 C_\sfp T^2)^{q}\EEE \int_0^t \Psi(W(s)) ds 
            \biggr) 
\end{equation}
\end{lemma}  
%
%
\begin{proof}
Recalling Theorem \ref{T:BASIC} and \eqref{eq:34},
equation (\ref{E:DIFFERENTIALX1}) yields
\begin{displaymath}
  V(t)= 
  P_{\H_{X(t)}}\Bigl(\bar V + \int_0^t
  Z(s)ds\Bigr)
\quad \text{for every $t \in \calT^0$},
\end{displaymath}
where $\calT^0$ has full measure in $(0,T)$.
Hence, by Lemma \ref{L:CONTRACTION2} 
$$
 \Psi( V(t))= \Psi \biggl( P_{\H_{X(t)}}\Bigl(\bar V + \int_0^t Z(s)ds\Bigr) \biggr) \leq 
\Psi \Bigl(\bar V + \int_0^t Z(s)ds\Bigr),
$$   
which proves (\ref{E:CONDITIONonV-1}).

Using (\ref{EQ:112}) and Jensen's inequality we obtain
\begin{equation}\label{E:LATE1}
 \Psi( V(t)) \leq 2^{q-1}\Bigl( \Psi(\bar V) + \Psi(\int_0^t Z(s)ds)\Bigr) \leq 
2^{q-1}\Bigl( \Psi(\bar V) + t^{q-1} \int_0^t \Psi(Z(s)) ds\Bigr).
\end{equation}
We use the fact that $F$ is linearly bounded, $\psi$ is even,  and $\psi(\|W\|_1)\LS \Psi(W)$ by Jensen's inequality, 
\EEE
to find to obtain that for all $s\GS 0$ 
\begin{equation}
 \Psi(Z(s))  \leq \Psi[F(W(s))] 
 \LS \Psi[C_\sfp (1+|W(s)|+\|W\|_1)]
 \LS 3^{q-1}  C_\sfp^q \Big( \psi(1) + 2 \Psi[W(s)] \Big).
\label{E:BOUNDPSIWIL1}
\end{equation}
The first inequality in Equation (\ref{E:BOUNDPSIWIL1}) was obtained via Lemma \ref{L:CONTRACTION2}. We combine Equations (\ref{E:LATE1}, \ref{E:BOUNDPSIWIL1}) to obtain Equation (\ref{E:CONDITIONonV0}). 
By Equation (\ref{E:CONDITIONonV0}) 
\begin{eqnarray}
\int_0^t  \Psi( V(s))ds 
        &\leq &  
                2^{q-1}\biggl( t \Psi(\bar V) + 
                                ( 3 T)^{q-1} C_\sfp^q \Big( \psi(1){t^2 \over 2}           
                                +  2 \int_0^t \int_0^s \Psi(W(l))dl \Big) 
                         \biggr)\nonumber\\ 
                         & = & \Lambda(t) +  2 ( 6  T)^{q-1} C_\sfp^q \int_0^t (t-l) \Psi(W(l))dl
         \nonumber\\
        & \leq & \Lambda(t) + ( 6 T C_\sfp)^{q}\int_0^t \Psi(W(l))dl
         \label{E:CONDITIONonV1}  
\end{eqnarray} 
where  
$$ 
\Lambda(t):= 2^{q-1}\Bigl( t \Psi(\bar V) + 
            ( 3 T)^{q-1} C_\sfp^q \psi(1){t^2 \over 2} \Bigr)
$$
We have  
$$
\Psi(X(t))= \Psi \Bigl(\bar X + \int_0^t V(s)ds \Bigr) \leq 
2^{q-1} \Bigl( \Psi(\bar X) + t^{q-1} \int_0^t \Psi(V(s))ds\Bigr), 
$$
where have used (\ref{EQ:112}) and then Jensen's inequality.
This, together with (\ref{E:CONDITIONonV1}) yields (\ref{E:BOUNDonX0}).
\end{proof}

\section{The semigroup property and generalized Lagrangian solutions}
\label{S:SGL}
We have seen that 
Lagrangian solution may fail to satisfy the semigroup property 
in the natural phase space for the variables $(X,V)$, $V=\dot X$
(stated in Proposition \eqref{P:SEMIGROUP} for sticky Lagrangian
solutions). In fact, the formulation given by the system \eqref{eq:26}
shows that the natural variables for the semigroup property are the
couple $(X,Y)$. 

This motivates an alternate notion of solution (still linked to
\eqref{E:CONS2})
which tries to recover
a mild semigroup property, at the price of loosing uniqueness 
with respect to initial data.

Recall that for any transport map $X\in\K$ the orthogonal projection
$\PROJ{\H_X}$ onto the closed subspace $\H_X \subset \L^2(\Omega)$ leaves the
given function unchanged in $\Omega\setminus \Omega_X$ and replaces it with its
average in every maximal interval $(\alpha, \beta) \subset\Omega_X$; see
\eqref{E:CONSTANT}.
As a consequence, the function $\PROJ{\H_X}(F[X])$ is
constant wherever $X$ is. 

\begin{definition}
  \label{D:GENERALIZED}
  A \emph{generalized solution} to \eqref{E:FODI} is a curve $X\in
\LIP_{\rm loc}([0,\infty);\K)$ such that 
\begin{enumerate}
\item{\em Differential inclusion:}
\begin{equation}
  \label{eq:28}
  \dot X(t)+\partial I_\K(X(t))\ni \bar V+\int_0^t Z(s)\,
  \d s
  \quad\text{for a.e.\ $t\in (0,\infty)$},
\end{equation}
for some map $Z\in L^\infty_{\rm loc}([0,\infty);\L^2(\Omega))$ with
\begin{equation}
  Z-F[X(t)]\in \H_{X(t)}^\perp
  \quad\text{and}\quad
  Z \prec F[X(t)]\quad \text{for a.e.\ $t\in (0,\infty)$}.
\label{E:DOMINO}
\end{equation}
\item {\em Semigroup property:}
For all $t\GS t_1\GS 0$ the right derivative $V=\frac {\d^+}{\d t}X$ satisfies
\begin{equation}
  V(t) + \partial I_\K(X(t)) \ni V(t_1) + \int_{t_1}^{t} Z(s) \,\d s,
\label{E:FODI2}
\end{equation}
\item {\em Projection formula:}
For all $t\GS t_1\GS 0$ 
\begin{equation}
  X(t_2) = \PROJ{\K} \bigg( 
      X(t_1) + (t_2-t_1) V(t_1) + \int_{t_1}^{t_2} (t_2-s) Z(s) \,\d s \bigg),
\label{E:PROJREP2}
\end{equation}
\end{enumerate}
\end{definition}
\EEE
Note that for generalized Lagrangian solutions the semigroup property
and the projection one \eqref{E:PROJREP2} 
are part of the definition, while for sticky Lagrangian
solutions \eqref{E:PROJREP1} and \eqref{E:PROJREP3} are consequences of the
monotonicity property \eqref{E:MODI1}. The obvious choice in \eqref{E:DOMINO} is
$Z(t) := F[X(t)]$ for all times $t\GS 0$, which also shows that any sticky
Lagrangian solution is a weak solution. 
\begin{remark}
  If one is ultimately interested only in the existence of
  solutions to the conservation law \eqref{E:CONS2}, for this purpose
  any $Z$ stisfying \eqref{E:DOMINO} is sufficient. In fact, we proved
  in Theorem \ref{T:BASIC} that if the force functional $F[X]$ is
  induced by an Eulerian force field $f[\RHO]$, so that \eqref{eq:39}
  holds whenever $X\in\K$ and $X_{\#}(\leb) = \RHO$, then any strong
  Lagrangian solution yields a solution of the conservation law
  \eqref{E:CONS2}. The same argument works for weak Lagrangian
  solutions. Because of \eqref{E:DOMINO} we have that
  $\PROJ{\H_{X(t)}}(Z(t)) = \PROJ{\H_{X(t)}}(F[X(t)])$. On the other
  hand, it holds
$$
\int_\Omega \varphi(X(m)) F[X](m) \,\d m = \int_\Omega \varphi(X(m))
\PROJ{\H_X}(F[X])(m) \,\d m
$$
for all $\varphi\in\D(\R)$, with a similar formula for $Z$ in place of
$F[X]$.  Then the argument on page \pageref{E:MKO} can be adapted to
prove the claim; see in particular \eqref{E:ARG}.
\end{remark} 
Since $X$ is everywhere right differentiable, we have $V(t)\in
T_{X(t)}\K$ for every $t\GS 0$, so that \eqref{E:FODI2} yields
\begin{equation}
  V(t) = \PROJ{T_{X(t)}\K} \bigg( 
      V(t_1) + \int_{t_1}^{t_2} F[X(s)] \,\d s \bigg)\quad\text{for
        every }t\GS t_1\ge0,
\label{E:PROJREPTAN}
\end{equation}
which also yields 
\begin{equation}
  V(t) = \PROJ{\H_{X(t)}} \bigg( 
      V(t_1) + \int_{t_1}^{t_2} F[X(s)] \,\d s \bigg)\quad\text{for
        almost every }t\GS t_1\ge0.
\label{E:PROJREP4}
\end{equation} 
It is immediate to check that any solution is also a generalized
solution, corresponding to the choice $H(t):=0$.
By introducing the new variable 
\begin{displaymath}
  Y(t):=\bar V+\int_0^t Z(s)\,\d  s 
\end{displaymath}
we easily see that \eqref{eq:28} is equivalent to the evolution
system
\begin{equation}
  \label{eq:26tris}
  \left\{
    \begin{aligned}
      \dot{X}(t) + \partial I_\K(X(t)) &\ni Y(t),\\
      \dot Y(t)&=Z(t),\\
      Z(t)-F[X(t)]&\in \H_{X(t)}^\perp,\\
      Z(t)&\prec F[X(t)]
    \end{aligned}
    \right.\quad \text{for $t\GS 0$,}\quad
    (X(0),Y(0))=(\bar X,\bar V),
\end{equation}
where $H\equiv0$ in the case of \eqref{E:FODI}.

\EEE



\subsection{Stability of generalized Lagrangian solutions}\label{SS:C}

In this section, we will prove a stability result 
for generalized Lagrangian solutions. Instead of relying on
a semigroup estimate, strong compactness now follows from an argument based on
Helly's theorem (recall Lemma \ref{L:COMPACTNESS}) and 
on the closure properties of the map $X \mapsto \PROJ{\H_X}(F[X])$ for
$X\in\K$.

\begin{lemma}\label{L:CLOSURE}
Consider $\{(X_n,Z_n,F_n)\} \subset \L^\p(\Omega)$ with 
$$
  X_n\in\K,
  \quad
  Z_n - F_n \in \H_{X_n}^\perp,
$$
If $X_n \longrightarrow X$ strongly and $(Z_n,F_n)\WEAK (Z,F)$ weakly in
$\L^\p(\Omega)$, then 
\begin{gather}
  Z - F\in \H_X^\perp,
\label{E:88}\\
  Z \prec F
  \quad\text{if $F_n\longrightarrow F$ strongly in $\L^1(\Omega)$ and}  \quad Z_n \prec F_n.
\label{E:105}
\end{gather}
Analogously, for any $T>0$ consider $\{(X_n,Z_n,F_n)\} \subset
\L^\p((0,T),\L^\p(\Omega))$ with
$$
  X_n(t)\in\K,
  \quad
  Z_n(t) - F_n(t) \in \H_{X_n(t)}^\perp,
  \quad
  Z_n(t) \prec F_n(t)
  \quad\text{for a.e.\ $t\in(0,T)$.}
$$
If $X_n\longrightarrow X$ strongly and $(Z_n,F_n)\WEAK (Z,F)$ weakly
in $\L^2((0,T),\L^2(\Omega))$, then
\begin{gather*}
  Z(t) - F(t)\in \H_{X(t)}^\perp
  \quad\text{a.e.,}
\\
  Z(t)\prec F(t)
  \quad\text{a.e.\quad if $F_n\longrightarrow F$ strongly in 
    $\L^1((0,T),\L^1(\Omega))$},\ 
  Z_n(t)\prec F_n(t)\quad\text{a.e.}
\end{gather*}
\end{lemma}

\begin{proof}
By assumption, we know that
\begin{equation}
  \int_\Omega Z_n(m) \varphi(X_n(m)) \,\d m 
    = \int_\Omega F_n(m) \varphi(X_n(m)) \,\d m
\label{E:83}
\end{equation}
for every $\varphi\in\C_b(\R)$. Passing to the limit in \eqref{E:83} we get
$$
  \int_\Omega Z(m) \varphi(X(m)) \,\d m 
    = \int_\Omega F(m) \varphi(X(m)) \,\d m,
$$
which yields \eqref{E:88} since the set $\{\varphi\circ X \colon \varphi\in
\C_b(\R)\}$ is dense in $\H_X$.

In order to prove \eqref{E:105} we pass to the limit in the inequality 
$$
  \int_\Omega \psi(Z_n(m)) \,\d m \LS \int_\Omega \psi(F_n(m)) \,\d m
$$
for arbitrary convex functions $\psi\colon\R\longrightarrow\R$ with linear
growth, noticing that
\begin{equation}
\begin{aligned}
  \int_\Omega \psi(Z(m)) \,\d m 
    &\LS \liminf_{n\rightarrow\infty} \int_\Omega \psi(Z_n(m)) \,\d m,
\\
  \int_\Omega \psi(F(m)) \,\d m
    &= \lim_{n\rightarrow\infty} \int_\Omega \psi(F_n(m)) \,\d m.
\end{aligned}
\label{E:110}
\end{equation}
The corresponding inequality for convex functions $\psi$ with arbitrary growth
at infinity can be obtained from \eqref{E:110} by monotone approximation.

The time-dependent result follows by applying Ioffe's Theorem.
\end{proof}

\begin{theorem}[Stability of Generalized Lagrangian Solutions]\label{T:STABCOMP}
Suppose that $F\colon \K \longrightarrow \L^\p(\Omega)$ is
pointwise linearly bounded and uniformly continuous. Consider a sequence
$\{X_n\}$ of weak Lagrangian solutions with initial data
$$
  \bar{X}_n \in \K 
  \quad\text{and}\quad
  \bar{V}_n \in \H_{\bar{X}_n} 
$$
that converges strongly in $\L^\p(\Omega)$ to $\bar{X} \in \K$
and $\bar{V} \in \H_{\bar{X}}$. 
Then there exists a subsequence
(still denoted by $\{X_n\}$) with the following properties:
\begin{enumerate}
\item We have $X_n(t) \longrightarrow X(t)$ in $\L^2(\Omega)$
uniformly on compact time intervals.
\item For any $T>0$ we have $V_n \longrightarrow V$ in
$\L^2((0,T),\L^2(\Omega))$.
\item The limit function $X$ is a generalized Lagrangian solution.
\end{enumerate}
\end{theorem}

\begin{proof}
Since $(\bar{X}_n,\bar{V}_n) \longrightarrow (\bar{X},\bar{V})$ strongly
in $\L^\p(\Omega)$ we can find a convex function $\psi$ satisfying \eqref{EQ:130}
and $\lim_{r\rightarrow\infty} \psi(r)/r^\p = \infty$ such that
$$
  [\bar{X}_n] + \Psi[\bar{V}_n] \LS C
  \quad\text{for all $n$.}
$$
Here $\Psi$ denotes the functional \eqref{E:113} induced by $\psi$. By Lemma
\ref{L:DOUBLE}, it is not restric\-tive to assume that $\psi$ satisfies
\eqref{E:131}. The estimates of  Lemma \ref{L:VELOCITY1} (with $W:=X$) and Gronwall lemma yields 
\begin{equation}\label{E:BOUND.ON.X_n}
  \Psi[X_n(t)] + \Psi[V_n(t)] \LS C_T
  \quad\text{for all $t\in[0,T]$ and all $n$.}
\end{equation}
By Lemma~\ref{L:COMPACTNESS} it then follows that the $X_n$ take values in a
fixed compact subset of $\L^\p(\Omega)$ and $X_n$ are uniformly
Lipschitz continuous in $\L^\p(\Omega)$. 
Recall that pointwise linearly bounded operators $F$ are also bounded. We
can then apply Ascoli-Arzel\`a theorem to obtain a convergent subsequence, which
we still denote by $\{X_n\}$ for simplicity. The convergence is uniform in each
compact time interval and the limit function $X$ satisfies the same Lipschitz
bound.

Consider now the sequence $\{Z_n\}$ of functions given by
Definition~\ref{D:GENERALIZED}. Since $Z_n(t) \prec F[X_n(t)]$ for a.e.\ $t$ and
since $F$ is bounded, \eqref{E:BOUND.ON.X_n} implies that the $Z_n$ are
uniformly bounded in $\L^\infty((0,T),\L^\p(\Omega))$ for all $T>0$. Extracting
another subsequence if necessary, we may therefore assume that
$$
  Z_n \WEAK Z 
  \quad\text{weak* in $\L^\infty((0,T),\L^\p(\Omega))$.}
$$
On the other hand, by uniform continuity of $F$ we have that
$$
  F_n := F[X_n] \longrightarrow F[X] =: F
  \quad\text{strongly in $\L^\p((0,T),\L^\p(\Omega))$.}
$$
Lemma~\ref{L:CLOSURE} then shows that $Z$ satisfies \eqref{E:DOMINO}. 

The uniform bound on $Z_n$ implies that the maps
$$
  Y_n(t) := \bar{V}_n + \int_0^t Z_n(s) \,\d s
  \quad\text{for all $t\GS 0$}
$$
are uniformly Lipschitz continuous in each time interval $[0,T]$ with values in
$\L^\p(\Omega)$. Starting from \eqref{eq:28} and applying
standard stability results for differential inclusions (cf.\ Theorem~3.4 in
\cite{Brezis}, here the strong convergence of $X_n$ is crucial), we obtain that $X$ solves
\begin{equation}
  \dot{X}(t) + \partial I_\K(X(t)) \ni \bar{V} + \int_0^t Z(s) \,\d s
  \quad\text{for a.e.\ $t\GS 0$.}
\label{E:FODI3}
\end{equation}
In particular, the map $X$ is right-differentiable in $\L^2(\Omega)$ for each
$t\GS 0$, with right-continuous right-derivative $V$; see Propositions 3.3 and
3.4 in \cite{Brezis}. Therefore \eqref{E:FODI3} holds for all $t\GS 0$ if
$\dot{X}(t)$ is replaced by $V(t)$. We may also assume that
$$
  V_n \WEAK V
  \quad\text{weak* in $\L^\infty((0,T),\L^\p(\Omega))$}
$$
for all $T>0$ (extracting another subsequence if necessary). To
show that $V_n\longrightarrow V$ strongly in $\L^2((0,T),\L^2(\Omega))$, we
multiply the differential inclusion by
$$
  \frac{\d}{\d t} \Big( (T-t)X_n(t) \Big) = (T-t)V_n(t) - X_n(t)
$$
and integrate in time over $(0,T)\times\Omega$. Now notice that since $X_n(t),
V_n(t) \in \H_{X_n(t)}$ and since $\partial I_\K(X) \subset \H_X^\perp$
for all $X\in\K$, the subdifferential terms vanish after
integration over $\Omega$. Integrating by parts in the force term, we obtain
\begin{align*}
  & \int_0^T (T-t) \|V_n(t)\|_{\L^2(\Omega)}^2 \,\d t
    -\int_0^T \bigg( \int_\Omega V_n(t,m) X_n(t,m) \,\d m \bigg) \,\d t
\\
  &\quad = -T \int_\Omega \bar{V}_n(m) \bar{X}_n(m) \,\d m
  -  \int_0^T (T-t) \bigg( \int_\Omega Z_n(t,m) X_n(t,m) \,\d m \bigg) \,\d t.
\end{align*}
A similar identity holds in the limit. Since the sequence $\{X_n\}$ converges
strongly and the sequence $\{(V_n,Z_n)\}$ converges weakly, we can pass to the
limit and get
$$
  \lim_{n\rightarrow\infty} \int_0^T (T-t) \|V_n(t)\|_{\L^2(\Omega)}^2 \,\d t
    = \int_0^T (T-t) \|V(t)\|_{\L^2(\Omega)}^2 \,\d t
$$
for every $T>0$. This, together with (\ref{E:BOUND.ON.X_n}) yields the desired strong convergence. 
Therefore there exists an 
$\LEB^1$-negligible set $N\subset[0,\infty)$ such that (up to extraction of a
subsequence if necessary) $V_n(t)\longrightarrow V(t)$ in $\L^2(\Omega)$ for
every $t\in[0,\infty)\setminus N$. We can then pass to the limit in
 \eqref{E:FODI2} written for $(X_n,V_n)$ and obtain the corresponding inclusion
 for $(X,V)$ in $(t_1,\infty)$  for all $t_1\in [0,\infty)\setminus N$.
 Since $V$ is right-continuous, formula \eqref{E:FODI2} eventually holds
 for all $t_1\GS 0$.
 
 \eqref{E:PROJREP2} follows by the same argument, first passing to the
 limit for $t_1\in [0,\infty)\setminus N$ and recalling that 
 by \eqref{E:EQUIVA2} if $K_n=\PROJ{\K}(H_n)$ and $K_n\longrightarrow
 K$, $H_n\WEAK H$ in $\L^2(\Omega)$ then $K=\PROJ{\K}(H)$.
\end{proof}
We conclude this section with 
the main existence result for generalized Lagrangian solutions. As for
sticky 
evolutions, its proof relies on the discrete particle approach we will study in
the next section, see Remark \ref{R:EXGENERALIZED}.
\begin{theorem}[Existence of generalized Lagrangian solutions]\label{T:EXLAGR}
Let us assume that the force functional $F\colon \K \longrightarrow
\L^\p(\Omega)$ is pointwise linearly bounded and uniformly
continuous. Then
for every couple $
\bar{X} \in \K$ and $
 \bar{V} \in \H_{\bar{X}}
$
there exists a generalized Lagrangian solution with initial data $(\bar{X},\bar{V})$. 
\end{theorem}

\EEE

\section{Dynamics of Discrete Particles}\label{S:DODP}

We discussed in the Introduction that the conservation law \eqref{E:CONS2}
formally admits particular solutions for which the density consists of finite
linear combinations of Dirac measures; see \eqref{E:PART} above. In this
section, we will reformulate these solutions in the Lagrangian framework and
will prove their global existence. In fact, they are Lagrangian solutions
in the sense of Definitions~\ref{D:STRONG} and \ref{D:GENERALIZED}.

For every $N\in\N$ let us introduce the convex sets
\begin{gather*}
  \MM^N := \Big\{ \mm\in\R^N \colon \text{$m_i>0$ and $\sum_{i=1}^Nm_i=1$} \Big\},
\\
  \KK^N := \Big\{ \xx\in\R^N \colon x_1\LS x_2\LS \ldots \LS x_N \Big\}.
\end{gather*}
For all times $t\GS 0$, a discrete solution to \eqref{E:CONS2} of the form
\eqref{E:PART} is therefore determined by a unique number $N\in\N$ and a vector
$(\mm,\xx,\vv) \in \MM^N \times \KK^N \times \R^N$. To find a Lagrangian
representation of \eqref{E:PART} we consider a partition of $\Omega$ given
by
\begin{equation}
  0 =: w_0 < w_1 < \ldots < w_N := 1
  \quad\text{where}\quad
  w_i := \sum_{j=1}^i m_j\label{eq:50}
\end{equation}
for $i=1,\ldots,N-1$. Writing $W_i := [w_{i-1},w_i)$ we define functions
\begin{equation}
  X := \sum_{i=1}^N x_i\CHAR_{W_i}
  \quad\text{and}\quad
  V := \sum_{i=1}^N v_i\CHAR_{W_i},
\label{E:DEFXV}
\end{equation}
the (finite dimensional) Hilbert space
\begin{equation}
  \label{eq:52}
  \H_\smm:=\Big\{X = \sum_{i=1}^N x_i\CHAR_{W_i}\colon 
  \xx=(x_1,\cdots,x_N)\in \R^N\Big\}\subset \L^2(\Omega)
\end{equation}
and its closed convex cone
\begin{equation}
  \label{eq:51}
  \K_\smm:=\Big\{X = \sum_{i=1}^N x_i\CHAR_{W_i}\colon
  \xx=(x_1,\cdots,x_N)\in \KK^N\Big\} \subset \K\subset \L^2(\Omega)
\end{equation}
Then clearly $X\in\K_\smm\subset \K$ and $V\in\H_X$, and
we easily have
\begin{equation}
  \label{eq:53}
  \RHO=X_\#\leb=\sum_{i=1}^N m_i\,\delta_{x_i},\quad 
  V=v\circ X,\quad
  (\RHO v)=\sum_{i=1}^N m_i v_i \,\delta_{x_i}.
\end{equation}

\subsection{Discrete Lagrangian solutions}
\label{S:DLS}
We can reproduce at the discrete level the same approach we followed 
in Section \ref{S:SR}: we can introduce the projected forces
\begin{equation}
  \label{eq:54}
  F_\smm[X]:=\PROJ{\H_\smm}(F[X])=\sum_{i=1}^Na_{\smm,i}\CHAR_{W_i},\quad
  a_{\smm,i}=\fint_{W_i} F[X(t)](m) \,\d m,
\end{equation}
which satisfies the analogous of \eqref{eq:39}
\begin{equation}
  \label{eq:44}
  \int_\R \psi(x)f[\RHO](\d x)=\int_\Omega \psi(X)F_\smm\,\d  m,\quad
  \text{if }X\in \H_\smm,\ \RHO=X_\#\leb\text{ as in \eqref{eq:53},}
\end{equation}
and we can simply solve the differential inclusion
\begin{equation}
  \label{eq:42}
  \dot X(t)+\partial I_{\K_\smm}(X)=\bar V+\int_0^t F_{\smm}[ X(s)]\,
  ds,\quad X(0)=\bar X
\end{equation}
for given initial data $(\bar X,\bar V)\in \K_\smm\times \H_{\bar X}$.
Introducing $Y(t):=\bar V+\int_0^t F_{\smm}[ X(s)]\,
  ds=\sum_{i=1}^N y_i(t)\CHAR_{W_i}$,
we end up with the system
\begin{equation}
  \label{eq:26bis}
  \left\{
    \begin{aligned}
      \dot{X}(t) + \partial I_{\K_\smm}(X(t)) &\ni Y(t),\\
      \dot Y(t)&=F_\smm[X(t)],
    \end{aligned}
    \right.\quad \text{for $t\GS 0$,}\quad
    (X(0),Y(0))=(\bar X,\bar V),
\end{equation}
which is equivalent to \eqref{eq:20}.

If, e.g., $F$ is Lipschitz, then $F_\smm:\K_\smm\to \H_\smm$ is also
Lipschitz and 
the analogous statements of Theorems \ref{T:BASIC} and 
\ref{T:EUS}
hold at this discrete level. In particular, as in \eqref{E:TANGENT}, we have 
\begin{equation}
  \label{eq:82}
  V(t)=\frac{\d^+}{\d t}X(t)=\PROJ{T_{X(t)}\K_\smm}(Y(t));
\end{equation}
the discrete analog of Lemma \ref{L:TANGENT} thus justifies condition 
\eqref{eq:10} we introduced in the simplified situation of a collision 
of two particles. 

Let us now consider a sequence $X_n$ of discrete Lagrangian solutions of
\eqref{eq:42}
corresponding to initial data $(\bar X_n,\bar V_n)\in
\K_{\smm_n}\times \H_{\smm_n}$ strongly converging to $(\bar X,\bar
V)\in \K\times \L^2(\Omega)$. 
We want to show that $X_n\longrightarrow X$ locally uniformly in
$\C([0,\infty);\L^2(\Omega))$
where $X$ is the Lagrangian solution associated to $(\bar X,\bar V)$.
To make the analysis simpler, we will assume that the distributions of
masses $\mm_n$ give raise by \eqref{eq:50} to suffiently fine partitions of the
interval $(0,1)$, i.e.
\begin{equation}
  \label{eq:55}
  \text{for every }K\in \K\text{ there exist $K_n\in
    \K_{\smm_n}$ such that }K_n\longrightarrow K\quad\text{in }\L^2(\Omega).
\end{equation}
Since $\K_{\smm_n}\subset \K$, \eqref{eq:55} is equivalent to say that 
the sequence $\K_{\smm_n}$ Mosco-converge to $\K$ in the Hilbert space
$\L^2(\Omega)$
\cite[Section 3.3.2]{Attouch84}. 
By first approximating $C^1([0,1])$ functions (which belong to
$\K-\K$) and then applying a density argument, it is not difficult to show that
\eqref{eq:55}
implies a similar property for the closed subspaces $\H_{\smm_n}$ in $\L^2(\Omega)$,
i.e.
\begin{equation}
  \label{eq:55bis}
  \text{for every }H\in \L^2(\Omega)\text{ there exist $H_n\in
    \H_{\smm_n}$ such that }H_n\longrightarrow H\quad\text{in }\L^2(\Omega).
\end{equation}
Both \eqref{eq:55} and \eqref{eq:55bis} surely holds if, e.g., 
$$\lim_{n\to \infty} \|\mm_n\|_\infty=0,$$
where for a generic $\mm=(m_1,\cdots,m_N)\in \MM^N$ we set 
$\|\mm\|_\infty=\sup_i m_i$.

It is not surprising that we have the following approximation result:
\begin{theorem}[Convergence of discrete Lagrangian solutions]
  Let $F:\K\to \L^2(\Omega)$ be Lipschitz 
  and pointwise linearly bounded, and let $\mm_n\in \MM^{N_n}$ be a sequence satisfying \eqref{eq:55}
  and let $X_n\in \LIP_{\rm loc}([0,\infty);\K_{\smm_n})$ of discrete
  Lagrangian solutions corresponding to the initial data $(\bar X_n,\bar V_n)\in
  \K_{\smm_n}\times \H_{\smm_n}$ strongly converging to $(\bar X,\bar
  V)\in \K\times \L^2(\Omega).$ 
  Then $X_n\longrightarrow X$ locally uniformly in
  $\C([0,\infty);\L^2(\Omega))$
  where $X$ is the unique Lagrangian solution starting from $(\bar
  X,\bar V)$.
\end{theorem}
\begin{proof}
  We cannot directly apply the stability estimates of Theorem
  \ref{T:EUS}, since the discrete Lagrangian solutions are associated
  to
  convex sets $\K_{\smm_n}$ depending on $n$, so we combine the
  compactness argument of the proof of Theorem \ref{T:STABCOMP}
  and a classical stability result for differential inclusion
  \cite[Theorem 3.74]{Attouch84}
  generated
  by a Mosco-converging sequence of convex sets.
  
  In fact, we can choose a convex and superquadratic functional $\psi$
  satisfying \eqref{EQ:130} such that 
  \begin{displaymath}
    \Psi[\bar X_n]+\Psi[\bar V_n]\LS C;
  \end{displaymath}
  the estimates of Lemma \ref{L:VELOCITY1} 
  (which can be extended to the discrete case) yield
  \begin{displaymath}
    \Psi[X_n(t)] + \Psi[V_n(t)] \LS C_T
    \quad\text{for all $t\in[0,T]$ and all $n$.}
  \end{displaymath}
  Arguing as in the proof of Theorem \ref{T:STABCOMP} we can find a
  subsequence (still denoted by $X_n$) locally uniformly converging to
  a limit $X\in \LIP_{\rm loc}([0,\infty);\L^2(\Omega))$ which takes
  its value in $\K$. 
  We easily get that 
  $F_{\smm_n}[X_n]\longrightarrow F[X]$ in 
  $L^2_{\rm loc}([0,\infty);\L^2(\Omega))$ since for every time $t\ge0$
  \begin{displaymath}
    \big\|F_{\smm_n}[X_n]-F[X]\big\|_{\L^2(\Omega)}\LS C
    \big\|X_n-X\big\|_{\L^2(\Omega)}+\big\|F_{\smm_n}[X]-F[X]\big\|_{\L^2(\Omega)}
  \end{displaymath}
  and $F_{\smm_n}[X]=\PROJ{\H_{\smm_n}}\big(F[X]\big)\longrightarrow
  F[X]$ in $\L^2(\Omega)$ by \eqref{eq:55bis}.
  
  It follows that 
  $$Y_n(t)=\bar V_n+\int_0^t
  F_{\smm_n}[X_n(s)]\,\d s\longrightarrow Y(t):=\bar V+\int_0^t F[X(s)]\, \d s$$
  locally uniformly in $\C([0,\infty);\L^2(\Omega))$.
  We can then apply \cite[Theorem 3.74]{Attouch84} to show that 
  the limit $X$ also satisfies the differential inclusion 
  \begin{displaymath}
    \dot X+\partial I_\K(X)\ni Y
  \end{displaymath}
  and therefore it is a Lagrangian solution associated to $(\bar
  X,\bar V)$. Since the limit is uniquely determined (by Theorem 
  \ref{T:EUS}) we conclude that the whole sequence $X_n$ converges to $X$.
\end{proof}

\subsection{A sticky evolution dynamic for discrete particles}
\label{S:SED}
In this section we will describe
a different discrete procedure to construct evolution of a finite number of particles.
In the general case, this approach will lead to
generalized Lagrangian solutions; 
when $F$ is sticking, we will obtain a sticky
evolution which in
in fact will coincide with the construction we considered
in the previous section.

We already explained the basic idea in the introduction: 
at the discrete level, a collision 
between two or more particles at some time $t'$ corresponds to the impact of the vector
$\xx$ with the boundary $\partial \KK^N$ 
(equivalently, of the
Lagrangian parametrization $X$ with the boundary of $\K_{\smm}$ in
$\H_{\smm}$): in this case, we relabel the particles and consider the
evolution for $t\GS t'$ in a reduced convex cone attached to the new
configuration
up to the next collision.

In order to get a precise description of the evolution, \EEE
let us observe that the boundary $\partial\KK^N$ of the cone $\KK^N$ 
in $\R^N$ consists of vectors whose
components are not all distinct. For any $\xx\in\partial\KK^N$ we define $I_i :=
\big\{ k \colon x_k=x_i \}$ for all $i=1,\ldots,N$. Then there exists a 
interger $N'<N$ and an increasing map 
$$
  \sigma \colon \{ 1,\cdots,N' \} \longrightarrow \{ 1,\cdots,N \}
$$
with the property that $\sigma(j) = \min I_{\sigma(j)}$ for all $j=1,\ldots,N'$.
We set 
\begin{equation}
  m'_j := \sum_{i\in I_{\sigma(j)}} m_i,
  \quad
  x'_j := x_{\sigma(j)},
  \quad\text{and}\quad
  m'_j v'_j := \sum_{i\in I_{\sigma(j)}} m_i v_i
\label{E:71}
\end{equation}
for all $j$ and obtain a new state vector $(\mm',\xx',\vv') \in \MM^{N'}\times\KK^{N'}
\times\R^{N'}$. In terms of the corresponding functions $X'\in\K$ and $V'\in
\H_{X'}$ (defined as in \eqref{E:DEFXV}), this means that $X'\in
\K_{\smm'}$, $V'\in \H_{X'}\subset \H_{\smm'}$ and 
\begin{equation}
  X' = X,
  \quad
  V' = \PROJ{\H_{X'}}(V).
\label{E:72}
\end{equation}
Starting from this remark, we can now
introduce the precise evolution algorithm for the Lagrangian
parametrization $X$.
\EEE
Assume without loss of generality that $\bar X$ does not belong to
the boundary of $\K_{\bar \smm}$ in $\H_{\bar \smm}$.
We construct a map $t\mapsto X(t)\in\K$ as follows: On the time interval $[t_0,
t_1)$, where $t_0:=0$ and $t_1>0$ is to be determined later so that $X(t)$ does not touch the boundary of $\partial\K_{\bar \smm}$
in $[t_0,t_1)$, 
\EEE
we obtain functions
$$
  X(t,\cdot) = \sum_{i=1}^N x_i(t) \CHAR_{W_i}
  \quad\text{and}\quad
  V(t,\cdot) = \sum_{i=1}^N v_i(t) \CHAR_{W_i}
$$
by solving the system
\begin{equation}
  \dot{X}(t) = V(t),
  \quad
  \dot{V}(t) = \PROJ{\H_{\bar \smm}}(F[X(t)]).
\label{E:SYSTEM0}
\end{equation}
Since $\H_{\bar \smm}=\H_{X(t)}$ in $[t_0,t_1)$, 
\EEE
we notice that the projection onto $\H_{\bar\smm}$ returns a function that is piecewise
constant on the same partition on which $(X,V)$ is constant. More precisely, we
find
$$
  \PROJ{\H_{\bar \smm}}(F[X(t)]) = \sum_{i=1}^N a_i(t) \CHAR_{W_i}
  \quad\text{where}\quad
  a_i(t) := \fint_{W_i} F[X(t)](m) \,\d m
$$
for $i=1,\ldots,N$. Hence \eqref{E:SYSTEM0} is equivalent to the system
\begin{equation}
  \label{eq:56}
  \dot{\xx}(t) = \vv(t),\quad \dot{\vv}(t) = \aa(t)\quad\text{for all $t\in[t_0,t_1)$,}
\end{equation}
which is
well-defined. The time $t_1$ is taken as the smallest $t>0$ for which
$X(t)$ hits the boundary of $\K_{\bar \smm}$ in $\H_{\bar\smm}$. 
As explained above, at time $t_1$ we can
find an integer $N'<N$ and compute a new state vector $(\bar\mm',\bar\xx',\bar\vv') \in
\M^{N'}\times\K^{N'}\times\R^{N'}$ by \eqref{E:71}. On the interval
$[t_1,t_2)$, with $t_2>t_1$ to be determined, we obtain 
$$
  X(t,\cdot) = \sum_{j=1}^{N'} x_j'(t) \CHAR_{W'_j}
  \quad\text{and}\quad
  V(t,\cdot) = \sum_{j=1}^{N'} v_j'(t) \CHAR_{W'_j}
$$
by solving \eqref{E:SYSTEM0} and \eqref{eq:56} with $\bar \mm$ replaced by
$\bar \mm'$, the initial condition $(\xx',\vv')(t_1) := (\bar
\xx',\bar \vv')$, and the new subdivision
$W'_j := [w'_{j-1},w'_j)$ defined by
$$
  0 =: w'_0 < w'_1 < \ldots < w'_{N'} := 1
  \quad\text{where}\quad
  w'_j := \sum_{k=1}^j m'_k
\quad\text{for all $j=1,\ldots,N'$.}
$$
Again the problem reduces to solving a finite
dimensional ordinary differential equation and the time $t_2$ is taken to be
the smallest $t>t_1$ for which $X(t)$ is in the boundary of $\K^{N'}$. Then we
continue in the same fashion.

We obtain an integer $K\in\N$, a sequence of ``collision times''
$$
  0 =: t_0 < t_1 < \ldots < t_{K-1} < t_K := \infty,
$$
and a pair of functions $(X,V)$ such that
\begin{equation}
  \H_{X(t)} = \H_{X(t_{k-1})},\quad
  \dot{X}(t) = V(t),
  \quad
  \dot{V}(t) = \PROJ{\H_{X(t)}}(F[X(t)])
\label{E:ODE}
\end{equation}
for all $t\in[t_{k-1}, t_k)$ and $k=1,\ldots,K$. At collision times the space
$\H_{X(t_k)}$ is strictly smaller than $\H_{X(t)}$ for all $t<t_k$, which
implies that $K\LS N$. We have
\begin{equation}
  X(t_{k}+) = X(t_{k}-),
  \quad
  V(t_{k}+) = \PROJ{\H_{X(t_k)}}(V(t_{k}-)).
\label{E:OPK}
\end{equation}
It is easy to check that the monotonicity condition \eqref{E:MODI1} is
satisfied.

\subsection{Sticky and generalized Lagrangian solutions for discrete particles}

The next Theorem shows that 
by the algorithm described in the previous section
we will obtain a generalized Lagrangian solution
in the original cone $\K$ starting from the discrete data $(\bar
X,\bar V)$;
when $F$ is sticking, this coincides with the unique sticky
Lagrangian solution.
\EEE
\begin{theorem}[Generalized and sticky Lagrangian solutions for
  discrete particles]\label{T:EXDISCRETE}
\ \\Suppose that $F\colon\K \longrightarrow \L^2(\Omega)$ is
uniformly continuous. Consider functions $(\bar{X},
\bar{V})$ of the form \eqref{E:DEFXV} for some $N\in\N$ and $(\bar{\mm},\bar{\xx},
\bar{\vv}) \in \MM^N \times\KK^N\times\R^N$. 
\begin{enumerate}
\item The
curve $(X,V)$ described by the previous section is a generalized Lagrangian
solution to \eqref{E:FODI} with initial data $(\bar{X},\bar{V})$.
\item If $F$ is
sticking, then $(X,V)$ is a sticky Lagrangian solution.
\end{enumerate}
\end{theorem}

\begin{proof} 
Let us first prove that the map $t\mapsto X(t)$ is a generalized
Lagrangian solution with respect to the choice
$$
  Z(t) := \PROJ{\H_{X(t)}}(F[X(t)])
  \quad\text{for all $t\GS 0$.}
$$
The fact that $V$ is the right-derivative of $X$ follows immediately from the
construction. To prove \eqref{E:FODI2} it is not restrictive to
assume $t_1=0$. We argue by induction on the collision times. In the first
interval $[t_0, t_1)$ inclusion \eqref{E:FODI2} is satisfied by taking the null
selection in the subdifferential $\partial I_\K(X(t))$. 

Assume now that \eqref{E:FODI2} is satisfied in $[t_{k-1},t_k)$ for some $k$.
Then
\begin{align}
  \dot{X}(t) 
    &= V(t_k+) + \int_{t_k}^t \PROJ{\H_{X(s)}}(F[X(s)]) \,\d s
\nonumber\\
    &= \Big( V(t_k+)-V(t_k-) \Big) + V(t_k-)
      + \int_{t_k}^t \PROJ{\H_{X(s)}}(F[X(s)]) \,\d s
\label{E:EINS}
\end{align}
for any $t\in[t_k,t_{k+1})$, by \eqref{E:ODE}. By induction assumption, we have
that
\begin{equation}
  V(t_k-) + \xi = \bar{V} + \int_0^{t_k} \PROJ{\H_{X(s)}}(F[X(s)]) \,\d s
\label{E:ZWEI}
\end{equation}
for some $\xi\in\partial I_\K(X(t_k))$. Combining \eqref{E:EINS} and
\eqref{E:ZWEI}, we obtain
$$
  \dot{X}(t) + \xi + \Big( V(t_k-)-V(t_k+) \Big)
    = \bar{V} + \int_0^t \PROJ{\H_{X(s)}}(F[X(s)]) \,\d s
$$
Because of \eqref{E:ODE}, we have that
$$
  V(t_k-) = \lim_{h\rightarrow 0+} h^{-1} \Big( X(t_k)-X(t_k-h) \Big).
$$
Using \eqref{E:OPK}, we then obtain
\begin{align*}
  & V(t_k-)-V(t_k+)
\\
  &\quad\vphantom{\Big(}
    = V(t_k-) - \PROJ{\H_{X(t_k)}}(V(t_k-))
\\
  &\quad
    = \lim_{h\rightarrow 0+} h^{-1} \Big( X(t_k)-X(t_k-h) 
      - \PROJ{\H_{X(t_k)}}\big( X(t_k)-X(t_k-h) \big) \Big)
\\
  &\quad
    = \lim_{h\rightarrow 0+} h^{-1} \Big( 
      \PROJ{\H_{X(t_k)}}\big( X(t_k-h) \big) - X(t_k-h) \Big).
\end{align*}
We now use Lemmas~\ref{L:TANGENT} and \ref{L:MINIMAL} and conclude that
$V(t_k-)-V(t_k+) \in \partial I_\K(X(t_k))$, noticing that $N_X\K =
\partial I_\K(X)$ for all $X\in\K$. Property \eqref{E:MODI1} implies the
monotonicity of the subdifferentials, which are closed convex cones. This yields
$$
  \xi + \Big( V(t_k-)-V(t_k+) \Big) \in \partial I_\K(X(t))
$$
for all $t\in[t_k,t_{k+1})$. Identities~\eqref{E:PROJREP2} and
\eqref{E:PROJREP4} can be proved as in Proposition~\ref{P:PROJECTION}. We conclude
that $X$ is a generalized Lagrangian solution.

It remains to show that if $F$ is sticking, then \eqref{E:FODI} holds. 
Because of \eqref{E:MODI1}, we have that $X(t) \in \H_{X(s)}$ for all $s\LS
t$. Then Definition~\ref{D:NONSPLITTING} yields
\begin{equation}
  \int_{0}^{t} \Big( F[X(s)] - \PROJ{\H_{X(s)}}(F[X(s)]) \Big) \,\d s
    \in \partial I_\K(X(t)).
\label{E:HEL}
\end{equation}
Adding \eqref{E:HEL} to either side of 
$$
  V(t) + \partial I_\K(X(t)) \ni \bar V
  + \int_{0}^{t} \PROJ{\H_{X(s)}}(F[X(s)]) \,\d s,
$$
we obtain \eqref{E:FODI}. Therefore $X$ is a sticky Lagrangian solution.
\end{proof}

We already know that any (even generalized) Lagrangian solution induces a
solution of the conservation law \eqref{E:CONS2}. Since for each time $t\GS 0$
the transport map $X(t,\cdot)$ is piecewise constant, it is easy to check that
the corresponding solution is in fact a discrete particle solution: the
density/momentum is of the form \eqref{E:PART}.

\begin{remark}
  \label{R:EXGENERALIZED}
  Notice that piecewise constant functions as in \eqref{E:DEFXV} are
  dense in $\L^\p(\Omega)$, so we can approximate any given initial
  data and then combine the existence result \ref{T:EXDISCRETE} with
  the stability Theorem~\ref{T:STABCOMP} to
  get the proof of Theorem \ref{T:EXLAGR}.
\end{remark}
\begin{remark}\label{R:MONO} 
The proof of Theorem~\ref{T:MAINSTICK} follows by a similar
approximation argument. 
By Theorem
\ref{P:SEMIGROUP}
it is sufficient to show that any Lagrangian solution $X$ with $\bar X\in
\K$ and
$\bar V\in \H_{\bar X}$ satisfies
$$
  \Omega_{\bar{X}} \subset \Omega_{X(t)}
  \quad\text{for all $t\GS 0$.}
$$
That is, if $\bar{X}$ is constant on some interval $(\alpha,\beta) \subset
\Omega$, then $X(t)$ remains constant on $(\alpha,\beta)$ for all times $t\GS
0$. We approximate $(\bar{X},\bar{V})$ by a sequence $(\bar{X}_n,\bar{V}_n)$
of the form \eqref{E:DEFXV} such that $\bar{X}_n$ is constant on
$(\alpha,\beta)$. Since this property is preserved by the discrete Lagrangian
solution constructed in Theorem~\ref{T:EXDISCRETE}, 
the stability estimates of Theorem \ref{T:EUS} show that the limit
function $X(t)$ is still constant on $(\alpha,\beta)$. 
\end{remark}


\section{Global Existence  in Eulerian coordinates }\label{S:GE}

Theorems~\ref{T:EUS}, \ref{T:MAINSTICK}, \ref{T:SCHAUDER}, and
\ref{T:EXLAGR} of the previous sections 

immediately translate into 
global existence results for the  Euler system of  
conservation laws \eqref{E:CONS2}. Before
stating  some of 
the related results, let us explore in more detail the relation between the
force functionals $f[\RHO]$ in \eqref{E:CONS2} and their reformulation in the
Lagrangian framework. 
\subsection{The Eulerian description of the force field}
\label{S:EDFF}
Let
us first introduce the space 
$$
  \T_\p(\R) := \Big\{ (\RHO,v) \colon \RHO\in\SP_\p(\R), v\in\L^\p(\R,\RHO) \Big\}.
$$
For all $(\RHO_i,v_i)\in\T_\p(\R)$ with $i=1\ldots 2$, we then define 
\cite[\S 2]{NatileSavare}
$$
  D_\p\Big( (\RHO_1,v_1),(\RHO_2,v_2) \Big)
    := \max\Big\{ W_\p(\RHO_1,\RHO_2), 
      U_\p\Big( (\RHO_1,v_1),(\RHO_2,v_2) \Big) \Big\},
$$
where $W_\p$ is the Wasserstein distance and $U_\p$ denotes the semi-distance
\begin{align}
  U_\p^\p\Big( (\RHO_1,v_1), (\RHO_2,v_2) \Big)
    &:= \int_{\R\times\R} |v_1(x)-v_2(y)|^\p \RRHO(dx,dy)
\nonumber\\
    &\hphantom{:}=
      \int_\Omega |v_1(X_{\RHO_1}(m))-v_2(X_{\RHO_2}(m))|^\p \,\d m.
\label{E:SEMID}
\end{align}
Here $\RRHO \in \Gamma_\OPT(\RHO_1,\RHO_2)$ is the unique optimal transport
map between the measures $\RHO_1$ and $\RHO_2$. It can be expressed in terms of
the transport maps defined in \eqref{E:90}; see \eqref{E:146}. The sequence
$\{(\RHO_n,v_n)\}$ converges to $(\RHO,v)$ in the metric space $(\T_\p(\R),D_\p)$
if and only if $W_\p(\RHO_n,\RHO) \longrightarrow 0$, if $\RHO_n v_n \WEAK \RHO
v$ weak* in $\M(\R)$, and if 
$$
  \int_\R |v_n|^\p \,\RHO_n \longrightarrow \int_\R |v|^\p \,\RHO.
$$
We refer the reader to \cite[Prop.~2.1]{NatileSavare} and to \cite{AmbrosioGigliSavare} for further details (see in
particular Definition~5.4.3).

We consider a continuous map
(with respect to the Wasserstein topology in $\SP_2(\R)$ and 
the weak$^*$ topology on $\M(\R)$ induced by $\C_b(\R)$)
\begin{equation}
  f \colon \SP_\p(\R) \longrightarrow \M(\R),\quad 
  f[\RHO]=f_\RHO\,\RHO,\quad
  f_\RHO\in \L^2(\R,\RHO),
  \label{eq:72}
\end{equation}
with the property that $f[\RHO]$ is absolutely continuous with respect to $\RHO
\in\SP_\p(\R)$: $f_\RHO$ is the Radon-Nikodym-derivative of
$f[\RHO]$ with respect to $\RHO$ and assume that $f_\RHO \in
\L^\p(\R,\RHO)$.

\begin{definition}[Boundedness]\label{D:FBD}
We say that a map $f \colon \SP_\p(\R) \longrightarrow \M(\R)$ as in \eqref{eq:72}
is bounded if there exists a
constant $C \GS 0$ such that
$$
  \|f_\RHO\|_{\L^\p(\R,\RHO)}^2 \LS C\Big( 1 + \int_\R |x|^2\,\d\RHO
\Big)
  \quad\text{for all $\RHO\in\SP_\p(\R)$.}
$$
We say that $f$ is pointwise linearly bounded if there exists a $C_\sfp \GS
0$ such that
$$
  |f_\RHO(x)| \LS C_\sfp \Big(1 + |x|+\int_\R |x|\,\d\RHO\Big)
  \quad\text{for a.e.\ $x\in\R$ and all $\RHO\in\SP_\p(\R)$.}
$$
\end{definition}

\begin{definition}[Uniform continuity I]\label{D:FCO}
We say that a map $f \colon \SP_\p(\R) \longrightarrow \M(\R)$ as in
\eqref{eq:72} 
is 
uniformly continuous if  there exists a modulus of
continuity $\omega$ such that 
\begin{equation}
  U_\p\Big( (\RHO_1,f_{\RHO_1}), (\RHO_2,f_{\RHO_2}) \Big)
    \LS \omega\Big( W_\p(\RHO_1,\RHO_2) \Big)
  \quad\text{for all $\RHO_1,\RHO_2\in\SP_\p(\R)$.}
\label{E:UNFO}
\end{equation}
In the case $\omega(r) = Lr$ for some constant $L\GS 0$ and all $r\GS 0$,
we say that $f$ is Lipschitz continuous.
\end{definition}
\EEE
As discussed in Section~\ref{SS:OT}, there is a one-to-one correspondence
between measures $\RHO\in\SP_\p(\R)$ and optimal transport maps 
$X\in\L^\p(\Omega)$, given by
\begin{equation}
  X\in\K
  \quad\text{and}\quad
  X_\#\leb = \RHO.
\label{E:RELATION}
\end{equation}
We now want to construct a functional $F\colon \K\longrightarrow \L^\p(\Omega)$
such that
\begin{equation}
  \int_\R \varphi(x) \, f[\RHO](\d x) 
    = \int_\Omega \varphi(X(m)) F[X](m) \,\d m
  \quad\text{for all $\varphi\in\C_b(\R)$,}
\label{E:CONN}
\end{equation}
whenever $(X,\RHO)$ are related by \eqref{E:RELATION}. One possible choice is 
to set 
\begin{equation}
  \label{eq:73}
  F[X] := f_\RHO\circ X\quad\text{for all $(X,\RHO)$ satisfying \eqref{E:RELATION},}
\end{equation}
\EEE
which easily gives $F[X]\in\L^\p(\Omega)$. Then the boundedness and continuity
assumptions on the functional $f$ in Definitions~\ref{D:FBD} and \ref{D:FCO}
translate immediately into the corresponding properties for $F$ in
Definitions~\ref{D:BOUNDEDNESS} and \ref{D:CONTINUITY}. It can be useful,
however, to also consider different choices for $F$.

\begin{definition}[ Uniform  continuity II]
We say that a map $f \colon \SP_\p(\R) \longrightarrow \M(\R)$ as in
\eqref{eq:72} is \emph{densely} uniformly
continuous if \eqref{E:UNFO} holds for
measures that are absolutely continuous with respect to $\LEB^1$ with bounded
densities. We define dense Lipschitz continuity similarly.
\end{definition}

\begin{lemma}\label{L:EXTEN}
If 
$f\colon\SP_\p(\R)\longrightarrow\M(\R)$ is densely uniformly
continuous,
then there exists a unique
uniformly continuous map $F\colon\K\longrightarrow\L^\p(\Omega)$ such that
\eqref{E:CONN} holds for all $(X,\RHO)$ satisfying \eqref{E:RELATION}. 
\end{lemma}

Note that \eqref{E:CONN} imlies that $f_\RHO\circ X = \PROJ{\H_X}(F[X])$ for
all $(X,\RHO)$ with \eqref{E:RELATION}.

\begin{proof}
We denote by $\K_\REG$ the dense subset of $\K$ whose elements are
$\C^1(\bar{\Omega})$-maps with strictly (hence uniformly) positive
derivatives. For every $X\in\K_\REG$ the push-forward $\RHO :=
X_\#\leb$ is absolutely continuous with respect to the Lebesgue
measure $\LEB^1$ and has a bounded density. We can then define
\begin{equation}
  F[X] := f_\RHO\circ X
  \quad\text{for all $X\in\K_\REG$.}
\label{E:2}
\end{equation}
Applying definition \eqref{E:SEMID} and \eqref{E:UNFO} we obtain
$$
  \|F[X_1]-F[X_2]\|_{\L^\p(\Omega)} 
    \LS \omega\Big( \|X_1-X_2\|_{\L^\p(\Omega)} \Big)
  \quad\text{for all $X_1,X_2\in\K_\REG$.}
$$
Then $F$ can be extended to all of $\K$ by density. One can check that
this functional satisfies \eqref{E:2}, therefore it is uniquely determined
by $f$.  
\end{proof}

\begin{definition}[Sticking]\label{D:NSP}
Let $f \colon \SP_\p(\R) \longrightarrow \M(\R)$ be 
densely uniformly continuous and let $F$ be
the functional from Lemma~\ref{L:EXTEN}. We say that $f$ is sticking
if $F$ is sticking.
\end{definition}

\subsection{Existence results and examples}
\label{S:ER} 
We state here a simple example of possible applications of the previous
Lagrangian results; for the sake of simplicity, we omit
to detail all the information which could be derived
by the finer structure properties and by the a priori estimates 
we obtained for the Lagrangian formulation.
It is worth noticing that all the solutions can be obtained as a
suitable limit of discrete particle evolutions.

The first statement follows by Theorem \ref{T:EXLAGR}, the second one
by Theorem \ref{T:EUS}, Theorems \ref{P:SEMIGROUP} and
\ref{T:MAINSTICK}
yields the last assertion.
\EEE
\begin{theorem}[Global Existence]\label{T:GLOBEX}
Let us fix $
  \bar{\RHO} \in \SP_\p(\R)
  \quad\text{and}\quad
  \bar{v} \in \L^\p(\R,\bar{\RHO}).
$.
\begin{enumerate}
\item Suppose that the force functional $f\colon \SP_\p(\R) \longrightarrow
\M(\R)$ is pointwise linearly bounded and densely uniformly continuous. Then
there exists a solution $(\RHO,v)$ of the conservation law
\eqref{E:CONS2} with initial data $(\bar{\RHO},\bar{v})$. 
\item If $f\colon \SP_\p(\R) \longrightarrow \M(\R)$ is densely
Lipschitz continuous, then there exists a stable selection of a 
solution $(\RHO,v)$ of \eqref{E:CONS2} with
respect to the initial data $(\bar{\RHO},\bar{v})$ in $(\T_\p(\R),D_\p)$.
\item If $f\colon \SP_\p(\R) \longrightarrow \M(\R)$ is densely
Lipschitz continuous, and sticking, then there exists a stable
sticky solution $(\RHO,v)$ of \eqref{E:CONS2} with initial data
$(\bar{\RHO},\bar{v})$.
The map $S_t:(\bar \RHO,\bar v)\mapsto \RHO(t,\cdot),v(t,\cdot))$ is a semigroup in $(\T_\p(\R),D_\p)$.
\end{enumerate}
\end{theorem}

We finish the paper by giving a number of examples of force functionals.

\begin{example}
Let $v\colon\R\longrightarrow\R$ be a continuous function satisfying
\begin{equation}
  |v(x)| \LS C_v(1+|x|)
  \quad\text{for all $x\in\R$,}
\label{E:24}
\end{equation}
with $C_v\GS 0$ some constant. Then the operator defined by
$$
  f[\RHO] := \RHO v
  \quad\text{for all $\RHO\in\SP_\p(\R)$}
$$
is pointwise linearly  bounded; it is Lipschitz continuous if $v$ is a
Lipschitz function. Note that $f[\RHO]$ is the Wasserstein differential of the
potential energy (see \cite{AmbrosioGigliSavare})
$$
  \mathscr{V}[\RHO]:=\int_\R V(x) \,\RHO(\d x)
  \quad\text{where $v=V'$.}
$$
\end{example}

\begin{example}
Let $w\colon\R\longrightarrow\R$ be a continuous function satisfying
\eqref{E:24}. Then 
$$
  f[\RHO] := \RHO (w\star\RHO)
    = \RHO \bigg( \int_\R w(\cdot-y) \,\RHO(\d y) \bigg)
  \quad\text{for all $\RHO\in\SP_\p(\R)$}
$$
is pointwise linearly bounded, since 
\begin{align*}
  \big|(w\star \RHO)(x)\big|\LS C_w\int_\R (1+|x-y|)\,\RHO(\d y)\LS 
  C_w\Big(1+|x|+\int_\R |y|\,\RHO(\d y)\Big)
\end{align*} 
It is Lipschitz continuous if $w$ is a
Lipschitz function. In fact, writing $f_\RHO := w\star\RHO$ for all
$\RHO\in\SP_\p(\R)$, we have that
\begin{align*}
  |f_{\RHO_1}(x)-f_{\RHO_2}(y)|
    &= \bigg| \int_\R w(x-x') \,\RHO_1(\d x') - \int_\R w(y-y') \,\RHO_2(\d y')
\bigg|
\\
    &= \bigg| \int_{\R\times\R} \Big( w(x-x') - w(y-y') \Big) 
      \,\RRHO(\d x',\d y') \bigg|
\\
    &\LS L \bigg( |x-y| + \int_{\R\times\R} |x'-y'| \,\RRHO(\d x',\d y') \bigg),
\end{align*}
where $L\GS 0$ is the Lipschitz constant of $w$ and $\RRHO\in
\Gamma_\OPT(\RHO_1,\RHO_2)$. This implies
$$
  U_\p\Big( (\RHO_1,f_{\RHO_1}),(\RHO_2,f_{\RHO_2}) \Big)
    \LS 4 L W_\p(\RHO_1,\RHO_2)
  \quad\text{for all $\RHO_1,\RHO_2\in\SP_\p(\R)$.}
$$
Note that $f[\RHO]$ is the Wasserstein differential of the interaction energy
(see \cite{AmbrosioGigliSavare})
$$
  \mathscr{W}[\RHO] = \int_{\R\times\R} W(x-y) \,\RHO(\d x) \,\RHO(\d y)
  \quad\text{where $w=W'$.}
$$
\end{example}

\begin{example}
  \label{ex:EP1}
Let us consider the previous example with the Borel function
$$
  w(x) := \begin{cases}
    1 & \text{if $x>0$}
\\
    0 & \text{if $x=0$}
\\
   -1 & \text{if $x<0$}
  \end{cases},
$$
which corresponds to $W(x) := |x|$. To show that $f[\RHO]$ is continuous,
note that
$$
  f_\RHO(x) 
    = m_\RHO(x) + M_\RHO(x) - 1
  \quad\text{for all $x\in\R$,}
$$
where $m_\RHO(x) := \RHO\big( (-\infty,x) \big)$ and
$M_\RHO(x):=\RHO\big((-\infty,x]\big)$
as in
\eqref{E:53} above. Up to rescaling and adding constants, the function $f_\RHO$
is the precise representative of the cumulative distribution function of the
measure $\RHO$. For convenience, we define 
$$
  \tilde{f}_\RHO(x) := f_\RHO(x)+1
  \quad\text{for all $x\in\R$,}\quad
   \tilde f[\RHO]:=\tilde f_{\RHO}\,\RHO. 
$$
We now introduce the sets
$$
  J_\RHO := \Big\{ x\in\R \colon \RHO\big( \{x\} \big) > 0 \Big\}
  \quad\text{and}\quad
  \J_\RHO := \bigcup_{x\in J_\RHO} \big( m_\RHO(x),M_\RHO(x) \big).
$$
Note that $J_\RHO$ is at most countable. If $X_\RHO$ is defined by
\eqref{E:90}, then
\begin{gather*}
  X_\RHO(m) = x
  \quad\text{for all $m\in [m_\RHO(x),M_\RHO(x)]$ and $x\in\J_\RHO$,}
\\
  \tilde{f}_\RHO(X_\RHO(m)) = 2m 
  \quad\text{for all $m\in\Omega\setminus\J_\RHO$.}
\end{gather*}
For any $\varphi\in\C_b(\R)$ we have
\begin{align*}
  & \int_\Omega \varphi(x)\,\tilde f[\RHO](\d x)\EEE=\int_\Omega \varphi(X_\RHO(m)) \tilde{f}_\RHO(X_\RHO(m)) \,\d m
\\
  &\quad
    = \int_{\Omega\setminus\J_\RHO} \varphi(X_\RHO(m)) 
      \tilde{f}_\RHO(X_\RHO(m)) \,\d m
    + \sum_{x\in J_\RHO} \int_{m_\RHO(x)}^{M_\RHO(x)} \varphi(X_\RHO(m))
      \tilde{f}_\RHO(X_\RHO(m)) \,\d m
\\
  &\quad
    = 2 \int_{\Omega\setminus\J_\RHO} \varphi(X_\RHO(m)) \,m \,\d m
    + \sum_{x\in J_\RHO} \Big( M^2_\RHO(x)-m^2_\RHO(x) \Big) \varphi(x) 
\\
  &\quad
    = 2\int_{\Omega\setminus\J_\RHO} \varphi(X_\RHO(m)) \,m \,\d m
    + 2\sum_{x\in J_\RHO} \int_{m_\RHO(x)}^{M_\RHO(x)} \varphi(x) \,m \,\d m
\\
  &\quad
    = 2\int_\Omega \varphi(X_\RHO(m)) \,m \,\d m.
\end{align*}
It follows that
$$
  \int_\R \varphi(x) f_\RHO(x) \,\RHO(\d x)
    = \int_\Omega \varphi(X_\RHO(m)) \, (2m-1) \,\d m
$$
Then the map $f$ is pointwise linearly bounded because $|f_\RHO(x)|
\LS 1$ for all $x\in\R$. It is continuous since $\RHO_n \longrightarrow \RHO$ in
$\SP_\p(\R)$ implies that $X_{\RHO_n} \longrightarrow X_\RHO$ in $\L^\p(\Omega)$.
It is densely Lipschitz continuous since the associated functional $F$
is given by
\begin{equation}
  F[X](m) := 2m-1
  \quad\text{for all $m\in\Omega$,}
\label{EQ:61}
\end{equation}
which does not even depend on $X\in\K$ anymore.
\end{example}

\begin{example}
For $\sigma\in\L^\infty(\R)$ let $q_\RHO$ be the solution of (recall \eqref{eq:2})
\begin{equation}
   - \partial^2_{xx} q_\RHO=\lambda\big(\RHO-\sigma\big).
\label{E:EP1}
\end{equation}
Then $q_\RHO$ is locally Lipschitz continuous and its (opposite) derivative $a_\RHO :=
 - \partial_x q_\RHO$ is locally of bounded variation. Choosing its precise
representative we then define
\begin{equation}
  f[\RHO] := \RHO a_\RHO
  \quad\text{for all $\RHO\in\SP_\p(\R)$.}
\label{E:EP2}
\end{equation}
Setting $Q_\sigma(x) := \int_0^x \sigma(y) \,\d y$ it is not difficult to check
that
$$
  a_\RHO(x) =  -\lambda \Big( \frac{1}{2} \big( m_\RHO(x)+M_\RHO(x)-1 \big)
    -Q_\sigma(x) \Big)
  \quad\text{for all $x\in\R$,}
$$
so that the associated operator $F$ is given by
$$
  F[X](m) =  -\lambda  \Big( \frac{1}{2}(2m-1)-Q_\sigma(X(m)) \Big)
  \quad\text{for all $m\in\Omega$.}
$$
This corresponds to the Euler-Poisson system discussed in the
Introduction.
For simplicity, let us consider consider the case when $\sigma$
vanishes. 
\medskip
\paragraph{\em Sticky solutions for the attractive Euler-Poisson system}
In the  attractive case (when $\lambda>0$) 
the functional $F[X]$ is
sticking: Let $\Omega_X$ be defined by \eqref{E:OMEX} and let $(\alpha,
\beta) \subset \Omega_X$ be a maximal interval. Then $\PROJ{\H_X}(F[X])$ is
constant in $(\alpha,\beta)$ and equal to its average over the interval. We
define
\begin{equation}
  \Xi(m) := \int_\alpha^m \Big( F[X](m) - \PROJ{\H_X}(F[X])(m) \Big) \,\d m
  \quad\text{for all $m\in (\alpha,\beta)$.}
\label{E:DEFXI}
\end{equation}
Then $\Xi(\alpha) = \Xi(\beta) = 0$. Since  $\lambda>0$, $\Xi$ is concave
and, we obtain that
$\Xi(m) \GS 0$ in $(\alpha,\beta)$. By Lemma~\ref{L:NORMAL}, we conclude that
the functional $F$ is sticking. 
Sticky Lagrangian solutions to the
Euler-Poisson system \eqref{E:CONS2} (thus obtained as limit
of sticky particly dynamics)
are therefore unique and in fact form a semigroup in the metric space
$(\T_\p(\R),D_\p)$ by Theorems \ref{T:EUS}, \ref{P:SEMIGROUP}, and \ref{T:MAINSTICK}.

We can then apply the representation formula \eqref{E:PROJREP01} and
\eqref{eq:60} to
obtain the following result:
\begin{theorem}[Representation formula for 
  attractive Euler-Poisson system]
\label{T:RFAE}
  The unique sticky Lagrangian solution of the Euler-Poisson system
  ($\lambda\GS 0$)
  corresponding to initial data $(\bar \RHO, \bar v)$
  with $\bar \RHO=\bar X_\#\leb$,
  $ \bar X \in \K$, and $\bar V=\bar v\circ \bar X$,
  can be obtained by the formula
  \begin{equation}
    \label{eq:74}
    \RHO(t,\cdot)=X(t,\cdot)_\# \leb,\quad
    X(t,m)=\frac \partial{\partial m}\mathcal X^{**}(t,m)
  \end{equation}
  where $\mathcal X^{**}(t,m)$ is the convex envelope (w.r.t.\ $m$,
  see \eqref{eq:75})
  of 
  \begin{equation}
    \label{eq:76}
    \mathcal X(t,m):=\int_0^m \Big(\bar X(\ell)+t \bar V(\ell)-\lambda
    \frac {t^2} 4\big(2\ell-1\big)\Big)\,\d \ell
  \end{equation}
\end{theorem}
Notice that when $\lambda=0$ we find the sticky particle solution of 
\cite{NatileSavare}.
\medskip
\paragraph{\em Lagrangian solutions for the repulsive Euler-Poisson system}
In the repulsive case $\lambda<0$
the
function $\Xi$ defined in \eqref{E:DEFXI} is convex and vanishes at the
endpoints of $(\alpha,\beta)$, thus $\Xi(m) \LS 0$ for all
$m\in(\alpha,\beta)$,
and the map $F$ does not satisfies the sticking condition. 
In this case \eqref{eq:74}-\eqref{eq:76} may be different from the solution
given by Theorem~\ref{T:EUS}. 

Here is a simple example for $\lambda=-2$:
consider the initial condition
\begin{equation}
  \label{eq:78}
  \bar X(m):=m-1/2,\quad
  \bar V(m):=-\sign(m-1/2),
\end{equation}
for which \eqref{eq:76} yields
\begin{equation}
  \label{eq:79}
  \mathcal X(t,m)=\frac 12(1+t^2)(m-1/2)^2-t|m-1/2|-c(t),\quad 
  c(t):=\frac 18 (1+t^2-4t).
\end{equation}
It is easy to check that 
\begin{equation}
  \label{eq:80}
  \mathcal X^{**}(t,m)=
  \begin{cases}
    \mathcal X(t,m)&\text{if }|m-1/2|\ge \delta(t),\\
    -\frac{t^2}{2(1+t^2)}-c(t)&\text{if }|m-1/2|\le \delta(t),
  \end{cases}
  \qquad
  \text{where}\quad
  \delta(t):=\frac t{1+t^2},
\end{equation}
so that $X(t,\cdot)$ is the piecewise linear continuous map
\begin{displaymath}
  X(t,m)=
  \begin{cases}
    \bar X(t,m) &\text{if }|m-1/2|\ge \delta(t),\\
    0&\text{if }|m-1/2|\le \delta(t),
  \end{cases}
  \qquad\text{where }
  \bar X(t,m):=\frac{\partial}{\partial m}\mathcal X(t,m).
\end{displaymath}
If we eventually introduce
\begin{displaymath}
  Y(t,m):=\frac\partial{\partial t}\bar X(t,m)=
  \frac{\partial^2}{\partial t\,\partial m} \mathcal X(t,m)=2t(m-1/2)-\sign(m-1/2),
\end{displaymath}
recalling \eqref{eq:26} $X$ is a Lagrangian solution if and only if
\begin{displaymath}
  Y(t,\cdot)-\dot X(t,\cdot)\in \partial
  I_\K(X(t,\cdot))\quad\text{a.e.\ in }(0,\infty).
\end{displaymath}
By Lemma \ref{L:NORMAL} we obtain the equivalent condition
\begin{equation}\label{eq:83}
  \frac{\partial}{\partial t}\big(\mathcal X(t,m)-\mathcal
  X^{**}(t,m)\big)\ge 0\quad\text{in }-\delta(t)< m< \delta(t),
\end{equation}
which is not compatible with \eqref{eq:79} and \eqref{eq:80}: to see
this, 
fix e.g. $0<\delta<1/2$, $m_*:=1/2+ \delta$, and
$t_\pm:=\frac{1\pm\sqrt{1-4\delta^2}}{2\delta}$, so that
$\delta(t_\pm)=\delta<\delta(t)$ for every $t\in (t_-,t_+)$.
We thus have
\begin{displaymath}
  \mathcal X(t_\pm,m_*)-
  \mathcal X^{**}(t_\pm,m_*)=0,\quad
  \mathcal X(t,m_*)-
  \mathcal X^{**}(t,m_*)>0\quad\text{for }t_-<t<t_+,
\end{displaymath}
which contradicts \eqref{eq:83}.  
\end{example}
%
\section{Convergence of the Time Discrete Scheme of Section~\ref{SS:TDS}}
\label{S:CONV}

In this section, we establish the convergence of the time discrete scheme
of Section~\ref{SS:TDS}. Since the proof does not substantially
differ from the one provided in \cite{Brenier2} for order-preserving vibrating
strings, we only sketch the main steps. The key point is the non-expansive
property of the time-discrete scheme. Indeed, we first observe that the
rearrangement operator, even in the periodic case, is non-expansive in
$\L^2(\Omega)$. More precisely, we have that
$$
  \int_0^1 |Y^*(m)-Z^*(m)|^2 \,\d m \LS \int_0^1 |Y(m)-Z(m)|^2 \,\d m
$$
for all pairs $(Y,Z)$ of maps such that $Y-\ID$ and $Z-\ID$ are 1-periodic and
square integrable. Next, we see that the harmonic oscillations \eqref{harmonic}
are isometric in phase space for $(X(t,m)-m,V(t,m))$, for each fixed $m$. Let
$(X_{ \tau, n},V_{ \tau, n})$, $(Y_{\tau
  ,n},W_{ \tau,n})$ be generated by the
time-discrete scheme. Then
\begin{align*}
  & \|X_{\tau,n+1}-Y_{\tau,n+1}\|_{\L^2(\Omega)}^2
      +\|V_{\tau,n+1}-W_{\tau,n+1}\|_{\L^2(\Omega)}^2
\\
  & \quad
    \LS \|\hat X_{\tau,n+1}-\hat Y_{\tau,n+1}\|_{\L^2(\Omega)}^2
      +\|V_{\tau,n+1}-W_{\tau,n+1}\|_{\L^2(\Omega)}^2
\\
  & \quad
    = \|X_{\tau,n}-Y_{\tau,n}\|_{\L^2(\Omega)}^2
      +\|V_{\tau,n}-W_{\tau,n}||_{\L^2(\Omega)}^2.
\end{align*}
Since  $(X=\ID,V=0)$ is a trivial solution of the scheme, we immediately get
$$
  \|X_{\tau,n+1}-\ID\|_{\L^2(\Omega)}^2+\|V_{\tau,n+1}\|_{\L^2(\Omega)}^2
    \LS \|\bar{X}-\ID\|_{\L^2(\Omega)}^2+\|\bar{V}\|_{\L^2(\Omega)}^2.
$$
Because the scheme is translation invariant in $m$ and (discretely) in $n$, we
easily deduce the strong compactness in $\C^0_t(\L^2_m)$ of the discrete
solutions, linearly interpolated in time, for each 1-periodic initial condition
$(\bar{X}-\ID,\bar{V})$, first in $\H^1$ and then (by a density argument, using
the non-expansive property of the scheme) in $\L^2$. Let us now examine the
consistency of the scheme. To do that, let us compare a solution of the discrete
scheme to any smooth test function $m\rightarrow (Y(m),W( m))$ where $Y$ is
nondecreasing and $(Y(m)-m,W(m))$ is 1-periodic. Since the rearrangement
operator is non-expansive and $Y=Y^*$ is nondecreasing, we first get
\begin{align*}
  & \|X_{\tau,n+1}-Y\|_{\L^2(\Omega)}^2 
      + \|V_{\tau,n+1}-W\|_{\L^2(\Omega)}^2
\\
  & \quad
    \LS \|\hat X_{\tau,n+1}-Y\|_{\L^2(\Omega)}^2
      + \|V_{\tau,n+1}-W\|_{\L^2(\Omega)}^2
\\
  & \quad
    = \int_0^1 \Big\{ \big|(X_{\tau,n}(m)-m)\cos(\tau)
      +V_{\tau,n}(m)\sin(\tau)- (Y(m)-m) \big|^2
\\
  & \quad\hphantom{= \int_0^1 \Big\{}
  +\big|(X_{\tau,n}(m)-m)\sin(\tau)-V_{\tau,n}(m)\cos(\tau)+W(m)\big|^2 
      \Big\} \,\d m.
\end{align*}
One can then check that
\begin{align*}
  & \|X_{\tau,n+1}-Y\|_{\L^2(\Omega)}^2
      + \|V_{\tau,n+1}-W\|_{\L^2(\Omega)}^2
\\
  & \quad
    \LS \|X_{\tau,n}-Y\|_{\L^2(\Omega)}^2
      + \|V_{\tau,n}-W||_{\L^2(\Omega)}^2
\\
  & \quad 
    + 2\tau \int_0^1 \Big\{ \big(X_{\tau,n}(m)-Y(m)\big) V_{\tau,n}(m) 
      -\big(X_{\tau,n}(m)-m\big)\big(V_{\tau,n}(m)-W(m)\big) \Big\}
      \,\d m
      +\kappa \tau^2,
\end{align*}
with constant $\kappa$ depending only on the test functions $(Y,W)$ and the
initial data $(\bar{X},\bar{V})$. Clearly, this estimate is consistent with 
the differential inequality
\begin{align}
  & \frac{\d}{\dt} \Big\{ \|X(t,\cdot)-Y\|_{\L^2(\Omega)}^2
    +\|V(t,\cdot)-W\|_{\L^2(\Omega)}^2 \Big\}
\label{E:METRIC}
\\
  & \quad
    \LS 2\int_0^1  \Big\{ \big(X(t,m)-Y(m)\big)V(t,m)-\big(X(t,m)-m\big)\big(V(t,m)-W(m)\big)
    \Big\} \,\d m,
\nonumber
\end{align}
valid for all pair of 1-periodic functions of form $m\mapsto (Y(m)-m,W(m))$ with
$Y$ nondecreasing, which is nothing but the ``metric formulation'' of
\eqref{eq:59}. Indeed, for a.e.\ $t\GS 0$ fixed, by choosing $Y=X(t,\cdot)$ and
$W=V(t,\cdot)\pm Z$ for arbitrary 1-periodic $Z\in\L^2(\Omega)$, we find that
$$
  \dot{V}(t,m) + X(t,m)-m = 0;
$$
cf.\ \eqref{harmonic}. On the other hand, by choosing $W=V(t,\cdot)$ and $Y=0$ 
resp.\ $Y=2X(t,\cdot)$, we obtain
\begin{align*}
  \int_0^1 X(t,m) ( \dot{X}(t,m)-V(t,m) ) \,\d m &= 0
\\
  \int_0^1 Y(m) ( \dot{X}(t,m)-V(t,m) ) \,\d m &\GS 0
  \quad\text{for all $Y$ nondecreasing with $Y(m)-m$ 1-periodic.}
\end{align*}
This implies precisely that $-\dot{X}(t,\cdot)+V(t,\cdot) \in \partial I_\K(X(t,\cdot))$, which gives \eqref{eq:59}. This concludes the proof of convergence for 
the time-discrete scheme.


\section*{Acknowledgments}

YB's work is partially supported by the ANR grant OTARIE ANR-07-BLAN-0235.
WG gratefully acknowledges the support provided by NSF grants DMS-06-00791 and DMS-0901070. 
GS was partially supported by MIUR-PRIN'08 grant for the project ``Optimal mass transportation,
geometric and functional inequalities, and applications''. 
The research of MW was supported by NSF grant DMS-0701046.
This project started during a visit of GS to the School of Mathematics
of the Georgia Institute of Technology, whose support he gratefully
acknowledges. We note that this work has essentially been completed when we learned of Tadmor and Wei's result \cite{TadmorW} which is to be compared with Theorem \ref{T:RFAE} of the current manuscript.  





\end{document}